\documentclass[11pt,reqno]{amsart}
\usepackage{amsmath,amssymb,amsthm,amsfonts,latexsym, amsthm, amscd, mathrsfs, stmaryrd}
\usepackage{mathtools}
\usepackage[mathlines]{lineno}
\usepackage[colorlinks=true,linkcolor=blue,citecolor=blue]{hyperref}
\usepackage{cite}
\usepackage{graphicx,color,colortbl}
\usepackage{upgreek}
\usepackage{hhline}
\numberwithin{equation}{section}
\usepackage{dynkin-diagrams}
\usepackage{epic}
\allowdisplaybreaks
\numberwithin{equation}{section}
\usepackage{float}
\usepackage{tikz}
\usetikzlibrary{
    cd,
	shapes,
	arrows,
	positioning,
	decorations.markings,
	decorations.pathmorphing,
	circuits.logic.US,
	circuits.logic.IEC,
	fit,
	calc,
	plotmarks,
	matrix
}
\usepackage{cancel}

\usepackage{todonotes}
\usepackage{enumerate}

\usepackage[english]{babel}
\usepackage[T1]{fontenc}


\newtheorem{proposition}{Proposition}[section]
\newtheorem{lemma}[proposition]{Lemma}
\newtheorem{corollary}[proposition]{Corollary}
\newtheorem{theorem}[proposition]{Theorem}

\newtheorem*{conjintro}{Conjecture}

\newtheorem{innercustomthm}{{\bf Theorem}}
\newenvironment{customthm}[1]
  {\renewcommand\theinnercustomthm{#1}\innercustomthm}
  {\endinnercustomthm}
\theoremstyle{definition}
\newtheorem{definition}[proposition]{Definition}
\newtheorem{example}[proposition]{Example}
\newtheorem{conjecture}[proposition]{Conjecture}

\theoremstyle{remark}
\newtheorem{remark}[proposition]{Remark}

\newcommand{\Rmnum}[1]{\expandafter\@slowromancap\romannumeral #1@}
\def \g{\mathfrak{g}}

\def \N{\mathbb{N}}
\def \Q{\mathbb{Q}}
\def \C{\mathbb{C}}
\def \Z{\mathbb{Z}}
\def \I{\mathbb{I}}

\def \bP{\mathbf{P}}

\def \f{\mathbf{f}}
\def \A{\mathcal{A}}
\def \tY{\widetilde{Y}}

\def \wI{\I_{\circ}}

\def \bI{\I_{\bullet}}

\def \bw{w_\bullet}

\def \ba{\mathbf{a}}

\def \tX{\widetilde{X}}

\def \tY{\widetilde{Y}}

\def \tT{\widetilde{T}}

\def \tW{\widetilde{W}}

\def \ww{w^\bullet}

\def \ss{S}

\def \Id{\mathrm{Id}}

\newcommand{\U}{\mathbf{U}}

\def \ov{\overline}
\def \un{\underline}

\def \O{\mathcal{O}}
\def \X{\mathcal{X}}

\newcommand{\nc}{\newcommand}
\nc{\greentext}[1]{\textcolor{green}{#1}}
\nc{\redtext}[1]{\textcolor{red}{#1}}
\nc{\bluetext}[1]{\textcolor{blue}{#1}} 
\nc{\green}[1]{\greentext{ #1}}
\nc{\red}[1]{\redtext{ #1}}
\nc{\blue}[1]{\bluetext{ #1}}

\def \Hom{\mathrm{Hom}}

\def \vs{\varsigma}
\def \U{\mathbf{U}}
\def \Ui{\mathbf{U}^\imath}

\def \QQ{\mathbb{Q}(q)}

\def \bB{\mathbf{B}}
\def \bE{\mathbf{E}}
\def \bF{\mathbf{F}}

\def \het{\mathrm{ht}\,}

\def \op{\mathrm{op}}

\def \tc{\widetilde{c}}
\def \tC{\widetilde{C}}
\def \tN{\widetilde{N}}
\def \tR{\widetilde{R}}
\def \tB{\widetilde{B}}
\def \tI{\widetilde{\I}}

\def \tg{\widetilde{\g}}
\def \tG{\widehat{\Gamma}}
\def \tE{\widetilde{E}}
\def \tbE{\widetilde{\bE}}

\def \cO{\mathcal{O}}
\def \cS{\mathcal{S}}

\def \Fr{\mathrm{Fr}}

\def \one{{\mathbf 1}}
\def \dB{\dot{\mathbf B}}
\def \dP{\dot{\mathbf P}}
\def \dU{\dot{\U}}
\def \dUi{\dot{\Ui}}
\def \dUiA{\dot{\U}^\imath_{\A}}
\def \DCU{\mathfrak{U}}

\def \brU{\breve{\U}}

\newcommand{\arxiv}[1]{\href{http://arxiv.org/abs/#1}{\tt arXiv:\nolinkurl{#1}}}


\title{Quantum symmetric spaces at roots of 1}

\author[Jinfeng Song]{Jinfeng Song}
\address{Department of Mathematics, The Hong Kong University of Science and Technology, Clear Water Bay, Hong Kong SAR, PR China}
\email{jfsong@ust.hk}

\author[Weinan Zhang]{Weinan Zhang}
\address{School of Mathematical Sciences, Shenzhen University, Shenzhen, Guangdong, 
518086, PR China}
\email{mathzwn@szu.edu.cn}
\date{}

\begin{document}

\begin{abstract}
We study the integral form of quantized coordinate algebras of symmetric spaces and its specializations at roots of unity. Our main tool is an integral localization, whose classical specialization recovers the open Borel orbit in the symmetric space. At a suitable odd root of unity, we prove that the quantum symmetric space is a finite projective module over its Frobenius center, which is canonically isomorphic to the classical coordinate algebra of the symmetric space. We show that the fibre algebras are isomorphic along each double Bruhat cell of the symmetric space. We determine the PI degree and propose a conjectural dimension formula for irreducible modules. In the simply connected case, we introduce a framed Kac--Moody data, which leads to a description of fibre algebras of quantum symmetric spaces via those of quantum Schubert cells.


\end{abstract}

\maketitle
\setcounter{tocdepth}{1}
\tableofcontents

\section{Introduction}\label{sec:introduction}

\subsection{Background}

The quantized coordinate algebra of a reductive group $G$ arises from the theory of Drinfeld-Jimbo quantum groups \cite{Dri87,Jim85}. The matrix coefficients of finite-dimensional type~$1$ modules over quantum groups form a Hopf algebra $\O_q[G]$, which quantizes the classical coordinate algebra $\O[G]$. The theory of canonical bases provides an integral form $\O_\A[G]$ over $\A=\Z[q,q^{-1}]$; see \cite{Kas93,Lus94}. Lusztig \cite{Lus09} used this construction to establish the Chevalley group scheme over $\Z$. 

De Concini--Lyubashenko \cite{DCL94} initiated the study of the quantized coordinate algebra $\O_\epsilon[G]$ at an root of unity $\epsilon$ of suitable odd order; see also Brown--Gordon \cite{BG02} and the references therein. A central feature of this theory is its connection with Poisson geometry, in the spirit of the orbit method. The dual quantum Frobenius map \cite{Lus90a} identifies $\O[G]$ with a central subalgebra over which $\O_\epsilon[G]$ is finite projective, reducing the study of irreducible modules to that of finite-dimensional fibre algebras. These fibre algebras are isomorphic along each double Bruhat cell \cite[Section~9]{DCL94}; the double Bruhat cells are precisely the $T$-leaves of the standard Poisson structure on $G$.

It is natural to extend the above root-of-unity picture from reductive groups to symmetric spaces. The construction of quantum symmetric spaces by Letzter was one of the motivations for the theory of quantum symmetric pairs; cf. \cite{Let03}. To pursue the root-of-unity extension, one needs an integral form of the quantum symmetric space. Using the icanonical bases developed by Bao--Wang \cite{BW18b} in the theory of quantum symmetric pairs, Bao and the first author \cite{BS24} constructed such an integral form for quantum symmetric spaces, whose specialization at $q=1$ recovers the coordinate algebra of symmetric spaces.


\subsection{Goal}

Let $\theta$ be an involution on $G$, and let $S=G/K$ be the symmetric space associated with $K=G^\theta$. Fix opposite Borel subgroups $B=B_+$ and $B_-$ with a $\theta$-stable maximal torus $T=B_+\cap B_-$ such that they form a $\theta$-split pair; see Section~\ref{sec:quantized-coordinate-algebras} for details.

The goal is to study structures and representation theory for quantum symmetric spaces at roots of unity. Along the way, we will relate three structures attached to $\ss$:  the geometry of the $T$-leaf stratification of $\ss$, the integral form $\O_\A[S]$ of the quantum symmetric space, and the representation theory of its specialization $\O_\epsilon[S]$ at roots of unity.

\subsection{Frobenius center}

Let $\ell>0$ be an odd integer, coprime to $3$ whenever $G$ has a component of type $G_2$.  Let $\epsilon$ be a primitive $\ell$-th root of unity. We use $\O_\epsilon[G]$ and $\O_\epsilon[\ss]$ to denote specializations of $\O_\A[G]$ and $\O_\A[\ss]$ at $q=\epsilon$, respectively. 

Our first result extends the Frobenius-center construction to quantum symmetric spaces. We show that the dual quantum Frobenius map $\Fr^*:\O[G]\rightarrow\O_\epsilon[G]$ restricts to an embedding of $\O[S]$ into the center of $\O_\epsilon[S]$. For $x\in S$, let $\mathfrak m_x\subset\O[S]$ be its maximal ideal, and write
\[
\O_\epsilon[S](x)=\O_\epsilon[S]/\Fr^*(\mathfrak m_x)\O_\epsilon[S]
\]
for the corresponding fibre algebra.

\begin{customthm}{A}[Theorems~\ref{thm:Fr}, \ref{thm:projectivity}, and~\ref{thm:fiberiso}; Proposition~\ref{prop:poisson-order}]
\label{thm:intro-frobenius}
The dual quantum Frobenius map $\Fr^*$ identifies
$\O[\ss]$
with a central subalgebra of $\O_\epsilon[\ss]$. More precisely, the following diagram commutes: 
\begin{equation}
        \begin{tikzcd}
            & \O[S] \arrow[r,hook,"\Fr_S^*"] \arrow[d,hook] & \O_\epsilon[S] \arrow[d,hook] \\ & \O[G] \arrow[r,hook,"\Fr^*"] & \O_\epsilon[G]
        \end{tikzcd}
\end{equation}
Moreover, we have:
\begin{enumerate}[(1)]
\item $\O_\epsilon[\ss]$ is a finite projective $\O[\ss]$-module of rank $\ell^{\dim S}$;
\item $\O_\epsilon[S]$ admits a Poisson order structure over its Frobenius central subalgebra $\O[S]$, inducing the standard Poisson bracket on $\O[S]$;
\item 
for $x,x'$ lying in the same $T$-leaf of $\ss$, their fibre algebras are isomorphic:
\[
 \O_\epsilon[S](x)\cong\O_\epsilon[S](x').
\]
\end{enumerate}
\end{customthm}

$T$-leaves of the natural Poisson structure on $S$ were determined by Evens--Lu \cite{EL06}; see also \cite{Lu14}. 
More explicitly, let $(\I=\wI\sqcup\bI,\tau)$ be the Satake diagram associated to $\theta$. Let $W$ be the Weyl group of $G$, and write $w^\bullet=w_0w_\bullet$, where $w_0$ and $w_\bullet$ are the longest elements of $W$ and the parabolic subgroup $W_{\bI}$, respectively. The symmetric space $S$ admits a stratification
\[
S=\bigsqcup_{u\leq w^\bullet} S_u.
\]
into $T$-leaves; see Section~\ref{sec:stratification} for the definition of $S_u$. Following Bao \cite{Bao26}, we call $S_u$ the \emph{double Bruhat cells of $S$}.
This stratification has found applications in recent work on total positivity \cite{Bao26}, cluster structures \cite{Ye26}, and Frobenius splittings \cite{BSprep} of symmetric spaces.

Theorem~\ref{thm:intro-frobenius} reduces the study of the Frobenius fibres to finitely many finite-dimensional algebras, obtained by choosing one point in each double Bruhat cell. 



\subsection{An open $B$-orbit}

A main ingredient in the proof of Theorem~\ref{thm:intro-frobenius} is a localization of $\O_\A[\ss]$ below, which allows us to establish projectivity on the open $B$-orbit and then extend it to $S$.
Let $N$ and $N_-$ be the unipotent radicals of $B_+$ and $B_-$, respectively. For $w\in W$, put $N(w)=N\cap\dot w N_-\dot w^{-1}$.

The map
\begin{equation}\label{eq:open-chart}
B\rightarrow \ss,\qquad b\mapsto bK
\end{equation}
has an image equal to the open dense $B$-orbit $B_\theta=BK/K\subset S$.

We construct in Section~\ref{sec:Btheta} an $\A$-subalgebra $\O_\A[B_\theta]$ of the integral quantum Borel algebra $\O_\A[B]$. It admits the smash-product description
\[
\O_\A[B_\theta]\cong\O_\A[N(w^\bullet)]\#\O_\A[T_\theta],
\qquad T_\theta=T/(T\cap K).
\]
Its specialization at $q=1$, followed by base change to $\C$, is the coordinate algebra of $B_\theta$.

For every dominant spherical weight $\mu$, we construct in Proposition~\ref{prop:psimu} a distinguished element $\psi_\mu\in\O_\A[S]$.  These elements form a multiplicative system, denoted by $\Psi$.

\begin{customthm}{B}[Proposition~\ref{prop:integral-localization}]
\label{thm:intro-localization}
The multiplicative system $\Psi$ is a left and right Ore set consisting of regular elements. There is an injective algebra homomorphism
\[
 \Gamma:\O_\A[S]\hookrightarrow
 \O_\A[B_\theta]
\]
which extends to an algebra isomorphism
\begin{equation}\label{eq:intro-localization}
 \Psi^{-1}\O_\A[S]\overset{\sim}{\longrightarrow}
 \O_\A[B_\theta]
\end{equation}
Under this isomorphism, $\psi_\mu$ is mapped to the torus character $\one_{-\mu}^*$.
\end{customthm}

Theorem~\ref{thm:intro-localization} is the main technical tool of this paper. Its integral formulation provides a simultaneous localization for both classical and quantum coordinate algebras. At $q=1$, after base change to $\C$, the map $\Gamma$ becomes the restriction map $\O[S]\rightarrow\O[B_\theta]$, and the localization recovers the open $B$-orbit $B_\theta\subset S$. At roots of unity, it allows us to use the fibre algebras of $\O_\epsilon[B_\theta] $ to understand the fibre algebras of $\O_\epsilon[\ss]$; see Proposition~\ref{le:GB}.  

\subsection{Dimensions of irreducible modules}
Since $\O_\epsilon[S]$ is finite over its Frobenius central subalgebra $\O[S]$, it is a polynomial identity (PI) algebra. We determine its PI degree and show that every irreducible module with Frobenius central character in the open double Bruhat cell has maximal dimension.

For $u\leq w^\bullet$, let $d_u$ denote the dimension of the symplectic leaves contained in $S_u$. Let $e\in W$ denote the identity element. Then $S_e$ is the
open dense double Bruhat cell of $S$.

\begin{customthm}{C}[Corollary~\ref{cor:PI-geometric} and Theorem~\ref{cor:dim}]
\label{thm:intro-PI}
\leavevmode

\begin{itemize}
\item[(1)] The PI degree of $\O_\epsilon[S]$ is
\[
\operatorname{PIdeg}\O_\epsilon[S]
=\ell^{d_e/2}
=\ell^{\frac12(\dim S-\dim\ker(1-w_0\tau))},
\]
where $\tau=-w_\bullet\theta$ acts on the character lattice $X=X^*(T)$ of $T$, and the kernel is taken on $X\otimes_\Z\Q$.

\item[(2)] Every irreducible $\O_\epsilon[S]$-module $V$ satisfies $\dim V\leq\ell^{d_e/2}$. If $V$ has Frobenius central character $x\in S_e$, then $\dim V=\ell^{d_e/2}$.
\end{itemize}
\end{customthm}

In the group case, the dimensions of irreducible modules of $\O_\epsilon[G]$ were studied by De Concini--Lyubashenko \cite{DCL94} and subsequently determined by De Concini--Procesi \cite{DCP97}. Precisely, for sufficiently good $\ell$, every irreducible $\O_\epsilon[G]$-module  with Frobenius central character $x\in G$ has the dimension $\ell^{\dim\mathcal L_x^G/2}$, where $\mathcal L^G_x$ is the symplectic leave of $G$ containing $x$; see also \cite[Theorem~2.3(b)]{BG02}. 

Motivated by these results, we propose the following dimension conjecture for general irreducible modules of quantum symmetric spaces.

\begin{conjintro}[Conjecture~\ref{conj:dimV}]
\label{conj:intro}
    Let $x\in S_u$ and let $V$ be an irreducible module of $\O_\epsilon[S](x)$. There exists a positive integer $N_u$, such that when $\ell$ is coprime to $N_u$, one has $\dim V=\ell^{d_u/2}$. 
\end{conjintro}

When $u=e$, this conjecture follows from Theorem~\ref{thm:intro-PI}.

\subsection{Framed Kac-Moody data}


To study irreducible modules with a general Frobenius central character, we turn to individual fibre algebras of $\O_\epsilon[\ss]$. By Theorem~\ref{thm:intro-localization}, it suffices to study fibre algebras for $\O_\epsilon[B_\theta]$. 

The variety $B_\theta$ inherits a natural stratification from the stratification of $\ss$ (see Section~\ref{sec:rwg})
\[
B_\theta =\bigsqcup_{v \leq w^\bullet} B_{\theta,v}, 
\]
which suggests the existence of a Bruhat atlas on the wonderful compactification of the adjoint type symmetric spaces; see \cite{HKL13,BH21, KS25} for the construction of Bruhat atlas for other varieties.

Now assume that the $G$ is simply connected. Inspired by the construction of Bruhat atlas, we construct a {\em framed Kac-Moody data}, by adjoining one framing vertex for every $\tau$-orbit in $\I_\circ$; see Section~\ref{sec:fdata} and Table~\ref{table:fdata}. Then $B_\theta$ is connected to the unipotent subgroup $\widetilde N(rw^\bullet)$ associated to this framed Kac-Moody data, while each strata $B_{\theta,v}$ is connected to the open Richardson strata $\widetilde R_{v,rw^\bullet}$; see Proposition~\ref{prop:framed-strata}. Here $r$ is the product of all simple reflections associated with the framing vertices.


The framed construction, together with Theorem~\ref{thm:intro-localization}, allows us to describe the fibre algebras of quantum symmetric spaces using those of quantum Schubert cells.

\begin{customthm}{D}[Theorem~\ref{thm:OSN}]
\label{thm:intro-fibres}
Assume that $G$ is simply connected. For $v\leq w^\bullet$, $x\in S_v$, and $y\in\widetilde R_{v,rw^\bullet}$, there is an algebra isomorphism
\[
 \O_\epsilon[S](x)\cong
 \O_\epsilon[\widetilde N(rw^\bullet)](y).
\]
\end{customthm}

Thanks to Theorem~\ref{thm:intro-fibres}, once dimensions of irreducible modules for quantum Schubert cells (of Kac-Moody type) are established at roots of unity, we can obtain dimensions of irreducible modules for quantum symmetric spaces and the Conjecture can be proved.



\subsection{Organization of the paper}
The paper is organized as follows. Section~\ref{sec:preliminaries} recalls the quantum-group and symmetric-pair notation and the stratification of symmetric spaces. Section~\ref{sec:integral-localization} establishes the integral localization theorem and its geometric interpretation. Section~\ref{sec:root-of-unity} studies specialization at roots of unity. In particular, we construct the Frobenius center, prove projectivity, establish the Poisson order structure and constancy of fibre algebras along $T$-leaves, prove the maximal-order property, and compute the PI degree and the dimensions of irreducible modules over the open double Bruhat cell. Section~\ref{sec:framed} develops the framed Kac--Moody construction and compares the fibre algebras of quantum symmetric spaces with those of quantum Schubert cells.

\subsection*{Acknowledgments}  The authors thank Milen Yakimov for informing them of ongoing work on quantum Schubert cells at roots of unity and for helpful discussions. The authors are grateful to Huanchen Bao for helpful discussions on the framed Kac--Moody construction. JS is supported by the Glorious Sun Charity Fund. WZ is partially supported by the New Cornerstone Science Foundation through the New Cornerstone Investigator Program awarded to Xuhua He.

\subsection*{Generative AI Usage} The authors used ChatGPT to assist with writing throughout the manuscript. Its use for exploring proofs was confined to Sections~\ref{sec:maximal-order} and~\ref{sec:PI-general}. The authors take full responsibility for the content of the paper and the correctness of its results.

\section{Preliminaries}\label{sec:preliminaries}

\subsection{Root datum and the algebra \texorpdfstring{$\f$}{f}}\label{sec:rootdata}
Let $(\I,\cdot)$ be a Cartan datum of finite type, with a root datum $(Y,X,\langle \cdot ,\cdot\rangle,\dots)$; cf. \cite[2.2.1]{Lus94}. Then $Y,X$ are two finitely generated free abelian groups with
\begin{itemize}
\item[(1)] a perfect bilinear form $\langle \cdot,\cdot\rangle: Y\times X\rightarrow \Z$,

\item[(2)] two embeddings $\Z[\I] \hookrightarrow X,i\mapsto \alpha_i$ and $\Z[\I] \hookrightarrow Y,i\mapsto \alpha_i^\vee$ such that $\langle \alpha_i^\vee,\alpha_j\rangle=2 \frac{i\cdot j}{i\cdot i}$.
\end{itemize}
Set $d_i=i\cdot i/2$ for $i\in\I$. We normalize the Cartan datum so that $\min_{i\in J}d_i=1$ for every irreducible component $J\subseteq\I$.
Let $c_{ij}=2 \frac{i\cdot j}{i\cdot i}$ and $C=(c_{ij})_{i,j\in\I}$ be the Cartan matrix. Let $Q=\Z[\alpha_i\mid i\in\I]\subset X$ be the root lattice, and $Q_+=\N[\alpha_i\mid i\in\I]$. Let $\Phi$ be the root system associated with $C$ and $\Phi^+$ be the set of positive roots.

Write
\[
X^+=\{\lambda\in X\mid\langle\alpha_i^\vee,\lambda\rangle\geq0\text{ for all }i\in\I\}
\]
for the set of dominant weights.

Let $q$ be an indeterminate and $\QQ$ be the field of rational functions in $q$. Set $q_i=q^{d_i}$ for $i\in \I$. Denote 
\[
[m]_i=\frac{q_i^m-q_i^{-m}}{q_i-q_i^{-1}},\qquad [m]_i!= [m]_i [m-1]_i \cdots [1]_i.
\]

Following \cite[1.2]{Lus94}, set $\f$ to be the $\QQ$-algebra generated by $\theta_i$ for $i\in \I$ subject to the quantum Serre relations 
\[
\sum_{r=0}^{1-c_{ij}} (-1)^r\theta_i^{(r)} \theta_j\theta_i^{(1-c_{ij}-r)}=0
\]
for $i\neq j$ where $\theta_i^{(k)}=\theta_i^k\slash [k]_i!$. The algebra $\f$ admits an $\N[\I]$-grading $\f =\bigoplus_{v\in \N[\I]} \f_v$ such that $\theta_i\in \f_{i}$. For $x\in \f_\nu$ homogeneous, set $|x|=\nu$. For $\nu=\sum_{i\in \I} \nu_i \,i\in \Z[\I]$, set
\begin{align}\label{eq:tnu}
\widetilde{\nu}=\sum_{i\in \I} \big(\frac{i\cdot i}{2}\big)\nu_i \,i \in \Z[\I].
\end{align}

Let $r:\f \rightarrow \f \otimes \f$ be the unique algebra homomorphism such that $r(\theta_i)=\theta_i \otimes 1 + 1\otimes \theta_i$ for all $i\in \I$. Here the multiplication on $\f\otimes \f$ is given by
\[
(x_1 \otimes y_1) (x_2\otimes y_2)=q^{|y_1|\cdot |x_2|} x_1 x_2 \otimes y_1 y_2,
\]
for homogeneous $x_1,x_2,y_1,y_2\in \f$. By \cite[1.2.2 and Proposition~1.2.3]{Lus94}, there is a symmetric nondegenerate bilinear form $(\cdot,\cdot)$ on $\f$ such that $(1,1)=1, (\theta_i,\theta_j)=\frac{\delta_{ij}}{1-q_i^{-2}}$ and
\begin{align}\label{eq:inprod}
(x,y_1 y_2)=(r(x),y_1\otimes y_2),\qquad \forall x,y_1,y_2\in \f.
\end{align}
By \eqref{eq:inprod}, for homogeneous $x,y\in \f$, $(x,y)=0$ unless $|x|=|y|$, and the restriction of this bilinear form to any $\f_\mu$ is nondegenerate.

Set $\A=\mathbb{Z}[q,q^{-1}]$. Unless otherwise specified, tensor products of integral forms are taken over $\A$. Let $\f_\A$ be the $\A$-subalgebra of $\f$ generated by $\theta_i^{(k)},k\in \N,i\in \I$. The algebra $\f_\A$ admits an $\A$-basis $\bB$, known as the \emph{canonical basis}; cf. \cite[14.4]{Lus94}. Set $\f_{\mu,\A}=\f_\mu\cap \f_\A$ and then $\f_\A=\bigoplus_{\mu\in \N[\I]} \f_{\mu,\A}$. It is known that $\bB\cap \f_{\mu,\A}$ is an $\A$-basis of $\f_{\mu,\A}$. 

Let $\{b^*|b\in \bB\}$ be the dual basis of $\f$ with respect to the bilinear form, that is, $(b^*,b')=\delta_{b,b'}$ for $b'\in \bB$. We define $\f_{\A}^*$ to be the $\A$-submodule of $\f$, spanned by $\{b^*|b\in \bB\}$. Then $\f_{\A}^*$ is an $\A$-subalgebra of $\f$.




\subsection{Quantum groups}

Let $G$ be the connected complex reductive group associated with the root datum $(Y,X,\langle \cdot ,\cdot\rangle,\dots)$. Let $\g$ be the Lie algebra of $G$, and let $W$ be the Weyl group of $\g$. We write $l(w)$ for the length of $w\in W$. Denote by $U(\g)$ the universal enveloping algebra of $\g$. 

Let $\U=\QQ\langle E_i,F_i,K_\mu;i\in\I,\mu\in Y\rangle$ be the Drinfeld--Jimbo quantum group associated with $(Y,X,\langle \cdot ,\cdot\rangle,\dots)$; cf. \cite[3.1.1]{Lus94}. Set $K_i=K_{d_i\alpha_i^\vee}$ for $i\in\I$. For $i\in \I$, we use $E_i^{(m)},F_i^{(m)}$ to denote the $q$-divided powers $\frac{E_i^m}{[m]_i!},\frac{F_i^m}{[m]_i!}$. Let $\U^+$ ({\em resp.,} $\U^-$) be the subalgebra of $\U$ generated by $E_i,i\in \I$ ({\em resp.,} $F_i,i\in \I$). Then there are natural algebra isomorphisms $\f\rightarrow \U^{\pm}, x\mapsto x^{\pm}$ such that $\theta_i^+=E_i,\theta_i^-=F_i$. We equip $\U^+$ with the same $\N[\I]$-grading and inner product as those on $\f$.

Via the isomorphism $\f\overset{\sim}{\rightarrow} \U^{+}$, $\f_\A$ is identified with the $\A$-subalgebra of $\U^{+}$ generated by $E_i^{(m)},m\in \N,i\in \I$, which we denote by $\U_{\A}^+$, and $\f_\A^*$ is identified with an $\A$-subalgebra of $\U^{+}$ which we denote by $\DCU^+$. One has $\U_{\A}^+ \otimes \Q(q) = \DCU^+ \otimes \Q(q)=\U^+$.

Let $\O_\A[N]=\bigoplus_{\mu\in Q_+}\text{Hom}_{\A}(\U^+_{\mu,\A},\A)$ be the graded dual of $\U^+_{\A}$. The nondegenerate bilinear form on $\U^+$ defines a linear isomorphism 
\begin{equation}\label{eq:alpha}
\alpha:\DCU^+\overset{\sim}{\longrightarrow} \O_\A[N].
\end{equation}
Let us endow $\O_\A[N]$ with the algebra structure such that $\alpha$ is moreover an algebra isomorphism.

\subsection{PBW basis and dual PBW basis}\label{sec:PBW}

Let $T_i=T''_{i,+1},i\in \I$ be braid group symmetries on $\U$; cf. \cite[37.1]{Lus94}. Let $w\in W$. Fix a reduced expression $\un{w}=s_{i_1}\cdots s_{i_l}$ for $w$. For $1\leq k\leq l$, set $\beta_k=s_{i_1} \cdots s_{i_{k-1}}(\alpha_{i_k})$. Define for $1\le k\le l$
\begin{equation}\label{eq:rootv}
E_{\beta_k}^{(m)}=T_{i_1} \cdots T_{i_{k-1}} (E_{i_k}^{(m)}),
\qquad
F_{\beta_k}^{(m)}=T_{i_1} \cdots T_{i_{k-1}} (F_{i_k}^{(m)}).
\end{equation}
For $\ba=(a_1,\ldots,a_l)\in \N^l$, set 
\[
E_{\un{w}}^{(\ba)}=E_{\beta_l}^{(a_l)}\cdots E_{\beta_1}^{(a_1)}, \qquad
F_{\un{w}}^{(\ba)}=F_{\beta_l}^{(a_l)}\cdots F_{\beta_1}^{(a_1)}.
\]
Let $\U^+_\A(w)$ be the $\A$-span of ordered monomials $\bE(\un{w})=\{E_{\un{w}}^{(\ba)}|\ba\in \N^{l(w)}\}$.  It was shown in \cite{DCP93} that $\U^+_\A(w)$ is an $\A$-subalgebra of $\U^+_\A$, independent of the choice of reduced expression of $w$, and the set $\bE(\un{w})$ forms an $\A$-basis of $\U^+_\A(w)$, which is called the PBW basis. By \cite[Proposition 2.17 and Theorem 4.5]{Lus90a}, 
\[
\U^+_\A(w_0)=\U_\A^+.
\]
Similarly, we set $\bF(\un{w})=\{F_{\un{w}}^{(\ba)}|\ba\in \N^{l(w)}\}$. 

By \cite[Proposition 38.2.3]{Lus94}, the PBW basis $\bE(\un{w})=\{E_{\un{w}}^{(\ba)}|\ba\in \N^{l(w)}\}$ is orthogonal, i.e., 
\begin{align}\label{eq:orthoPBW}
(E_{\un{w}}^{(\ba)}, E_{\un{w}}^{(\ba')})=0, \text{ unless } \ba=\ba'.
\end{align}

For $1\leq k\leq l$ and $\ba\in\N^l$, set
\begin{align}\label{def:tE}
\tE_{\beta_k}:=\frac{E_{\beta_k}}{(E_{\beta_k},E_{\beta_k})},
\qquad
\tE_{\un{w}}^{\ba}:=\frac{E_{\un{w}}^{(\ba)}}{(E_{\un{w}}^{(\ba)},E_{\un{w}}^{(\ba)})}.
\end{align}
Let $\U^+(w)=\U^+_\A(w)\otimes \Q(q)$. Then the set $\tbE(\un{w})=\{\tE_{\un{w}}^{\ba}|\ba\in \N^{l(w)}\}$ forms a $\Q(q)$-basis of $\U^+(w)$ dual to $\bE(\un{w})$; $\tbE(\un{w})$ is called the dual PBW basis. By definition, each dual PBW element is a nonzero $\QQ$-multiple of the ordered product of the individual dual root vectors:
\begin{align}
\tE_{\un{w}}^{\ba}\in\QQ^\times\,
\tE_{\beta_l}^{a_l}\cdots\tE_{\beta_2}^{a_2}\tE_{\beta_1}^{a_1}.
\end{align}
By abuse of notation, we sometimes simply write $\tE_{\beta_k}$ for $\alpha(\tE_{\beta_k})$ in $\O_\A[N]$.

Let $\DCU^+(w)$ be the $\A$-span of $\tbE(\un{w})$. By \cite[Theorem 4.27]{Ki12}, $\DCU^+(w)$ is an $\A$-subalgebra of $\U^+(w)$ generated by $\tE_{\beta_k}, 1\le k \le l(w)$. Again, one can identify $\DCU^+(w_0)$ with $\DCU^+$.

We write
\[
 \O_\A[N(w)]:=\alpha\bigl(\DCU^+(w)\bigr)\subseteq \O_\A[N]
\]
for the integral quantum coordinate algebra of the unipotent subgroup
$N(w)$.

\subsection{Modified quantum groups}

Let $\dot{\U}$ be the modified quantum group in \cite[23.1]{Lus94} and $\dot{\U}_\A$ be the $\A$-subalgebra of $\dot{\U}$ generated by $F_i^{(m)}\one_\lambda,E_i^{(m)}\one_\lambda$ for $\lambda\in X,m\in \N,i\in \I$. Let $\dB$ be the canonical basis of $\dU$; cf. \cite[25.2]{Lus94}. For a commutative $\A$-algebra $R$, set $\dU_R=\dU_\A\otimes_\A R$. Following \cite[1.11]{Lus09}, define
\[
\widehat{\U}_R=\prod_{b\in\dB}Rb,
\qquad
\widehat{\U}^{(2)}_R=\prod_{(b_1,b_2)\in\dB\times\dB}R(b_1\otimes b_2),
\]
where the basis elements are understood to be their images over $R$. The finiteness properties of the canonical-basis structure constants endow these spaces with $R$-algebra structures extending those on $\dU_R$ and $\dU_R\otimes_R\dU_R$, respectively. In particular, $\dU_R$ is naturally a subalgebra of $\widehat{\U}_R$. We regard $\widehat{\U}^{(2)}_R$ as the completed tensor product; all formal sums are interpreted coefficientwise.

Let $\lambda_1,\lambda_2\in X$. By \cite[23.1.5]{Lus94}, the comultiplication on $\U$ induces linear maps
\[
\Delta_{\lambda_1,\lambda_2}:\dU_\A \one_{\lambda_1+\lambda_2} \rightarrow \dU_\A  \one_{\lambda_1} \otimes \dU_\A  \one_{\lambda_2}.
\]
For $b\in \dB\cap \dU_\A \one_{\lambda_1+\lambda_2}$, write
\[
\Delta_{\lambda_1,\lambda_2}(b)=\sum_{\substack{b_1\in \dB \cap \dU_\A  \one_{\lambda_1}\\b_2\in \dB \cap \dU_\A  \one_{\lambda_2}}}\hat{m}_{b}^{b_1,b_2} b_1\otimes b_2.
\]
For $b_1\in \dB \cap \dU_\A  \one_{\lambda_1},b_2\in \dB \cap \dU_\A  \one_{\lambda_2},b\in \dB$, we set $\hat{m}_{b}^{b_1,b_2}=0$ if $b$ does not lie in $\dB\cap \dU_\A \one_{\lambda_1+\lambda_2}$.
\begin{lemma}[\text{cf. \cite[Lemma~1.16 and Section~1.17]{Lus09}}]
For any $b_1,b_2\in\dB$, there are only finitely many $b\in\dB$ such that $\hat m_b^{b_1,b_2}\neq0$. Consequently, for every commutative $\A$-algebra $R$, the formula
\[
\Delta_R\left(\sum_{b\in\dB}r_b b\right)
=\sum_{b_1,b_2\in\dB}
\left(\sum_{b\in\dB}r_b\hat m_b^{b_1,b_2}\right)b_1\otimes b_2
\]
defines an $R$-algebra homomorphism
\[
\Delta_R:\widehat{\U}_R\longrightarrow\widehat{\U}^{(2)}_R,
\]
where the structure constants are mapped from $\A$ to $R$.
\end{lemma}

Write $\Delta=\Delta_\A$. The following explicit formulas in $\widehat{\U}^{(2)}_\A$ can be derived from \cite[Lemma 3.1.4]{Lus94}:
\begin{align}
\label{eq:comul}
\begin{split}
\Delta(E_i^{(m)}\one_\lambda)=\sum_{\lambda'+\lambda''=\lambda} \sum_{k=0}^m q_i^{k(m-k)+k\langle \alpha_i^\vee, \lambda'\rangle} E_i^{(m-k)}\one_{\lambda'} \otimes E_i^{(k)}\one_{\lambda''},
\\
\Delta(F_i^{(m)}\one_\lambda)=\sum_{\lambda'+\lambda''=\lambda} \sum_{k=0}^m q_i^{k(m-k)-k\langle \alpha_i^\vee, \lambda''\rangle} F_i^{(k)}\one_{\lambda'} \otimes F_i^{(m-k)}\one_{\lambda''}.
\end{split}
\end{align}

The counit $\eta:{\dU}_\A\rightarrow \A$ is an $\A$-linear map such that 
$\eta(b)=\delta_{b,\one_0}$ for $b\in \dB$. It is known that $\eta$ is an $\A$-algebra homomorphism.

\subsection{Quantum symmetric pairs}
Let $(\I=\wI\sqcup \bI,\tau)$ be a Satake diagram which consists of a partition $\wI\sqcup \bI$ of $\I$ and a diagram involution $\tau:\I\rightarrow \I$ such that
\begin{itemize}
\item[(1)] $\tau(\bI)=\bI$,
\item[(2)]
 $\bw(\alpha_j) = - \alpha_{\tau j}$ for $j\in \bI$,
\item[(3)]
 If $j\in \wI$ and $\tau j =j$, then $\langle \rho_{\bullet}^\vee,\alpha_j\rangle\in \Z$.
\end{itemize}
Here $\bw$ is the longest element in the Weyl group $W_\bullet$ associated with $\bI$ and $\rho_{\bullet}^\vee$ is the half sum of positive coroots in the subcoroot system generated by $\I_\bullet$. 

Let
\begin{equation}\label{eq:iroot}
(\I,Y,X,\{\alpha_i^\vee\}_{i\in \I},\{\alpha_i\}_{i\in \I},\theta)
\end{equation}
be an iroot datum of type $(\I=\wI\cup \bI,\tau)$; cf. \cite[Definition 2.2]{So25}. Set 
\[
\breve{X}=(1-\theta)X,\qquad
\breve{X}^+=\breve{X}\cap X^+,
\]
and
\[
X_\imath=X/\breve{X},\qquad
Y^\imath=\{\mu\in Y\mid\theta\mu=\mu\}.
\]

The iquantum group $\Ui$ associated with the above iroot datum is the subalgebra of $\U$ generated by
\[
B_i=F_i+\vs_i T_{\bw}(E_{\tau i})K_i^{-1} \quad (i\in \wI),
\qquad K_\mu \quad (\mu\in Y^\imath),
\qquad E_j,F_j \quad (j\in \bI).
\]
Here we fix the parameters $\vs_i$ to be the same as in \cite[Lemma 4.2.1]{Wat26}.

Let $\dUi$ be the modified iquantum group defined in \cite{BW18b} and \cite[\S3.5]{BW21}, which contains the idempotents $\one_{\ov{\lambda}}$, for $\ov{\lambda}\in X_\imath$. The modified quantum group $\dot{\U}$ is naturally a $(\dUi,\dUi)$-bimodule by \cite[3.7]{BW18b}.  Following \cite[Definition 3.19]{BW18b}, we define 
\[
\dUiA:=\{u\in \dUi \mid u m\in \dot{\U}_{\A} , \forall m\in \dot{\U}_{\A} \}.
\]
The algebra $\dUiA$ can be viewed as an $\A$-subalgebra of $\widehat{\U}_\A$ via the map $x\mapsto \sum_{\lambda\in X} x\one_\lambda$.




Let $\dot{\bB}^\imath$ be the icanonical basis of $\dUi$ \cite{BW18b, BW21}, which forms an $\A$-basis of $\dUiA$. Let $\bP=\bP_{\I_\bullet}$ be the subalgebra of $\U$ generated by $E_i$, $F_j$, $K_\mu$, for $i\in \I_\bullet, j\in\I$, $\mu\in Y$. The algebra $\dot{\bP}=\bigoplus_{\lambda\in X} \bP\mathbf{1}_\lambda$ is a subalgebra of $\dot{\U}$, and its $\A$-form is defined by $\dot{\bP}_\A=\dot{\bP}\cap\dot{\U}_\A$. 
Note that $\dot{\bP}\cong \dot{\U}/\big(\sum_{\lambda\in X}\sum_x \dot{\U}x\mathbf{1}_\lambda\big)$, where the second sum is taken over all homogeneous $x\in\U^+$ whose weights are of the form $\sum_{i\in\I}a_i \alpha_i$ with $a_i\neq0$ for some $i\in\I_\circ$. 
For $\lambda\in X$, denote by $p_{\iota,\lambda}$ the composition map 
\begin{align}\label{def:dP}
\dU^\imath\one_{\ov{\lambda}}\overset{\iota_\lambda}{\longrightarrow} \dot{\U}\one_\lambda\longrightarrow \dot{\U}\mathbf{1}_\lambda/\big(\sum_{x}\dot{\U}x\one_\lambda\big)\longrightarrow \dot{\bP}\mathbf{1}_\lambda,
\end{align}
where the first map $\iota_\lambda$ is given by $x\one_{\ov{\lambda}}\mapsto x\one_\lambda$. 




\subsection{Base changes}\label{sec:bc}

Recall that $\dU_R=\dU_{\A}\otimes_{\A} R$ for a commutative $\A$-algebra $R$, and set $\dB_R=\{b\otimes 1\mid b\in \dB\}$. The relevant base change for us is as follows.

Let $\ell>0$ be an odd integer, which is coprime to $3$ if $\g$ has a component of G$_2$. Let $\epsilon\in\mathbb{C}$ be a primitive $\ell$-th root of unity. Write $\mathbb{C}_\epsilon$ (resp., $\mathbb{C}_1$) to denote the $\A$-algebra, which is isomorphic to $\mathbb{C}$ as rings and is viewed as an $\A$-algebra via the map $q\mapsto \epsilon$ (resp., $q\mapsto 1$).

We denote $\dU_\epsilon=\dU_{\mathbb{C}_\epsilon}$ and $\dU_1=\dU_{\C_1}$. We write $\dB_\epsilon=\dB_{\C_\epsilon}$ and $\dB_1=\dB_{\C_1}$. We similarly write $\widehat{\U}_\epsilon=\widehat{\U}_{\C_\epsilon}$ and $\widehat{\U}_1=\widehat{\U}_{\C_1}$ for the coefficientwise completions defined above.
 
Let $e_i^{(m)}\one_\lambda,f_i^{(m)}\one_\lambda$ be the images of $E_i^{(m)}\one_\lambda,F_i^{(m)}\one_\lambda$ in $\dU_1$. We still write $E_i^{(m)}\one_\lambda,F_i^{(m)}\one_\lambda$ to denote their images in $\dU_\epsilon$, if there is no confusion. 

\begin{proposition}[\text{cf. \cite[Theorem 35.1.9]{Lus94}}]
\label{prop:Fr}
There is a surjective algebra homomorphism $$\Fr:\dU_\epsilon \rightarrow \dU_1,$$ such that
for any $m\in \N,\lambda\in X$, one has
\begin{align}
\begin{split}
&\Fr(E_i^{(m)}\one_\lambda)=
\begin{cases}
e_i^{(m/\ell)}\one_{\lambda\slash \ell}, & \text{ if } \ell|m,\lambda\in \ell X,
\\
0, & \text{ otherwise}.
\end{cases}
\qquad
\Fr(F_i^{(m)}\one_\lambda)=
\begin{cases}
f_i^{(m/\ell)}\one_{\lambda\slash \ell}, & \text{ if } \ell|m,\lambda\in \ell X,
\\
0, & \text{ otherwise}.
\end{cases}
\end{split}
\end{align}
The map $\Fr$ is called the \emph{quantum Frobenius map}.
\end{proposition}

\begin{remark}\label{rem:Fr-projection}
The map $\Fr$ used here is Lusztig's quantum Frobenius map followed by a projection. More precisely, the target root datum in \cite[2.2.5 and 35.1.9]{Lus94} has character lattice
\[
X^\star=\{\lambda\in X\mid\langle\alpha_i^\vee,\lambda\rangle\in\ell\Z\text{ for all }i\in\I\}.
\]
After rescaling this lattice by $\ell^{-1}$, put $L=\ell^{-1}X^\star$. Lusztig's map becomes $\Fr_{\mathrm L}:\dU_\epsilon\rightarrow\dU_1(L)$, where $\dU_1(L)$ is the modified quantum group at $q=1$ for the character lattice $L$ and the original simple roots and coroots. Since $Q\subseteq X\subseteq L$, the root operators preserve weight classes modulo $X$. Hence retaining the summand with weight labels in $X$ defines an algebra projection $\pi:\dU_1(L)\rightarrow\dU_1$, given by
\[
\pi(e_i^{(n)}\one_\nu)=
\begin{cases}
e_i^{(n)}\one_\nu,&\nu\in X,\\
0,&\nu\notin X,
\end{cases}
\]
and similarly for $f_i^{(n)}\one_\nu$. Thus $\Fr=\pi\circ\Fr_{\mathrm L}$; the condition $\lambda\in\ell X$ in Proposition~\ref{prop:Fr} results from this projection. In general, $\ell X$ can be strictly smaller than $X^\star$.
\end{remark} 

For any $b\in \dB_\epsilon$, we write 
\[
\Fr(b)=\sum_{b'\in\dB_1}x_{b,b'}b',\quad \text{for }x_{b,b'}\in \C.
\]

\begin{lemma}[\text{cf. \cite[Lemma 5.1]{BS26}}]\label{lem:fm}
    For $b'\in \dB_1$, there are only finitely many $b\in \dB_\epsilon$, such that $x_{b,b'}\neq 0$. Hence $\Fr$ extends to an algebra homomorphism $\Fr:\widehat{\U}_\epsilon\rightarrow \widehat{\U}_1$.
\end{lemma}

\begin{remark}
    Although the setting in \cite{BS26} is restricted to quasi-split types, the proof of \cite[Lemma 5.1]{BS26} works for general types.
\end{remark}

\begin{proposition}[\text{\cite[Theorem 4.1]{BSa21}}]
\label{prop:dUiFr}
The quantum Frobenius map $\Fr$ restricts to an algebra homomorphism $\Fr:\dU^\imath_\epsilon \rightarrow \dU^\imath_1$.
\end{proposition}

\begin{remark}
The lattice assumption in \cite[Section~3.1]{BSa21}, namely
\[
X^\star\cap(1-\theta)X=(1-\theta)X^\star,
\]
is automatic under our assumptions. Indeed, let $z=(1-\theta)y\in X^\star$ with $y\in X$. Since $\ell X\subseteq X^\star$ and $\theta z=-z$, both
\[
\ell z=(1-\theta)(\ell y)
\quad\text{and}\quad
2z=(1-\theta)z
\]
belong to $(1-\theta)X^\star$. As $\ell$ is odd, this implies $z\in(1-\theta)X^\star$. The reverse inclusion follows from the $\theta$-stability of $X^\star$.

The map in Proposition~\ref{prop:dUiFr} is likewise the restriction constructed by Bao--Sale followed by a projection, analogous to that in Remark~\ref{rem:Fr-projection}.
\end{remark}

\subsection{Quantized coordinate algebras}\label{sec:quantized-coordinate-algebras}

Recall that $G$ is the connected complex reductive group associated with the
root datum fixed above. Let $\theta$ be the involution of $G$ attached to the
iroot datum \eqref{eq:iroot}; see \cite[Corollary~3.9]{So25}. Set
$K=G^\theta$ and $S=G/K$. 

Choose a Borel subgroup $B$ of $G$ and a $\theta$-stable maximal torus \(T\subset B\) such that \((T,B)\) is a \emph{split pair} for $\theta$; cf. \cite{Spr87}. Let $B_+=B$ and $B_-$ be the opposite Borel subgroup such that $T=B_+\cap B_-$. Then one has $\theta (B_\pm)=w^\bullet B_\pm w^\bullet$, where $w^\bullet=w_0w_\bullet$ is an involution in $W$. 


Let $R$ be an $\A$-algebra. Following \cite[\S3.1]{Lus09}, define
\[
\O_R[G]
=\bigl\{f\in\operatorname{Hom}_R(\dU_R,R)\bigm|
f(b)=0\text{ for all but finitely many }b\in\dB_R\bigr\}.
\]
This is a free $R$-module with the \emph{dual canonical basis}
$\{b^*\mid b\in\dB_R\}$, where $b^*$ are defined by $b^*(b')=\delta_{b,b'}$ for
$b,b'\in\dB_R$. We can also naturally view an element of $\O_R[G]$ as a linear function on $\widehat{\U}_R$.

Let us briefly recall the Hopf algebra structure on $\O_R[G]$. Let $\Delta_R$, $\eta_R$, and
$\mathsf S_R$ denote the comultiplication, counit, and antipode of the completion $\widehat{\U}_R$. The structure maps on $\O_R[G]$ are obtained by duality:
\begin{align*}
(fg)(x)&=(f\otimes g)\Delta_R(x),
&1_{\O_R[G]}&=\eta_R,\\
\Delta_G(f)(x\otimes y)&=f(xy),
&\varepsilon_G(f)&=f(1),
&\mathsf S_G(f)(x)&=f(\mathsf S_R(x)).
\end{align*}
Here $f\otimes g$ is evaluated on $\widehat{\U}^{(2)}_R$ using the finite supports of $f$ and $g$. The finiteness properties of the canonical basis ensure that these maps are
well defined and make $\O_R[G]$ an $R$-Hopf algebra; see
\cite[\S3.1]{Lus09}. 

Following \cite{BS24}, the quantized coordinate algebra of the
symmetric space is defined to be the following right-invariant subalgebra of $\O_R[G]$
\begin{equation}\label{eq:quantum-symmetric-space}
\O_R[S]=\O_R[G/K]
:=\bigl\{f\in\O_R[G]\bigm|
f(xy)=\eta_R(y)f(x),\quad
y\in\dU_R^\imath,\ x\in\dU_R\bigr\}.
\end{equation}

Thanks to \cite{Lus09,BS24}, the construction of $\O_R[G]$ and $\O_R[S]$ is functorial, that is, one has 
$$\O_R[G]\cong R\otimes _{\A}\O_\A[G],\qquad \O_R[S]\cong R\otimes _{\A}\O_\A[S]$$ as $R$-algebras.

For $R=\QQ$, we simply write $\O_q[G]=\O_{\QQ}[G]$ and $\O_q[S]=\O_{\QQ}[S]$.

Recall the $\A$-algebra $\C_1$ from Section \ref{sec:bc}. Write
$\O_1[G]=\O_{\C_1}[G]$. By \cite[Theorem~4.11]{Lus09}, one has an isomorphism of Hopf algebras
\begin{equation}\label{eq:O1G-classical}
\Phi_G:\O_1[G]\xrightarrow{\ \sim\ }\O[G],
\end{equation}
where $\O[G]$ is the ordinary coordinate algebra of the connected reductive group $G$. The formula
\begin{equation}\label{eq:Poisson-O1G}
\{\overline f,\overline g\}_1
=\overline{\frac{fg-gf}{q-1}},
\qquad f,g\in\O_\A[G],
\end{equation}
defines a Poisson Hopf structure on $\O_1[G]$, where for $f\in\O_\A[G]$, $\overline{f}$ denotes its image in $\O_1[G]$. Let us equip $G$ with the standard Poisson structure associated with the pair $(B_+,B_-)$. Then the map $\Phi_G$ can moreover be taken to be an isomorphism of Poisson Hopf algebras.

Similarly write $\O_1[S]=\O_{\C_1}[S]$. By \cite[Theorem~4.8]{BS24}, one has a canonical
algebra isomorphism
\begin{equation}\label{eq:O1S-classical}
\Phi_S:\O_1[S]\xrightarrow{\ \sim\ }\O[S],
\end{equation}
which fits into the commutative diagram
\begin{equation*}
    \begin{tikzcd}
        & \O_1[G] \arrow[r,"\Phi_G"] & \O[G] \\ 
        & \O_1[S] \arrow[u,hook] \arrow[r,"\Phi_S"] & \O[S] \arrow[u,hook]
    \end{tikzcd}
\end{equation*}

By \cite[Proposition 3.1]{So26}, under the assumption that $(B_+,B_-)$ is $\theta$-split, the subgroup $K$ is coisotropic with respect to the standard Poisson structure on $G$. In particular, the coordinate algebra $\O[S]$ is a Poisson subalgebra of $\O[G]$. It is then clear that $\Phi_S$ is an isomorphism of Poisson algebras.

\subsection{Matrix coefficients}

Let $\mu\in X^+$ and $V(\mu)$ be the irreducible $\U$-module with highest weight $\mu$ and a highest weight vector $v_{\mu}$; $V(\mu)$ is naturally a $\dU$-module. 

For $f\in V(\mu)^*,v\in V(\mu)$, let $c_{f,v}^\mu$ be the following linear function on $\dU$
\begin{align}
c_{f,v}^\mu:\dU \rightarrow \QQ, \quad c_{f,v}^\mu(x)=f(xv).
\end{align}

Let $\mathcal{B}$ be a basis of $V(\mu)$. For any $b\in \mathcal{B}$, define $b^*$ to be the linear function on $V(\mu)$ such that $b^*(b')=\delta_{b,b'}$ for any $b'\in \mathcal{B}$. The linear functions $c_{b_1^*,b_2}^\mu$ for $b_1,b_2\in \mathcal{B}$ are called the {\em matrix coefficients} with respect to the basis $\mathcal{B}$. Similarly, one can define matrix coefficients for general $\U$-modules.

The canonical basis of $V(\mu)$ is denoted by $\bB(\mu)=\{b^-v_\mu | b\in \bB, b^- v_\mu\neq 0 \}$. Let $V(\mu)_\A$ be the $\A$-span of $\bB(\mu)$; $V(\mu)_\A$ admits a natural action of $\dU_\A$. Denote by $\omega_i,i\in \I$ the fundamental weights.

\begin{proposition}

\begin{itemize}
\item[(1)] $c_{b_1^*,b_2}^\mu\in\O_\A[G]$ for $b_1,b_2\in \bB(\mu)$.

\item[(2)] The set $\{c_{b_1^*,b_2}^\mu\vert b_1,b_2\in \bB(\mu),\mu\in X^+\}$ forms an $\QQ$-basis of $\O_{\QQ}[G]$.

\item[(3)] The set $\{c_{b_1^*,b_2}^{\omega_i}\vert b_1,b_2\in \bB(\omega_i),i\in \I\}$ forms a generating set of $\O_{\QQ}[G]$ when G is simply connected.
\end{itemize}
\end{proposition}

\begin{proof}
We prove (1). Let $b\in \dB,b_1,b_2\in \bB(\mu)$. Write $b_1=y_1 v_\mu,b_2=y_2v_\mu$ for $y_1,y_2\in \bB^-$ and write
\[
b (y_2\one_\mu)=\sum_{b'\in \dB} m^{b'}_{b,y_2\one_\mu} b'.
\]
for some $m^{b'}_{b,y_2\one_\mu} \in \A$. By \cite[Theorem 25.2.1]{Lus94}, $b'v_\mu$ is either $0$ or an element in $\bB(\mu)$, and $b'v_\mu=y_1v_\mu$ if and only if $b'=y_1\one_\mu$. Hence, 
\[
c_{b_1^*,b_2}^\mu(b)=m^{y_1 \one_\mu}_{b,y_2\one_\mu}\in \A.
\]
This shows that $c_{b_1^*,b_2}^\mu\in \O_\A[G]$. 

(2). We show that $c_{b_1^*,b_2}^\mu$ form a basis of $\O_{\QQ}[G]$. Indeed, this follows from the Peter--Weyl decomposition (cf. \cite{Kas93})
\[
\O_{\QQ}[G]=\bigoplus_{\mu\in X^+} V(\mu)^* \otimes V(\mu).
\]

 (3). Thanks to (2), it suffices to show that every $c_{b_1^*,b_2}^\mu$ for $b_1,b_2\in \bB(\mu),\mu\in X^+$ lies in the subalgebra generated by $\{c_{b_1^*,b_2}^{\omega_i}\vert b_1,b_2\in \bB(\omega_i),i\in \I\}$. Write $\mu=\omega_{i_1} +\ldots +\omega_{i_k}$. By \cite[Proposition 27.1.7]{Lus94}, $V(\mu)$ is a based submodule of $V(\omega_{i_1})\otimes \cdots \otimes V(\omega_{i_k})$ and therefore $c_{b_1^*,b_2}^\mu$ is identified with a matrix coefficient with respect to the canonical basis $\bB(\omega_{i_1},\ldots, \omega_{i_k})$ of $V(\omega_{i_1})\otimes \cdots \otimes V(\omega_{i_k})$. Moreover, by \cite[Theorem 27.3.2]{Lus94}, the transition matrix between  $\bB(\omega_{i_1},\ldots, \omega_{i_k})$ and $\bB(\omega_{i_1})\otimes \cdots \otimes \bB(\omega_{i_k})$ is unipotent upper-triangular with entries in $\A$. Hence, $c_{b_1^*,b_2}^\mu$ is an $\A$-linear combination of matrix coefficients with respect to the basis $\bB(\omega_{i_1})\otimes \cdots \otimes \bB(\omega_{i_k})$; these matrix coefficients have the form 
\[
c^{\omega_{i_1}}_{x_{i_1}^*,y_{i_1}}\cdots c^{\omega_{i_k}}_{x_{i_k}^*,y_{i_k}}
\]
where $x_{i_j},y_{i_j}\in \bB(\omega_{i_j})$. Therefore, we have proved that $c_{b_1^*,b_2}^\mu$ lies in the subalgebra generated by $\{c_{b_1^*,b_2}^{\omega_i}\vert b_1,b_2\in \bB(\omega_i),i\in \I\}$ as desired.
\end{proof}

Let $\mu\in \breve{X}^+$. By \cite[Theorem 3.6]{Let03}, there is a nonzero vector, unique up to scalar, $w_\mu\in V(-w_0\mu)$, such that 
\begin{align}\label{eq:wmu}
uw_\mu=\eta(u)w_\mu
\end{align}
for all $u\in \Ui$. 

\begin{proposition}
The set $\{c^{-w_0\mu}_{b_1^*,w_\mu}| \mu\in \breve{X}^+,b_1\in \bB(-w_0\mu)\}$ forms a $\QQ$-basis of $\O_{\QQ}[\ss]$.
\end{proposition}

\begin{proof}
By the choice of $w_\mu$, $c^{-w_0\mu}_{b_1^*,w_\mu}\in \O_{\QQ}[\ss]$. Due to the following decomposition of $\O_{\QQ}[\ss]$ (cf. \cite[Theorem 4.9]{BS24})
\[
\O_{\QQ}[\ss]=\bigoplus_{\mu\in \breve{X}^+} V(-w_0\mu)^*,
\]
these matrix coefficients form a basis of $\O_{\QQ}[\ss]$.
\end{proof}

\subsection{Stratification of symmetric spaces}\label{sec:stratification}

Retain the same notation as in Section \ref{sec:quantized-coordinate-algebras}. Let $$\iota:\ss\rightarrow G;\qquad gK\mapsto g\theta(g)^{-1}.$$
It is known that the map $\iota$ maps $\ss$ isomorphically to a closed subvariety of $G$. For $u\in W$, define 
\[
\ss_u=\iota^{-1}(B_- u B_+w_0w_\bullet) =\iota^{-1}(B_- u w_0 w_\bullet \theta(B_+))
\]
to be the \emph{double Bruhat cell} of the symmetric space $S$.

\begin{remark}
For the symmetric space $(G\times G,G)$ of diagonal type, the strata $S_u,u\in W\times W$ are identified with a double Bruhat cell of $G$; see \cite[Section 7.3]{Lu14}.
\end{remark}

\begin{lemma}[\text{\cite[Lemma 2.4]{Bao26}}]
{\,}
\begin{itemize}
\item[(1)] $\ss_u$ is nonempty if and only if $u\leq \ww$, and $\ss=\bigsqcup_{u\leq \ww} S_u$.
\item[(2)] For $u\leq\ww$, $\ss_u\subset\ss$ is a smooth connected locally closed subvariety of codimension $l(u)$.
\end{itemize}
\end{lemma}

In particular, $S_e$ is the open dense double Bruhat cell of $S$.

Let $N$ be the unipotent radical of $B=B_+$, and $N_-$ be the unipotent radical of $B_-$. For $w\in W$, set $N(w)=N\cap w N_-w^{-1}$. Let us identify $N(w)$ with the Schubert cell $B w\cdot B\subset G/B$ via
\[
i:N(w)\overset{\sim}{\rightarrow} Bw \cdot B,\qquad  u\mapsto uw\cdot B.
\]

The torus $T$ acts on $\ss$ via left multiplication. Recall that $G$ is equipped with the standard Poisson structure associated with the pair $(B_+,B_-)$, and $S$ admits a canonical Poisson structure such that the projection $G\rightarrow S$ is Poisson. It is known that the $T$-action on $\ss$ preserves the Poisson structure.
\begin{proposition}[\text{cf. \cite[Theorem 1.1]{Lu14}}]\label{prop:dimSu}
    For $u\leq w^\bullet$, $\ss_u$ is a $T$-leaf, namely the union of the $T$-translates of a symplectic leaf of $\ss$. Any symplectic leaf contained in $\ss_u$ has dimension 
    \begin{equation}\label{eq:du}
    d_u=\dim \ss_u-\dim \text{ker}(1-uw_0\tau)=l(w_0w_\bullet)+\text{rank }(1-\theta)-l(u)-\dim \text{ker}(1-uw_0\tau).
    \end{equation}
    Here $\tau=-w_\bullet\theta$ acts on the full character lattice $X$ and extends the Satake involution, and the rank and kernel are computed on $X\otimes_\Z\Q$.
\end{proposition}

\section{Integral form and localization}\label{sec:integral-localization}

In this section, we construct an integral localization of $\O_\A[S]$ and identify it with $\O_\A[B_\theta]$, the quantum coordinate algebra of the open Borel orbit. The latter is a smash product of $\O_\A[N(w_0 w_\bullet)]$ and the torus algebra $\O_\A[T_\theta]$.

Retain the same notation as in the previous section. Set $N_0=l(w_0)$, $M=l(w_\bullet)$ and $L=N_0-M$. Fix a reduced expression $\un{w_0 w_\bullet}=s_{i_1}s_{i_2}\cdots s_{i_L}$, and a reduced expression $\un{\bw}=s_{j_1}s_{j_2}\cdots s_{j_M}$ in $W_\bullet$. Denote $i_{L+k}=j_k$ for $1\le k\le M$. Then $\un{w_0}=s_{i_1}s_{i_2}\cdots s_{i_{N_0}}$ is a reduced expression for $w_0$. 

The algebra $\U^+_{\A}(\bw)$ is naturally identified with the subalgebra of $\U_\A^+$ generated by $E_j^{(k)},j\in \bI,k\in \N$ with a PBW basis $\bE(\un{\bw})$. Let $\U^+_{\A}(\bw)_{>}$ be the $\A$-submodule of $\U^+_{\A}(\bw)$ spanned by $\{E_{\un{\bw}}^{(\ba)}|\ba\in \N^{M},\ba\neq 0\}$. Then the PBW basis implies the following direct sum decomposition
\begin{align}\label{eq:dsum}
\U_\A^+ = \U^+_{\A}(w_0\bw)\oplus \U^+_{\A}(w_0 \bw)\U^+_{\A}(\bw)_{>}.
\end{align}

\subsection{Quantum Iwasawa decomposition}

Recall that $\dU$ is a $(\dUi,\dUi)$-bimodule and $\dU_\A$ is a $(\dUiA,\dUiA)$-bimodule.

For $\lambda\in X$, set $\dU_{\A}^\lambda$ to be the $\A$-span of $\{x \one_\lambda y|x\in \U^+_\A(w_0 \bw), y\in \dUiA\}$. 
\begin{lemma}\label{lem:idecomp}
We have $\dU_\A=\bigoplus_{\lambda\in X} \dU_{\A}^\lambda$.
\end{lemma}

\begin{proof}
By \cite[Lemma 23.2.2]{Lus94}, the set of elements $x \one_\lambda y$, where $x,y$ run over PBW bases of $\U^+,\U^-$ respectively and $\lambda$ runs over $X$, forms an $\A$-basis of $\dU_\A$. Using the PBW basis $\bE(\un{w_0})$, it is clear that there is an $\A$-linear isomorphism
\[
\U^+_\A \cong \U^+_\A(w_0\bw) \otimes \U^+_\A(\bw).
\]
Hence, the set of elements $x_1 \one_\lambda x_2 y$, where $x_1,x_2, y, \lambda$ run over $\bE(\un{w_0\bw}),\bE(\un{\bw}), \bF(\un{w_0}), X$ respectively, forms an $\A$-basis of $\dU_\A$. Recall from \eqref{def:dP} the parabolic subalgebra $\dP_\A$ of $\dU_\A$ associated with $\bI$. Note that for any $\lambda\in X$, there is a natural $\A$-linear isomorphism between $\U^+_\A(\bw) \U^-_{\A}$ and $\one_\lambda\dP_\A$. Moreover, elements of the form $x_2y$, where $x_2, y$ run over $\bE(\un{\bw}), \bF(\un{w_0})$ respectively, form an $\A$-basis of $\U^+_\A(\bw) \U^-_{\A}$; we denote the image of this basis in $\one_\lambda\dP_\A$ by ${}_\lambda\mathcal{B}$. Then the set 
\[
\{x \one_\lambda y|\lambda\in X, x\in \bE(\underline{w_0 \bw}), y\in {}_\lambda\mathcal{B}\}
\]
forms an $\A$-basis for $\dU_\A$. By \cite[Corollary 6.20]{BW18b}, the left-idempotent version of $p_{\iota,\lambda}$ gives an $\A$-linear isomorphism $\one_{\ov{\lambda}} \dUiA\cong\one_\lambda\dP_\A$. This isomorphism is obtained by a leading-term argument; see \cite[Lemma 3.22 and the proof of Corollary 6.20]{BW18b}. More precisely, filter the PBW decomposition by the total height of its negative part. The generators $B_i$ have leading terms $F_i$, while the additional terms have smaller negative height. The integral version of this argument gives, for each $b\in{}_{\lambda}\mathcal B$, a lift $\widetilde b\in\one_{\ov\lambda}\dUiA$ such that
\[
\one_\lambda\widetilde b=b+\text{terms of smaller negative height},
\]
with all coefficients in $\A$; these lifts form an $\A$-basis of $\one_{\ov\lambda}\dUiA$. Multiplication on the left by $x\in\bE(\underline{w_0\bw})$ preserves this triangularity. Thus the elements $x\one_\lambda\widetilde b$ form an $\A$-basis of $\dU_\A$, by induction on the negative height. Replacing the basis of lifts by the icanonical basis of each $\one_{\ov\lambda}\dUiA$, we conclude that

\begin{align}\label{eq:iwabasis}
\{x \one_\lambda y|\lambda\in X, x\in \bE(w_0 \bw), y\in\dB^\imath\cap \one_{\ov{\lambda}}\dUiA \}
\end{align}
forms an $\A$-basis for $\dU_\A$. The desired statement follows. 
\end{proof}

\begin{lemma}\label{lem:localiso}
Let $\zeta\in X_\imath$ and $\lambda\in X$ such that $\ov{\lambda}=\zeta$. The following map
\[
\U^+_\A(w_0 \bw) \otimes \one_\zeta\dUiA  \rightarrow \dU_{\A}^\lambda,
\qquad x \otimes \one_\zeta y \mapsto x\one_\lambda y,
\]
is an isomorphism of right $\dUiA$-modules.
\end{lemma}

\begin{proof}
It follows from the basis \eqref{eq:iwabasis} of $\dU_\A$ that this map is a well-defined linear isomorphism. Clearly, this map is also a homomorphism of right $\dUiA$-modules.
\end{proof}

We define $\dU^+_\A(w_0 \bw)$ to be the $\A$-submodule of $\dU_\A$ spanned by all elements of the form $x\one_\lambda,x\in \U^+_\A(w_0 \bw),\lambda\in X$. 
By Lemma \ref{lem:idecomp}, the following map
\begin{equation}
m^\imath: \dU^+_\A(w_0 \bw) \otimes \dUiA \rightarrow \dU_\A, 
\qquad x\otimes y  \mapsto x  y 
\end{equation}
is a surjective homomorphism of right $\dUiA$-modules.

Recall the lattice $\breve{X}=(1-\theta)X$ and the quotient $X_\imath=X/\breve{X}$. We define $\brU^+_\A(w_0 \bw)$ to be the $\A$-submodule of $\dU_\A$ spanned by all elements of the form $x\one_\lambda,x\in \U^+_\A(w_0 \bw),\lambda\in \breve{X}$. By \cite[Lemma 23.2.2]{Lus94}, the set $\{x\one_\lambda| x\in\bE(\underline{w_0 \bw}),\lambda\in \breve{X}\}$ is an $\A$-basis of $\brU^+_\A(w_0 \bw)$.

\begin{proposition}\label{prop:psimu}
    (1) For any $\mu\in \breve{X}^+$, there is a unique element $\psi_\mu\in \mathcal{O}_{\A}[\ss]$, such that 
    \begin{equation}\label{eq:psimu}
        \psi_\mu(u\one_\lambda)=\eta(u)\delta_{\lambda,-\mu},
\end{equation} for $\lambda\in X$ and $u\in \U^+$. Here $\eta:\U\rightarrow \Q(q)$ is the counit of the quantum group.

    (2) For $\mu,\mu'\in \breve{X}^+$, one has $\psi_\mu\psi_{\mu'}=\psi_{\mu+\mu'}$.
\end{proposition}

\begin{proof}
    We first prove (1). The uniqueness follows from Lemma~\ref{lem:idecomp} and the definition of $\O_q[S]$; cf. \eqref{eq:quantum-symmetric-space}. It suffices to show the existence. For $\mu\in\breve{X}^+$, recall that $V(-w_0\mu)$ is the irreducible $\U$-module with the lowest weight $-\mu$. Recall from \eqref{eq:wmu} the vector $w_\mu\in V(-w_0\mu)$. Let $v_{-\mu}\in V(-w_0\mu)$ be a lowest weight vector, and $v_{-\mu}^*\in V(-w_0\mu)^*$ be the linear form such that $v_{-\mu}^*(v_{-\mu})=1$ and $v_{-\mu}^*(v)=0$ for $v\in V(-w_0\mu)_\xi$ where $\xi\neq -\mu$. 
    
    Consider the linear map $\psi_\mu':\dU\rightarrow\QQ$ defined by $u\mapsto v_{-\mu}^*(uw_\mu)$. Note that $\psi_\mu'$ is not a zero map, because $V(-w_0\mu)$ is irreducible. Since the set $\{b\in\dB\mid b\mid _{V(-w_0\mu)}\not\equiv 0\}$ is finite, we deduce that $\psi'_\mu\in \O_q[G]$. By the choice \eqref{eq:wmu} of $w_\mu$, we have $\psi_\mu'\in \O_q[\ss]$. For $u\in \U^+$ and $\lambda\in X$, one has
    \[
    \psi_\mu'(u\one_\lambda)=\left\{\begin{array}{ll}
        \eta(u)v^*_{-\mu}(w_\mu), & \text{if }\lambda=-\mu, \\
        0, & \text{if }\lambda\neq -\mu.
    \end{array}\right.
    \]
    Let $c=v^*_{-\mu}(w_\mu)=\psi_{\mu}'(\one_{-\mu})$. Then $c\neq 0$, otherwise $\psi_\mu'=0$ by Lemma~\ref{lem:idecomp}, which is a contradiction. Set $\psi_\mu=c^{-1}\psi_\mu'\in \mathcal{O}_q[\ss]$. Then $\psi_\mu(u\one_\lambda)$ satisfies \eqref{eq:psimu}. Moreover, $\psi_\mu$ belongs to the integral form $\O_{\A}[\ss]$ by Lemma~\ref{lem:idecomp}. This completes the proof.

    To prove (2), it suffices to show that $(\psi_\mu\psi_{\mu'})
    (u\one_\lambda)=\eta(u)\delta_{\lambda,-\mu-\mu'}$, for $u\in \U^+$ and $\lambda\in X$. This is direct to check.
\end{proof}

\subsection{The algebra \texorpdfstring{$\O_\A[B]$}{O A[B]}}

Set $\dot{\mathbf{B}}^+=\big\{b_1^+\one_\lambda\mid b_1\in \bB,\lambda\in X\big\}$. It is known that $\dot{\mathbf{B}}^+$ is a subset of $\dot{\mathbf{B}}$. Let $\O_\A[B]\subset\O_\A[G]$ be the $\A$-span of $\big\{b^*\mid b\in \dot{\mathbf{B}}^+\big\}$. Then $\O_\A[B]$ is an $\A$-module direct summand of $\O_{\A}[G]$. Let 
\begin{equation}\label{eq:pi}
    \pi:\O_\A[G]\rightarrow \O_\A[B]
\end{equation}
be the linear map given by $b^*\mapsto b^*$ if $b\in \dot{\mathbf{B}}^+$, and $b^*\mapsto 0$ if otherwise. Following \cite[3.5]{Lus09}, the composition
\[
\O_\A[B] \otimes \O_\A[B] \rightarrow \O_\A[G] \overset{\pi}{\longrightarrow} \O_\A[B],
\]
defines an algebra structure on $\O_\A[B]$, where the first map is given by the multiplication in $\O_\A[G]$. Thus we regard $\O_\A[B]$ as an $\A$-module direct summand of $\O_\A[G]$, equipped with the quotient algebra structure induced by $\pi$.

Let $\O_\A[T]$ be the $\A$-submodule of $\O_\A[B]$ spanned by $\{\one_\lambda^*\mid \lambda\in X\}$.

\begin{lemma}\label{lem:onela}
For $\lambda,\lambda'\in X$, we have $\one_\lambda^*\one_{\lambda'}^*=\one_{\lambda+\lambda'}^*$ in $\O_\A[B]$. Hence, $\O_\A[T]$ is a subalgebra of $\O_\A[B]$.
\end{lemma} 

\begin{proof}
Let $x\in \U^+$ be a homogeneous element. Write $r(x)=\sum x_{(1)}\otimes x_{(2)}$. Then 
\[
\Delta_{\lambda,\lambda'}(x\one_{\lambda+\lambda'})
=\sum q^{\langle \widetilde{|x_{(2)}|},\lambda\rangle} x_{(1)} \one_{\lambda}\otimes x_{(2)} \one_{\lambda'}.
\]
Since $|x_{(1)}|+|x_{(2)}|=|x|$, the coefficient of $\one_\lambda\otimes\one_{\lambda'}$ on the right-hand side is $\eta(x)$. Moreover, $(\one_\lambda^*\one_{\lambda'}^*)(x\one_\mu)=0$ unless $\mu=\lambda+\lambda'$. Hence
\[
(\one_\lambda^*\one_{\lambda'}^*)(x\one_\mu)
=\eta(x)\delta_{\mu,\lambda+\lambda'}
=\one_{\lambda+\lambda'}^*(x\one_\mu),
\]
which proves the assertion.
\end{proof}

Recall from \eqref{eq:alpha} the isomorphism $\alpha:\DCU^+\overset{\sim}{\longrightarrow} \O_\A[N]$. Define the smash product $\O_\A[N]\#\O_\A[T]$ as follows. It is isomorphic to $\O_\A[N]\otimes \O_\A[T]$ as an $\A$-module, and the algebra structure is given by
\begin{equation}\label{eq:com}
(\alpha(u)\otimes \one_{\lambda}^*)\cdot (\alpha(u')\otimes \one_{\lambda'}^*)=q^{\langle \widetilde{|u'|},\lambda\rangle}\alpha(uu')\otimes \one_{\lambda+\lambda'}^*,
\end{equation}
for homogeneous elements $u,u'\in\DCU^+$ and $\lambda,\lambda'\in X$. Here $\widetilde{|u'|}\in \Z[\I] \hookrightarrow Y$ is defined in \eqref{eq:tnu}.

Define the linear map 
\begin{equation}\label{eq:beta}
    \beta:\O_\A[N]\#\O_\A[T]{\longrightarrow} \O_\A[B]
\end{equation}
by $\beta(f\otimes\one_\lambda^*)(u\one_\mu)=f(u)\delta_{\lambda,\mu}$, for $f\in\O_\A[N]$, $\lambda,\mu\in X$, and $u\in\U_\A^+$.

\begin{proposition}
    The map $\beta$ is an algebra isomorphism.
\end{proposition}

\begin{proof}
    It is clear that $\beta$ is an isomorphism of $\A$-modules. It remains to show that $\beta$ is an algebra homomorphism. It suffices to show that the equality \eqref{eq:com} is preserved after applying $\beta$.

    Take $u,u'\in\DCU^+$ to be homogeneous elements, and $\lambda,\lambda'\in X$. Let us write $$F=\beta\big(\alpha(u)\otimes \one_{\lambda}^*\big)\beta \big(\alpha(u')\otimes \one_{\lambda'}^*\big)\in \O_\A[B].$$
    For $x\in \U_\A^+$ and $\mu\in X$, it is clear that $F(x\one_\mu)=0$ unless $\mu=\lambda+\lambda'$. Let $r:\U^+\rightarrow \U^+\otimes \U^+$ be the homomorphism as in \cite[1.2.2]{Lus94}. Write $r(x)=\sum x_1\otimes x_2$ with $x_1$, $x_2$ homogeneous. By \cite[3.1.5]{Lus94}, one has
    $
    \Delta(x)=\sum x_1 K_{\widetilde{|x_2|}}\otimes x_2,
    $
    and hence 
    \[
    \Delta_{\lambda,\lambda'}(x\one_{\lambda+\lambda'})=\sum q^{\langle \widetilde{|x_2|},\lambda\rangle}x_1\one_\lambda\otimes x_2\one_{\lambda'}.
    \]
Therefore
\[
F(x\one_{\lambda+\lambda'})=\sum q^{\langle \widetilde{|x_2|},\lambda\rangle}(u,x_1)(u',x_2)\stackrel{(\heartsuit 1)}{=}\sum q^{\langle \widetilde{|u'|},\lambda\rangle}(u,x_1)(u',x_2)\stackrel{(\heartsuit 2)}{=}q^{\langle \widetilde{|u'|},\lambda\rangle}(uu',x).
\]
Here $(\heartsuit 1)$ follows from the fact that $(u',x_2)=0$ unless $|u'|=|x_2|$, and $(\heartsuit 2)$ follows from \cite[Proposition 1.2.3]{Lus94}. Hence we conclude that $F=\beta\bigl(q^{\langle \widetilde{|u'|},\lambda\rangle}\alpha(uu')\otimes \one_{\lambda+\lambda'}^*\bigr)$, as desired.
\end{proof}

\subsection{The algebra \texorpdfstring{$\O_\A[B_\theta]$}{O A[B theta]}} \label{sec:Btheta}

Let 
\[
\O_\A[N(w_0w_\bullet)]=\alpha\left(\DCU^+(w_0w_\bullet)\right).
\]
Let $\O_\A[{T}_\theta]$ be the $\A$-submodule of $\O_\A[T]$ spanned by $\{\one_\lambda^*\mid \lambda\in \breve{X}\}$. In particular $\O_\A[{T}_\theta]$ is an $\A$-subalgebra. By \eqref{eq:com}, algebras $\O_\A[N]$ and $\O_\A[T]$ are naturally identified with subalgebras of $\O_\A[B]$, via maps $f\mapsto f\otimes \one_0^*$ and $\one_\lambda^*\mapsto \alpha(1)\otimes \one_{\lambda}^*$, for $f\in \O_\A[N]$ and $\lambda\in X$. Let $\O_\A[B_\theta]$ be the subalgebra of $\O_\A[B]$ generated by $\O_\A[N(w_0w_\bullet)]$ and $\O_\A[T_\theta]$. Still by \eqref{eq:com}, one has the following $\A$-linear isomorphism 
\begin{equation}\label{eq:Btheta}
\O_\A[B_\theta]\cong \O_\A[N(w_0w_\bullet)]\otimes \O_\A[T_\theta].
\end{equation}


\begin{lemma}\label{le:oqb}
   One has
   $$
   \O_\A[B_\theta]=\{f\in\O_\A[B]\mid f(u\one_\lambda)=0\text{ if } \lambda\notin \breve{X} \text{ or }u\in \U_\A^+(w_0w_\bullet)\U_\A^+(w_\bullet)_>\}.
   $$
\end{lemma}

\begin{proof}
By \eqref{eq:Btheta}, an element of $\O_\A[B_\theta]$ is a finite sum of elements $g\otimes\one_\lambda^*$, where $g\in\O_\A[N(w_0w_\bullet)]$ and $\lambda\in\breve X$. Each such element vanishes on $u\one_\mu$ when $\mu\notin\breve X$. Moreover, the orthogonality \eqref{eq:orthoPBW} of the PBW basis implies that $g$ vanishes on $\U_\A^+(w_0w_\bullet)\U_\A^+(w_\bullet)_>$. 

Conversely, let $f\in\O_\A[B]$ satisfy the stated vanishing conditions. Using the isomorphism $\beta$ from \eqref{eq:beta}, write
\[
f=\sum_{\lambda\in X}\beta(g_\lambda\otimes\one_\lambda^*),
\qquad g_\lambda\in\O_\A[N],
\]
with only finitely many nonzero summands. The first vanishing condition gives $g_\lambda=0$ unless $\lambda\in\breve X$. Using \eqref{eq:dsum} and the orthogonality of PBW basis, the second condition forces that every term in the dual PBW expansion of remaining $g_\lambda$ belongs to $\O_\A[N(w_0 w_\bullet)]$. Hence $f\in\O_\A[B_\theta]$, proving the reverse inclusion.
\end{proof}

\begin{remark}
The subspace $\U^+(w_0w_\bullet)\U^+(w_\bullet)_>$ is compatible with the canonical basis, while the subspace $\U^+(w_0w_\bullet)$ is not. Dually, Kimura showed in \cite[Theorem 4.25]{Ki12} that $\O_\A[N(w_0w_\bullet)]$ is compatible with the dual canonical basis.
\end{remark}

Recall the algebra homomorphism $\pi:\O_\A[G]\rightarrow\O_\A[B]$ from \eqref{eq:pi}. Consider the composition
\begin{equation}\label{eq:Gamma}
    \Gamma:\O_\A[\ss]\hookrightarrow\O_\A[G]\overset{\pi}{\rightarrow} \O_\A[B].
\end{equation}
Write $\Gamma_q$ for its scalar extension to $\QQ$.

\begin{lemma}\label{lem:Gamma-saturation}
The map $\Gamma$ is an injective algebra homomorphism with its image contained in $\O_\A[B_\theta]$. Moreover,
\[
\Gamma\bigl(\O_\A[S]\bigr)
=\Gamma_q\bigl(\O_q[S]\bigr)\cap\O_\A[B_\theta]
\qquad\text{inside }\O_q[B_\theta].
\]
\end{lemma}

\begin{proof}
It is clear that $\Gamma$ is an algebra homomorphism. We first show that its image is contained in $\O_\A[B_\theta]$. For $f\in\O_\A[S]$, $u\in\U_\A^+$, and $\lambda\in X$, the right invariance gives
\[
f(u\one_\lambda)
=f(u\one_\lambda)\eta(\one_{\overline\lambda}).
\]
Since $\eta(\one_{\overline\lambda})=0$ unless $\lambda\in\breve X$, the first vanishing condition in Lemma~\ref{le:oqb} holds. For the second, it suffices to consider $u=u'u''$ with $u'\in\U_\A^+(w_0w_\bullet)$ and $u''\in\U_\A^+(w_\bullet)_>$ homogeneous of weight $\nu$. Then
\[
f(u'u''\one_\lambda)
=f\bigl(u'\one_{\lambda+\nu}(u''\one_{\overline\lambda})\bigr)
=f(u'\one_{\lambda+\nu})\eta(u''\one_{\overline\lambda})=0.
\]
Here $u''\one_{\overline\lambda}\in\dUiA$ and its counit is zero. Lemma~\ref{le:oqb} now gives the desired containment.

By Lemma~\ref{lem:idecomp}, every $v\in\dU_\A$ is a finite sum
\[
v=\sum_j x_j\one_{\lambda_j}y_j,
\qquad
x_j\in\U_\A^+(w_0w_\bullet),\quad
\lambda_j\in X,\quad y_j\in\dUiA.
\]
Thus any functional $f\in \O_{\A}[\ss]$ is determined by its restriction to $\dU^+_\A(w_0w_\bullet)$, since
\[
f(v)=\sum_j f(x_j\one_{\lambda_j})\eta(y_j).
\]
This proves the injectivity of $\Gamma$, and the same argument after extending scalars proves the injectivity of $\Gamma_q$.

Finally, suppose that $f\in\O_q[S]$ and $\Gamma_q(f)\in\O_\A[B_\theta]$. For the above expression of $v\in\dU_\A$, we have
\[
f(v)=\sum_j f(x_j\one_{\lambda_j})\eta(y_j)=\sum_j\Gamma_q(f)(x_j\one_{\lambda_j})\eta(y_j)\in\A.
\]
The functional $f$ already has finite support on $\dB$, because $f\in\O_q[G]$. Since $f(v)\in \A$ for $v\in\dU_\A$, $f\in\O_\A[G]$. The right invariance then forces $f\in\O_\A[S]$ as desired.
\end{proof}

Let $\mu\in\breve{X}^+$. Recall the elements $\psi_\mu\in\O_\A[S]$ from Proposition~\ref{prop:psimu}. Their defining property shows that
\begin{equation}\label{eq:Gpsi}
\Gamma(\psi_\mu)=\one_{-\mu}^*.
\end{equation}
In particular, $\Gamma(\psi_\mu)$ is invertible in $\O_\A[B_\theta]$.

\subsection{Ore condition and localization}
Set $\Psi_\A=\{\psi_\mu\mid\mu\in\breve{X}^+\}$.

\begin{lemma}
    The subset $\Psi_\A\subset \O_\A[S]$ is a left and right Ore set consisting of regular elements.
\end{lemma}

\begin{proof}
 Identify $\O_\A[S]$ with its image under $\Gamma$, and fix $s=\Gamma(\psi_\mu)$ and $f\in\Gamma(\O_\A[S])$. By \eqref{eq:com}, for $u \in \DCU^+$ homogeneous and $\lambda\in\breve X$, we have
\[
s\bigl(\alpha(u)\otimes\one_\lambda^*\bigr)s^{-1}
=q^{-\langle\widetilde{|u|},\mu\rangle}
\alpha(u)\otimes\one_\lambda^*.
\]
Write $f$ as a linear combination of terms $\alpha(u)\otimes\one_\lambda^*$. Only finitely many eigenvalues occur under conjugation by $s$. Let
\[
P(T)=\prod_{\rho}(T-\rho)=\sum_{j=0}^m a_jT^j\in\A[T]
\]
be the product over those eigenvalues. For $f=0$ the Ore identities are immediate, so we may assume $f\ne0$. Then $m\geq1$, $a_m=1$, $a_0\in\A^\times$, and $P(\operatorname{Ad}_s)(f)=0$. Multiplying on the right by $s^m$, we obtain
\[
\sum_{j=0}^m a_js^jfs^{m-j}=0.
\]
Solving for the last and first terms, respectively, yields
\[
s^mf=
\left(-\sum_{j=0}^{m-1}a_js^jfs^{m-j-1}\right)s,
\qquad
fs^m=s\left(-a_0^{-1}\sum_{j=1}^m a_js^{j-1}fs^{m-j}\right).
\]
Note that the expressions in parentheses belong to $\Gamma(\O_\A[S])$, and $s^m=\Gamma(\psi_{m\mu})$. Thus two Ore conditions are verified. Regularity follows because $\Gamma$ is injective and $s$ is invertible in $\O_\A[B_\theta]$.
\end{proof}

Consequently, $\Gamma$ extends to an injective algebra homomorphism
\[
\tG:\O_\A[S][\psi_\mu^{-1}\mid\mu\in\breve{X}^+]
\longrightarrow\O_\A[B_\theta].
\]
Note that the same argument applies after extending scalars to $\QQ$. Let us denote the resulting homomorphism by $\tG_q$.

\begin{proposition}\label{prop:integral-localization}
    The image of $\tG$ is $\O_\A[B_\theta]$. Hence $\tG$ induces an algebra isomorphism
    $$
    \tG:\O_\A[\ss][\psi_{\mu}^{-1}\mid \mu\in \breve{X}^+]\overset{\sim}{\longrightarrow} \O_\A[B_\theta].
    $$
\end{proposition}

\begin{proof}
By Lemma~\ref{lem:Gamma-saturation}, $\operatorname{im}\Gamma \subset \O_\A[B_\theta]$, and then using \eqref{eq:Gpsi}, we have $\operatorname{im}\tG \subset \O_\A[B_\theta]$. It suffices to show that $\tG$ is a surjective map onto $\O_\A[B_\theta]$.
We will first prove the surjectivity of $\tG_q$ over $\QQ$ and then descend to $\A$ using Lemma~\ref{lem:Gamma-saturation}.

By \cite[Lemma 2.2]{BS25}, $\breve X$ is generated by $\breve{X}^+$ as an abelian group. Hence \eqref{eq:Gpsi} implies that $\operatorname{im}\tG_q$ contains $\O_q[T_\theta]$. It remains to show that it contains the root generators of $\O_q[N(w_0 w_\bullet)]$.

Fix the reduced expression $w_0w_\bullet=s_{i_1}\cdots s_{i_L}$ and write $x_k=\alpha(\tE_{\beta_k})$. We argue by induction on $\het\beta_k$. Choose $\mu\in\breve{X}^+$ such that $\langle\beta_k^\vee,\mu\rangle>0$ for all $1\leq k\leq L$. To see that such a weight $\mu$ exists, let $\rho$ be the half sum of the positive roots and $\rho_\bullet$ be the half sum of positive roots in $\Phi_\bullet$, where $\Phi_\bullet$ is the root system associated to $\I_\bullet$. Since $2\rho\in Q\subseteq X$ and $\theta=-w_\bullet\tau$ on the root lattice, the weight
\[
\mu_0=(1-\theta)(2\rho)=4(\rho-\rho_\bullet)
\]
belongs to $\breve{X}^+$. Its pairings with simple coroots are zero on $\I_\bullet$ and strictly positive on $\I_\circ$. Every $\beta_k$ lies outside the $\Phi_\bullet$, so $\langle\beta_k^\vee,\mu_0\rangle>0$. Thus one may take $\mu=\mu_0$.

Normalize the spherical vector $w_\mu\in V(-w_0\mu)$ so that its lowest-weight component is $v_{-\mu}$. Then
\[
\psi_\mu(u)=v_{-\mu}^*(uw_\mu).
\]
The matrix coefficient
\[
\psi_\mu^{\beta_k}(u)=\psi_\mu(F_{\beta_k}u)
\]
belongs to $\O_q[S]$, since it is right invariant. We claim that
\begin{equation}\label{eq:localization-triangular}
\Gamma_q(\psi_\mu^{\beta_k})
=[\langle\beta_k^\vee,\mu\rangle]_{i_k}
x_k\one_{-\mu}^*+\mathcal G,
\end{equation}
where $\mathcal G$ belongs to the $\QQ$-subalgebra of $\O_q[B_\theta]$ generated by $\O_q[T_\theta]$ and all $x_t$ with $\het\beta_t<\het\beta_k$.

To prove the claim, let $u_\gamma\in\U^+(w_0w_\bullet)$ be homogeneous of weight $\gamma$. If the value
\[
v_{-\mu}^*(F_{\beta_k}u_\gamma\one_\lambda w_\mu),
\]
is nonzero, then we must have $\lambda+\gamma-\beta_k=-\mu$. Moreover, $\lambda+\mu\in Q_+$ whenever $\one_\lambda w_\mu\ne0$. Hence $\beta_k-\gamma\in Q_+$. Therefore, if we expand $\Gamma_q(\psi_\mu^{\beta_k})$ in the dual PBW basis, then each PBW monomial in this expansion has weight $\gamma\leq\beta_k$. 

If such a monomial contains a root vector of height at least $\het\beta_k$, then the height inequality forces the monomial to be the single root vector associated to $\beta_k$. Thus every other monomial involves only root vectors of smaller height. By \eqref{def:tE},  these other terms belong to the stated subalgebra.

For the leading coefficient, set $\lambda=-\mu$ and the commutator between $F_{\beta_k}, E_{\beta_k}$ gives
\[
F_{\beta_k}E_{\beta_k}v_{-\mu}
=[\langle\beta_k^\vee,\mu\rangle]_{i_k}v_{-\mu},
\]
since $F_{\beta_k}v_{-\mu}=0$. As $x_k(E_{\beta_k})=1$, the claim \eqref{eq:localization-triangular} is proved.

By induction, assume that $\mathcal G\in\operatorname{im}\tG_q$. The quantum integer in \eqref{eq:localization-triangular} is nonzero, hence invertible in $\QQ$. Therefore
\[
x_k=
[\langle\beta_k^\vee,\mu\rangle]_{i_k}^{-1}
\bigl(\Gamma_q(\psi_\mu^{\beta_k})-\mathcal G\bigr)
\one_\mu^*
\in\operatorname{im}\tG_q.
\]
This proves $\operatorname{im}\tG_q=\O_q[B_\theta]$.

Now let $b\in\O_\A[B_\theta]$. Using the isomorphism $\tG_q$ over $\Q(q)$ and the right Ore condition, there exist $f\in\O_q[S]$ and $\mu\in\breve{X}^+$ such that
\[
b=\Gamma_q(f)\Gamma(\psi_\mu)^{-1}.
\]
Consequently,
\[
\Gamma_q(f)=b\Gamma(\psi_\mu)\in\O_\A[B_\theta].
\]
By Lemma~\ref{lem:Gamma-saturation}, this implies that $f\in\O_\A[S]$. Hence $b=\tG(f\psi_\mu^{-1})$, proving the surjectivity of $\tG$. Injectivity follows from Lemma~\ref{lem:Gamma-saturation}. Hence, the proof is complete.
\end{proof}

\subsection{Relation with geometry}\label{sec:rwg}

We explain the classical counterpart of the constructions in the previous sections. Let $$B_\theta=B_+K/K\subset \ss$$ be the $B_+$-orbit through the base point $K\in S$. Let $T_\theta=T/(T\cap K)\subset \ss$. Let $N_{\pm}$ be the unipotent radicals of $B_\pm$, respectively. For $v\in W$, write $N(v)=N_+\cap vN_-v^{-1}$. It is known that $B_\theta$ is an open dense subset of $S$, and moreover by the local structure theorem \cite[Proposition 2.3]{DCP83}, the action map gives an isomorphism of varieties 
\begin{equation}\label{eq:NTB}
N(w_0w_\bullet)\times T_\theta\overset{\sim}{\longrightarrow} B_\theta.    
\end{equation}

For an $\A$-algebra $R$, we set $\O_R[B_\theta]=R\otimes_{\A}\O_\A[B_\theta] $. Let $\O_1[B_\theta]=\O_{\C_1}[B_\theta]$. Similarly, we have $\C$-algebras $\O_1[T_\theta]$ and $\O_1[N(w_0 \bw)]$, which are obtained from the $\A$-algebras $\O_\A[T_\theta]$ and $\O_\A[N(w_0w_\bullet)]$ by base change. Let us write  $\O[B_\theta]$, $\O[T_\theta]$, and $\O[N(w_0w_\bullet)]$ to denote the $\C$-algebras of the regular functions on varieties $B_\theta$, $T_\theta$, and $N(w_0w_\bullet)$, respectively.

\begin{lemma}\label{lem:O1Btheta}
There is an algebra isomorphism $\Phi_{B_\theta}:\O_1[B_\theta]\overset{\sim}{\longrightarrow}\O[B_\theta]$ such that the following diagram commutes:
\[
\begin{tikzcd}
\O_1[S] \arrow[r,"\Gamma_1"] \arrow[d,"\Phi_S"',"\sim"]
& \O_1[B_\theta] \arrow[d,"\Phi_{B_\theta}","\sim"'] \\
\O[S] \arrow[r,"\mathrm{res}"'] & \O[B_\theta].
\end{tikzcd}
\]
Here $\Gamma_1$ is the specialization of $\Gamma$ at $q=1$, and $\mathrm{res}$ is restriction to the open orbit $B_\theta\subset S$.
\end{lemma}

\begin{proof}
Recall the $\A$-linear isomorphism \eqref{eq:Btheta}. By taking the base change to $\C_1$ and noticing that $\O_1[B_\theta]$ is commutative, we have the $\C$-algebra isomorphism
\begin{equation}\label{eq:O1}
\O_1[B_\theta] \cong \O_1[N(w_0 \bw)] \otimes \O_1[T_\theta].
\end{equation}
In our setting, $X=\Hom(T,\C^\times)$. For $\lambda\in X$, the classical limit of $\one_\lambda^*$ is identified with the character $\lambda: T\rightarrow \C^\times$. Since $\breve{X}=\{\lambda-\theta(\lambda)|\lambda\in X\}$, we have 
\[
\lambda(tt')=\lambda(t),\qquad \forall t\in T, t'\in T^\theta=T\cap K, \lambda\in \breve{X}.
\]
In fact, $X^*(T/T^\theta)=(1-\theta)X=\breve X$: the homomorphism $t\mapsto t\theta(t)^{-1}$ has kernel $T^\theta$, and its pullback on character lattices is $1-\theta$. Hence, one has a natural algebra isomorphism $\O_1[T_\theta] \cong \O[T_\theta], \one_\lambda^*\mapsto \lambda$. Moreover, by \cite{DCKP93b}, one has a natural algebra isomorphism $\Phi_N:\O_1[N(w_0 \bw)] \cong \O[N(w_0 \bw)]$. Together with \eqref{eq:O1} and \eqref{eq:NTB}, these define the isomorphism $\Phi_{B_\theta}$.

We check compatibility with $\Gamma_1$. Set $\O_1[B]=\C_1\otimes_\A\O_\A[B]$, and denote by $\pi_1$ the specialization of $\pi$. Under the classical identification $\Phi_B:\O_1[B]\overset{\sim}{\longrightarrow}\O[B_+]$, the map $\pi_1$ is restriction from $G$ to $B_+$. Indeed, $\pi$ is restriction of functionals to the span of the elements $u\one_\lambda$, with $u\in\U_\A^+$; after specialization, this records precisely the positive-root operators and torus weights of the restricted matrix coefficients. Thus
\[
\Phi_B\circ\pi_1=\mathrm{res}_{B_+}\circ\Phi_G.
\]
Let $j_1:\O_1[B_\theta]\rightarrow\O_1[B]$ be the specialized inclusion. The PBW decomposition and the defining formula \eqref{eq:beta} give
\[
\Phi_B\bigl(j_1(g\otimes\one_\lambda^*)\bigr)(nt)
=\Phi_N(g)(n)\lambda(t),
\qquad n\in N(w_0\bw),\quad t\in T,\quad\lambda\in\breve X.
\]
Here we use the tensor product identification \eqref{eq:O1}. The right-hand side is also $\Phi_{B_\theta}(g\otimes\one_\lambda^*)(ntK)$ by construction. Since $j_1\circ\Gamma_1$ is the restriction of $\pi_1$ to $\O_1[S]$, compatibility of $\Phi_S$ with $\Phi_G$ now gives, for $f\in\O_1[S]$,
\[
\Phi_{B_\theta}(\Gamma_1(f))(ntK)
=\Phi_B(\pi_1(f))(nt)
=\Phi_S(f)(ntK).
\]
Every point of $B_\theta$ has this form by \eqref{eq:NTB}. The diagram therefore commutes.
\end{proof}

For $w\in W$, let $i:N(w)\rightarrow G/B_+$ be the map $n\mapsto nwB_+$. For $v\leq w$ in the Bruhat order, define
\[
N(w)_v=i^{-1}(B_-vB_+/B_+\cap B_+wB_+/B_+).
\]
Then $N(w)_v$ is a locally closed subset of $N(w)$.

Recall the open subset $B_\theta\subset \ss$ and the isomorphism $N(\ww)\times T_\theta\overset{\sim}{\rightarrow} B_\theta$. For $v\leq \ww$, define 
\[
B_{\theta,v}=B_\theta\cap \ss_v.
\]

\begin{lemma}\label{le:Bv}
For $v\leq \ww$, the isomorphism $N(\ww)\times T_\theta\overset{\sim}{\rightarrow}  B_\theta$ restricts to $N(\ww)_v\times T_\theta\overset{\sim}{\rightarrow} B_{\theta,v}$.
\end{lemma}

\begin{proof}
Let $(n,t)\in N(\ww)\times T_\theta$, and choose a lift $\widetilde t\in T$ of $t$. Its image in $B_\theta$ is $n\widetilde tK$, which is independent of the lift since $T\cap K$ acts trivially on this coset. Since $\theta(B_+)=\ww B_+ \ww$, there exists $b_+\in B_+$ such that $\theta(n\widetilde t)^{-1}=\ww b_+ \ww$. Then $\iota(n\widetilde tK)=n\widetilde t \ww b_+ \ww$.

If $n\in N(\ww)_v$, then $n \ww B_+ \in B_- v B_+$. This implies that $\iota(n\widetilde tK)\in B_- v B_+ \ww$ and hence $n\widetilde tK\in S_v$.
Conversely, if $(n,t)\in B_{\theta,v}$, then $\iota(n\widetilde tK)\in B_- v B_+ \ww$. This implies that $n\widetilde t \ww\in B_- v B_+$ and then $n\in N(\ww)_v$. This completes the proof.
\end{proof}

In particular, $B_{\theta,v}$ is nonempty if $v\leq \ww$, and we have the stratification $B_{\theta}=\bigsqcup_{v\leq \ww} B_{\theta,v}$.

\section{Specialization at roots of unity}\label{sec:root-of-unity}
Retain the same notations as in Section \ref{sec:bc}. Let us set 
\[
\cO_\epsilon[G]=\cO_{\mathbb{C}_\epsilon}[G],\qquad \cO_1[G]=\cO_{\mathbb{C}_1}[G];
\]
and 
\[
\cO_\epsilon[\ss]=\cO_{\mathbb{C}_\epsilon}[\ss],\qquad \cO_1[\ss]=\cO_{\mathbb{C}_1}[\ss].
\]

The goal of this section is to study ring-theoretic properties of $\cO_\epsilon[\ss]$. In particular, we show that $\cO_\epsilon[\ss]$ is a projective module of rank $\ell^{\text{dim } \ss}$ over a canonical central subalgebra, which is isomorphic to the ordinary coordinate algebra of the symmetric space.

\subsection{Frobenius center of \texorpdfstring{$\cO_\epsilon[\ss]$}{O epsilon[S]}}

In this subsection, we show that there is a central subalgebra of $\cO_\epsilon[\ss]$ which is isomorphic to $\cO_1[\ss]$.

The counit $\eta:{\dU}_\A\rightarrow \A$ clearly induces $\C$-algebra homomorphisms $\dU_1\rightarrow \C$ and $\dU_\epsilon\rightarrow \C$ via base changes. We still denote these maps by $\eta$.

\begin{lemma}\label{lem:counit}
The quantum Frobenius map is compatible with the counit $\eta$, i.e., we have
\[
\eta(\Fr(x))=\eta(x),\qquad \forall x\in \widehat{\U}_\epsilon.
\]
\end{lemma}

\begin{proof}
It suffices to check this identity for generators $E_i^{(m)}\one_\lambda, F_i^{(m)}\one_\lambda$, which is straightforward.
\end{proof}

\begin{theorem}\label{thm:Fr}
{\,}
\begin{itemize}
\item[(1)] The dual of the quantum Frobenius map induces an injective unital algebra homomorphism $\Fr^*: \O_1[G] \rightarrow \O_\epsilon[G]$. Moreover, $\Fr^*$ maps $\O_1[G]$ to a central subalgebra of $\O_\epsilon[G]$.

\item[(2)] The map $\Fr^*$ restricts to an algebra homomorphism $\Fr_S^*: \O_1[S]\rightarrow \O_\epsilon[S]$.
Hence, we have the following commutative diagram
\begin{equation}
        \begin{tikzcd}
            & \O_1[S] \arrow[r,hook,"\Fr_S^*"] \arrow[d,hook] & \O_\epsilon[S] \arrow[d,hook] \\ & \O_1[G] \arrow[r,hook,"\Fr^*"] & \O_\epsilon[G]
        \end{tikzcd}
\end{equation}
\end{itemize}
\end{theorem}

\begin{proof}
(1) Thanks to Lemma \ref{lem:fm}, the dual of $\Fr$ restricts to a map $\O_1[G]\rightarrow \O_\epsilon[G]$. It is an injection since $\Fr$ is surjective. By \cite[35.1.10]{Lus94}, $\Fr$ is compatible with comultiplications. Hence $\Fr^*$ is an algebra homomorphism. It is unital by Lemma \ref{lem:counit}. 

We show that the image of $\Fr^*$ is central. Let $f\in \O_1[G],g\in \O_\epsilon[G]$. Note that $\Fr^*(f) g=g\Fr^*(f)$ is implied by the following identity
\begin{align}\label{eq:FDel}
(\Fr \otimes \Id) \Delta= (\Fr \otimes \Id) \Delta^{\op}.
\end{align}
where $\Delta^{\op}=P\circ \Delta$ and $P:x\otimes y\mapsto y\otimes x$.

Hence, it suffices to show \eqref{eq:FDel}. Let $\epsilon_i=\epsilon^{d_i}$. Since both sides are algebra homomorphisms, it suffices to show \eqref{eq:FDel} on generators $E_i^{(m)}\one_\lambda, F_i^{(m)}\one_\lambda$. By \eqref{eq:comul}, we have
\begin{align*}
(\Fr \otimes \Id) \Delta\big(E_i^{(m)}\one_\lambda\big)&=\sum_{\lambda'+\lambda''=\lambda} \sum_{k=0}^m \epsilon_i^{k(m-k)} \Fr\big(E_i^{(m-k)}\one_{\lambda'}\big) \otimes E_i^{(k)}\one_{\lambda''}
\\
&=\sum_{\lambda'+\lambda''=\lambda} \sum_{k=0}^m \epsilon_i^{k(m-k)} \Fr\big(E_i^{(k)}\one_{\lambda''}\big) \otimes E_i^{(m-k)}\one_{\lambda'}
=(\Fr \otimes \Id) \Delta^{\op}\big(E_i^{(m)}\one_\lambda\big).
\end{align*}
The case for $F_i^{(m)}\one_\lambda$ is similar. Hence, we have proved \eqref{eq:FDel} and then the image of $\Fr^*$ is central.

(2) Let $f\in \O_1[G \slash K]$. Then for $y\in \dU^\imath_\epsilon,x\in \dU_\epsilon$, we have
\begin{align*}
\Fr^*(f)(xy)=f\big(\Fr(xy)\big)=f \big(\Fr(x) \Fr(y)\big).
\end{align*}
By Proposition~\ref{prop:dUiFr}, $\Fr(y)\in \dU^\imath_1$. Hence, by the choice of $f$ and Lemma \ref{lem:counit}, we have
\[
f \big(\Fr(x) \Fr(y)\big)=\eta\big(\Fr(y)\big) f \big( \Fr(x)\big)=\eta(y) \Fr^*(f)(x).
\]
This completes the proof.
\end{proof}

\begin{remark}
    The statement (1) of Theorem \ref{thm:Fr} was known in \cite{DCL94} for $G$ semisimply and simply connected. The formulation here using Lusztig's quantum coordinate algebras for any connected reductive groups is new. 
\end{remark}

Due to \cite[Theorem 4.8]{BS24}, there is a natural algebra isomorphism $\O_1[\ss] \cong \O[\ss]$. 
By Theorem~\ref{thm:Fr}, via $\Fr^*$, $\O[G]$ is identified with a central subalgebra of $\O_\epsilon[G]$, and hence $\O[\ss]$ is identified with a central subalgebra of $\O_\epsilon[\ss]$. 

Take $x\in S$. Let $\mathfrak{m}_x\subset \O[S]$ be the maximal ideal consisting of regular functions vanishing at $x$. Define 
\[
\O_\epsilon[S](x)=\O_\epsilon[S]/\mathfrak{m}_x\O_\epsilon[S]
\]
to be the \emph{fibre algebra} of $\O_\epsilon [S]$ at the point $x$.

\subsection{Restricting to the open chart}\label{sec:open-chart}

Recall the $\A$-algebra $\O_\A[B_\theta]$ from Section \ref{sec:Btheta}. By Section \ref{sec:rwg}, $\O_\A[B_\theta]$ is a quantization of the open chart $B_\theta$ of $S$. As usual, write $\O_\epsilon[B_\theta]=\C_\epsilon\otimes _{\A}\O_\A[B_\theta]$. In this subsection, we establish a canonical central subalgebra of $\O_\epsilon[B_\theta]$ and relate the fibre algebras of $\O_\epsilon[S]$ with that of $\O_\epsilon[B_\theta]$.

It follows from Proposition \ref{prop:integral-localization} that one has an $\A$-algebra embedding $\Gamma:\O_\A[S]\rightarrow \O_\A[B_\theta]$. Let 
\[
\Gamma_\epsilon:\O_\epsilon[\ss] \rightarrow \O_\epsilon[B_\theta],
\qquad
\Gamma_1:\O_1[\ss] \rightarrow \O_1[B_\theta]
\]
be the $\C$-algebra homomorphisms obtained by base change from $\Gamma$.

Both maps are injective. Indeed, let $\xi\in\{1,\epsilon\}$ and suppose that $f\in\ker\Gamma_\xi$. The integral basis \eqref{eq:iwabasis} remains a basis after specialization at $q=\xi$. Thus every element of $\dU_\xi$ is a finite sum of terms $x\one_\lambda y$, with $x\in\U^+_\xi(w_0w_\bullet)$ and $y\in\dU^\imath_\xi$. By right invariance,
\[
f(x\one_\lambda y)
=f(x\one_\lambda)\eta(y)
=\Gamma_\xi(f)(x\one_\lambda)\eta(y)=0.
\]
Hence $f=0$, proving the injectivity. Here $\U^+_\xi(w_0w_\bullet)$ denotes the corresponding specialization of $\U^+_\A(w_0w_\bullet)$.


\begin{proposition}\label{prop:cdGamma}
    There is an algebra embedding $\Fr^*_\theta:\O_1[B_\theta]\rightarrow \O_\epsilon[B_\theta]$, such that the diagram 
    \begin{equation}\label{cd:Gamma}
        \begin{tikzcd}
           & \O_\epsilon[\ss] \arrow[r,"\Gamma_\epsilon"] & \O_\epsilon[B_\theta] \\ & \O_1[\ss] \arrow[u,"\Fr^*"] \arrow[r,"\Gamma_1"] & \O_1[B_\theta] \arrow[u,"\Fr^*_\theta"']
        \end{tikzcd}
    \end{equation}
    commutes.
\end{proposition}

\begin{proof}
By Proposition~\ref{prop:Fr}, the map $\Fr:\dU_{\epsilon}\rightarrow \dU_1$ restricts to $\dU^{\geqslant 0}_\epsilon\rightarrow \dU_1^{\geqslant0}$; that is, one has the following commutative diagram
\begin{equation}
    \begin{tikzcd}
        & \dU_\epsilon  \arrow[d,"\Fr"] & \dU^{\geqslant 0}_\epsilon \arrow[d,"\Fr_+"] \arrow[l,hook] \\ & \dU_1 & \dU_1^{\geqslant 0} \arrow[l,hook]
    \end{tikzcd}
\end{equation}
Taking the dual of the above diagram, one obtains
\begin{equation}\label{cd:Gamma2}
        \begin{tikzcd}
           & \O_\epsilon[G] \arrow[r] & \O_\epsilon[B] \\ & \O_1[G] \arrow[u,"\Fr^*"] \arrow[r] & \O_1[B] \arrow[u,"\Fr^*_+"']
        \end{tikzcd}
    \end{equation}
    It is known that $\Fr$ is compatible with the braid group symmetries. Hence, by Proposition~\ref{prop:Fr},
    \[
    \Fr_+^*(\tE_{\beta_k})=\tE_{\beta_k}^\ell,
    \qquad 
    \Fr_+^*(\one_\lambda^*)=\one_{\ell\lambda}^*.
    \]
    Recall that $\O_\A[N(w_0 \bw)]\cong \DCU^+(w_0 \bw)$ is generated by $\tE_{\beta_k}, 1\leq k\leq L$. Hence, $\Fr_+^*$ restricts to a map $\Fr_\theta^*:\O_1[B_\theta]\rightarrow \O_\epsilon[B_\theta]$. Then the diagram \eqref{cd:Gamma} follows by restricting \eqref{cd:Gamma2} to $\O_1[\ss]$.  

This completes the proof.
\end{proof}


For $\mu\in\breve X^+$ and $\xi\in\{1,\epsilon\}$, denote by $\psi_{\mu,\xi}$ the specialization of $\psi_\mu$ at $q=\xi$. By \eqref{cd:Gamma} and \eqref{eq:Gpsi},
\[
\Gamma_\epsilon\bigl(\Fr^*(\psi_{\mu,1})\bigr)
=\Fr_\theta^*(\one_{-\mu}^*)
=\one_{-\ell\mu}^*
=\Gamma_\epsilon\bigl(\psi_{\mu,\epsilon}^{\ell}\bigr).
\]
The injectivity of $\Gamma_\epsilon$ therefore gives
\begin{equation}\label{eq:Fr-psi-power}
\Fr^*(\psi_{\mu,1})=\psi_{\mu,\epsilon}^{\ell}.
\end{equation}
In particular, $\psi_{\mu,\epsilon}^{\ell}$ is central in $\O_\epsilon[S]$.

By Proposition~\ref{prop:integral-localization}, base change gives the localized isomorphisms $\tG_1$ and $\tG_\epsilon$. The identity \eqref{eq:Fr-psi-power} allows $\Fr^*$ to extend to these localizations. Thus Proposition~\ref{prop:cdGamma} and Lemma~\ref{lem:O1Btheta} give the following commutative diagram
\begin{equation}\label{cd:tG}
        \begin{tikzcd}
           & \O_\epsilon[\ss][\psi_{\mu,\epsilon}^{-1}\mid \mu\in \breve{X}^+] \arrow[r,"\tG_\epsilon"] & \O_\epsilon[B_\theta] 
           \\ 
           & \O_1[\ss][\psi_{\mu,1}^{-1}\mid \mu\in \breve{X}^+] \arrow[u,"\Fr^*"] \arrow[r,"\tG_1"] & \O_1[B_\theta] \arrow[u,"\Fr^*_\theta"']\\
           & \O[\ss][\psi_{\mu,1}^{-1}\mid \mu\in \breve{X}^+] \arrow[r]\arrow[u,phantom, sloped, "\cong"] & \O[B_\theta] \arrow[u,u,phantom, sloped, "\cong"]
        \end{tikzcd}
    \end{equation}
Here the last horizontal map is induced by the comorphism of the open embedding $B_\theta\hookrightarrow \ss$.

Let us identify $\O[B_\theta]$ with a central subalgebra of $\O_\epsilon[B_\theta]$ via the map $\Fr^*_\theta$. For $x\in B_\theta$, write $\O_\epsilon[B_\theta](x)=\O_\epsilon[B_\theta]/\mathfrak{m}_x\O_\epsilon[B_\theta]$ where $\mathfrak{m}_x\subset \O[B_\theta]$ is the maximal ideal corresponding to the point $x$. Then $\breve{\mathfrak{m}}_x=\Gamma_1^{-1}(\mathfrak{m}_x)$ is the maximal ideal of $\O[\ss]$ corresponding to the point $x$, and by abuse of notation, we also use $\breve{\mathfrak{m}}_x$ to denote $\Fr^*(\breve{\mathfrak{m}}_x)$. We set $\O_\epsilon[\ss](x):=\O_\epsilon[\ss]/ \breve{\mathfrak{m}}_x \O_\epsilon[\ss]$.

\begin{proposition}\label{le:GB}
    For $x\in B_\theta$, one has $\O_\epsilon[\ss](x)\cong \O_\epsilon[B_\theta](x)$ as finite-dimensional algebras.
\end{proposition}

\begin{proof}
For $\mu\in\breve X^+$, let $\overline{\psi_{\mu,\epsilon}}$ be the image of $\psi_{\mu,\epsilon}$ in $\O_\epsilon[S](x)$. By \eqref{eq:Gpsi}, $\Gamma_1(\psi_{\mu,1})=\one_{-\mu}^*$ is invertible in $\O[B_\theta]$, so $\psi_{\mu,1}(x)\ne0$. The identity \eqref{eq:Fr-psi-power} gives
\[
\bigl(\overline{\psi_{\mu,\epsilon}}\bigr)^\ell
=\psi_{\mu,1}(x)\,1.
\]
Thus $\overline{\psi_{\mu,\epsilon}}$ is invertible, with inverse
$\psi_{\mu,1}(x)^{-1}\bigl(\overline{\psi_{\mu,\epsilon}}\bigr)^{\ell-1}$.
Consequently, localizing at these elements does not change the fibre at $x$. The isomorphism now follows from the compatible localization diagram \eqref{cd:tG}.
\end{proof}

\subsection{Translation and projectivity}

For $x\in \widehat{\U}_{\C_\epsilon}$, $y\in\widehat{\U}_{\C_1}$, $f\in \O_{\epsilon}[G]$, and $g\in \O_1[G]$, define 
\[
f \lhd_\epsilon x = \sum f_{(1)}(x) f_{(2)}\in \cO_\epsilon[G],
\quad 
g \lhd_1 y = \sum g_{(1)}(y) g_{(2)}\in \cO_1[G],
\]
where $\Delta(f)=\sum f_{(1)}\otimes f_{(2)}$ and $\Delta(g)=\sum g_{(1)}\otimes g_{(2)}$; cf. \cite[\S 1.4]{DCL94}. 
The elements $f \lhd_\epsilon x$ and $g \lhd_1 y$ are uniquely determined by the following property
\begin{align}\label{eq:fxy}
(f \lhd_\epsilon x)(x')=f(xx'), \qquad \text{ and }\qquad (g \lhd_1 y)(y')=g(yy'),
\end{align}
for any $x'\in\widehat{\U}_{\epsilon}$ and $y'\in\widehat{\U}_{1}$.

Let $i\in \I$. For any $a,b\in \C$, set 
\begin{align}
x_{i,\epsilon}^+(a)=\sum_{\lambda\in X}\sum_{n\geq0} a^n E_i^{(n)}\one_\lambda \in \widehat{\U}_{\epsilon},
\qquad
x_{i,\epsilon}^-(b)=\sum_{\lambda\in X}\sum_{n\geq0} b^n F_i^{(n)}\one_\lambda \in \widehat{\U}_{\epsilon};
\end{align}
and 
\begin{align}
x_{i,1}^+(a)=\sum_{\lambda\in X}\sum_{n\geq0} a^n e_i^{(n)}\one_\lambda \in \widehat{\U}_{1},
\qquad
x_{i,1}^-(b)=\sum_{\lambda\in X}\sum_{n\geq0} b^n f_i^{(n)}\one_\lambda \in \widehat{\U}_{1}.
\end{align}
It is clear from \eqref{eq:comul} that $\Delta(x^\pm_{i,1}(a))=x^\pm_{i,1}(a)\otimes x_{i,1}^{\pm}(a)$. Hence we may identify $x_{i,1}^\pm(a)$ with points of the reductive group $G$. We shall write $x_i^\pm(a)$ to denote $x_{i,1}^\pm(a)$ if there is no confusion.

\begin{proposition}\label{prop:tran}
Let $i\in \I, a\in \C$.
\begin{itemize}
\item[(1)] For any $f\in \O_{1}[G],g\in\cO_\epsilon[G]$, one has
\[
(\Fr^*(f)g)\lhd_\epsilon x^\pm_{i,\epsilon}(a)=(\Fr^*(f)\lhd_\epsilon x^\pm_{i,\epsilon}(a))(g\lhd_\epsilon x^\pm_{i,\epsilon}(a)).
\]

\item[(2)] For any $f\in \O_{\epsilon}[\ss]$ and $f'\in \O_{1}[\ss]$ , one has 
$$
f \lhd_\epsilon x_{i,\epsilon}^\pm(a) \in \O_{\epsilon}[\ss]\quad \text{ and }\quad f' \lhd_1 x_{i,1}^\pm(a) \in \O_{1}[\ss].
$$

\item[(3)] For any $f\in \O_{1}[G]$, one has 
\[
\Fr^*(f \lhd_1 x_{i,1}^\pm(a^\ell))= \Fr^*(f) \lhd_\epsilon x_{i,\epsilon}^\pm(a).
\]
\item[(4)] 
The map $f\mapsto \big(f \lhd_1 x^\pm_{i,1}(a)\big)$ is a $\C$-algebra automorphism on $\O_1[G]\cong \C[G]$, which coincides with the comorphism of the left translation map
\[
G\rightarrow G,\qquad g\mapsto x^\pm_{i,1}(a)g.
\]
\end{itemize}
\end{proposition}

\begin{proof}
We show (1). By Proposition~\ref{prop:Fr}, we have $\Fr(x^\pm_{i,\epsilon}(a))=x^\pm_{i,1}(a^\ell)$. Using \eqref{eq:comul} and Proposition~\ref{prop:Fr}, we have
\[
(\Fr\otimes \Id)\Delta(x^\pm_{i,\epsilon}(a))=x^\pm_{i,1}(a^\ell) \otimes x_{i,\epsilon}^\pm(a)=\Fr(x^\pm_{i,\epsilon}(a))\otimes x^\pm_{i,\epsilon}(a).
\]
Hence, we have
\begin{align*}
\big((\Fr^*(f)g)\lhd_\epsilon x_{i,\epsilon}^\pm(a)\big)(y)=&(f\otimes g)(\Fr\otimes \Id)\Delta(x_{i,\epsilon}^\pm(a)y)
\\
=& f\Big(\Fr\big(x_{i,\epsilon}^\pm(a)y_{(1)}\big)\Big) g\big(x_{i,\epsilon}^\pm(a)y_{(2)}\big),
\end{align*}
where $\Delta(y)= y_{(1)}\otimes y_{(2)}$. The desired identity follows.

(2) follows from \eqref{eq:fxy}. 

We show (3). For any $u\in \dU_\epsilon$, we have
\[
\Fr^*(f \lhd_1 x^\pm_{i,1}(a^\ell))(u)= f(x^\pm_{i,1}(a^\ell) \Fr(u) ) = f\big(\Fr(x_{i,\epsilon}^\pm(a)u)\big)
=\big(\Fr^*(f) \lhd_\epsilon x^\pm_{i,\epsilon}(a)\big)(u).
\]

We show (4). Specializing $q$ at $1$, we have $\Delta(x_i^\pm(a))=x_i^\pm(a)\otimes x_i^\pm(a)$. Hence, the map $f\mapsto \big(f \lhd_1 x^\pm_{i,1}(a)\big)$ is an algebra homomorphism. Moreover, this map is invertible, since $(f \lhd_1 x^\pm_{i,1}(a)) \lhd_1 x^\pm_{i,1}(-a)=f$ for any $f\in \O_1[G]$. By \eqref{eq:fxy}, the map $f\mapsto \big(f \lhd_1 x_{i,1}^\pm(a)\big)$ coincides with the comorphism of the left translation by $x^\pm_{i,1}(a)$.
\end{proof}

Note that the algebra embedding $\Fr^*:\O_1[\ss]\rightarrow \O_\epsilon[\ss]$ makes $\O_\epsilon[\ss]$ into an $\O_1[\ss]$-module. For a multiplicative subset $\cS$ of $\O_1[\ss]$, write $\cS^{-1}\O_1[\ss]$ to denote the localized algebra, and write $\cS^{-1}\O_\epsilon[\ss]$ to denote the localization of the $\O_1[\ss]$-module $\O_\epsilon[\ss]$. In particular, $\cS^{-1}\O_\epsilon[\ss]$ is a $\cS^{-1}\O_1[\ss]$-module.

For $\mu\in\breve{X}^+$, recall the element $\psi_\mu\in \O_\A[\ss]$ from Proposition \ref{prop:psimu}. By abuse of notation we still write $\psi_\mu$ to denote its image in $\O_1[\ss]$. Set $\Psi=\{\psi_\mu\mid \mu\in \breve{X}^+\}\subset \O_1[\ss]$. Then $\Psi$ is a multiplicative subset of $\O_1[\ss]$ thanks to Proposition \ref{prop:psimu}. 

\begin{lemma}\label{le:loPsi}
    The localization $\Psi^{-1}\O_\epsilon[\ss]$ is a free $\Psi^{-1}\O_1[\ss]$-module of rank $\ell^{\dim \ss}$.
\end{lemma}

\begin{proof}
By \eqref{eq:Fr-psi-power}, inverting $\Fr^*(\psi_{\mu,1})=\psi_{\mu,\epsilon}^{\ell}$ is equivalent to inverting $\psi_{\mu,\epsilon}$. Hence Proposition~\ref{prop:integral-localization} and the commutative diagram \eqref{cd:tG} identify
\[
\Psi^{-1}\O_\epsilon[S]\cong\O_\epsilon[B_\theta],
\qquad
\Psi^{-1}\O_1[S]\cong\O_1[B_\theta],
\]
with the module structure on the right induced by $\Fr_\theta^*$.

Set $m=\operatorname{rank}\breve X$, and choose a $\Z$-basis $\lambda_1,\ldots,\lambda_m$ of $\breve X$. Let $t_a=\one_{\lambda_a}^*\in\O_\epsilon[T_\theta]$, and let $x_j\in\O_\epsilon[N(w_0w_\bullet)]$ be the specialization of $\tE_{\beta_j}$, for $1\leq j\leq L$. By the PBW basis and \eqref{eq:Btheta}, the ordered monomials
\[
x_L^{a_L}\cdots x_1^{a_1}t_1^{b_1}\cdots t_m^{b_m},
\qquad (a_1,\ldots,a_L)\in\N^L,\quad (b_1,\ldots,b_m)\in\Z^m,
\]
form a $\C$-basis of $\O_\epsilon[B_\theta]$. By the formulas in the proof of Proposition~\ref{prop:cdGamma}, the Frobenius image is the central subalgebra
\[
Z_0:=\Fr_\theta^*(\O_1[B_\theta])
=\C[x_1^\ell,\ldots,x_L^\ell,t_1^{\pm\ell},\ldots,t_m^{\pm\ell}].
\]
The PBW basis also shows that the indicated generators of $Z_0$ are algebraically independent, with the torus generators inverted.

Dividing each exponent by $\ell$ with remainder, and using the centrality of $Z_0$, shows that
\[
\left\{
 x_L^{a_L}\cdots x_1^{a_1}t_1^{b_1}\cdots t_m^{b_m}
 \;\middle|\;
 0\leq a_j<\ell,\quad 0\leq b_a<\ell
\right\}
\]
generates $\O_\epsilon[B_\theta]$ over $Z_0$. To check linear independence, expand the coefficients of a $Z_0$-linear relation in the monomials in $x_j^\ell$ and $t_a^{\pm\ell}$. Distinct residue vectors give distinct ordered PBW monomials, so every coefficient vanishes. Thus the displayed set is a $Z_0$-basis, of cardinality $\ell^{L+m}$. Finally, \eqref{eq:NTB} gives $L+m=\dim S$, proving the assertion.
\end{proof}

Let $\cS$ be a multiplicative subset of $\O_1[\ss]$. By Proposition \ref{prop:tran} (4), for $i\in\I$ and $a\in\C$, the set $\cS\lhd_1 x_{i,1}^\pm(a)$ is again a multiplicative subset of $\O_1[\ss]$. 

\begin{lemma}\label{le:tran}
    Let $\cS$ be a multiplicative subset of $\O_1[\ss]$ such that $\cS^{-1}\O_\epsilon[\ss]$ is a free $\cS^{-1}\O_1[\ss]$-module of rank $N$. Then for $i\in\I$ and $a\in \C$, $\big(\cS\lhd_1 x_{i,1}^\pm(a)\big)^{-1}\O_\epsilon[\ss]$ is a free $\big(\cS\lhd_1 x_{i,1}^\pm(a)\big)^{-1}\O_1[\ss]$-module of rank $N$.
\end{lemma}

\begin{proof}
    Take $b\in\C$ such that $b^\ell=a$. Define
    \[
    T(f)=f\lhd_\epsilon x_{i,\epsilon}^\pm(b),\qquad
    \sigma(g)=g\lhd_1 x_{i,1}^\pm(a),
    \]
    for $f\in\O_\epsilon[G]$ and $g\in\O_1[\ss]$. By Proposition \ref{prop:tran} (2), $T$ preserves $\O_\epsilon[\ss]$. By Proposition \ref{prop:tran} (4), $\sigma$ is an algebra automorphism of $\O_1[\ss]$, with inverse $g\mapsto g\lhd_1 x_{i,1}^\pm(-a)$.

    We first show that $T-\Id$ is locally nilpotent on $\O_\epsilon[G]$. Every element of $\O_\epsilon[G]$ belongs to a finite-dimensional subcoalgebra $C$. The formula defining $\lhd_\epsilon$ shows that $C$ is stable under this action. Moreover,
    \[
    C=\bigoplus_{\lambda\in X}C_\lambda,\qquad
    C_\lambda=\{f\in C\mid f\lhd_\epsilon\one_\lambda=f\},
    \]
    with only finitely many nonzero summands. Indeed, every element of $\O_\epsilon[G]$ has finite support on $\dB$, and hence has only finitely many nonzero components $f\lhd_\epsilon\one_\lambda$. The weight relations in $\dU_\epsilon$ give
    \[
    (T-\Id)(C_\lambda)\subseteq
    \bigoplus_{n\geq 1}C_{\lambda\mp n\alpha_i},
    \]
    where the upper sign corresponds to $x_{i,\epsilon}^+(b)$ and the lower sign to $x_{i,\epsilon}^-(b)$. Since $C$ has only finitely many weights, this implies that $T-\Id$ is nilpotent on $C$. Thus $T-\Id$ is locally nilpotent on $\O_\epsilon[G]$, and also on its $T$-stable subspace $\O_\epsilon[\ss]$. Consequently, $T$ restricts to a linear isomorphism of $\O_\epsilon[\ss]$, with inverse
    \[
    T^{-1}(f)=\sum_{r\geq 0}(-1)^r (T-\Id)^r(f),
    \qquad f\in\O_\epsilon[\ss],
    \]
    where the sum is finite for each $f$.

    By Proposition \ref{prop:tran} (1) and (3), we have
    \[
    T(\Fr^*(g)f)=\Fr^*(\sigma(g))T(f),
    \qquad g\in\O_1[\ss],\ f\in\O_\epsilon[\ss].
    \]
    Applying this identity to $\sigma^{-1}(g)$ and $T^{-1}(f)$ gives
    \[
    T^{-1}(\Fr^*(g)f)=\Fr^*(\sigma^{-1}(g))T^{-1}(f).
    \]
    Define the linear map
    \[
    t:\cS^{-1}\O_\epsilon[\ss]\longrightarrow
    \sigma(\cS)^{-1}\O_\epsilon[\ss],\qquad
    \frac{f}{s}\mapsto\frac{T(f)}{\sigma(s)}.
    \]
    To check that $t$ is well-defined, suppose that $f/s=f'/s'$. Then there exists $u\in\cS$ such that
    \[
    \Fr^*(u)\big(\Fr^*(s')f-\Fr^*(s)f'\big)=0.
    \]
    Applying $T$ and using the preceding identity yields
    \[
    \Fr^*(\sigma(u))\big(\Fr^*(\sigma(s'))T(f)
    -\Fr^*(\sigma(s))T(f')\big)=0,
    \]
    as desired. The identity for $T^{-1}$ shows in the same way that
    \[
    \sigma(\cS)^{-1}\O_\epsilon[\ss]\longrightarrow
    \cS^{-1}\O_\epsilon[\ss],\qquad
    \frac{g}{\sigma(s)}\mapsto\frac{T^{-1}(g)}{s}
    \]
    is well-defined. It is an inverse of $t$.

    For $f\in\cS^{-1}\O_\epsilon[\ss]$ and $g/s\in\cS^{-1}\O_1[\ss]$, it is direct to check that
    \[
    t\Big(\frac{g}{s}\cdot f\Big)=
    \frac{\sigma(g)}{\sigma(s)}\cdot t(f).
    \]
    Since $g/s\mapsto\sigma(g)/\sigma(s)$ is an algebra isomorphism from $\cS^{-1}\O_1[\ss]$ to $\sigma(\cS)^{-1}\O_1[\ss]$, a $\cS^{-1}\O_1[\ss]$-basis $\{f_1,\cdots,f_N\}$ is carried by $t$ to a $\sigma(\cS)^{-1}\O_1[\ss]$-basis $\{t(f_1),\cdots,t(f_N)\}$. Noting that $\sigma(\cS)=\cS\lhd_1 x_{i,1}^\pm(a)$ completes the proof.
\end{proof}

\begin{theorem}\label{thm:projectivity}
    $\O_\epsilon[\ss]$ is a finitely generated projective module over $\O_1[\ss]$ of rank $\ell^{\dim \ss}$.
\end{theorem}

\begin{proof}
Identify $\O_1[S]$ with the coordinate algebra $\C[S]$. Let $G'=[G,G]$ be the derived group of $G$, which is generated by $x_{i,1}^\pm(a)$, for $a\in\C$ and $i\in\I$. We use the left action of $G'$ on $\C[S]$ given by
\[
(g\cdot u)(z)=u(g^{-1}z),\qquad g\in G',\quad u\in\C[S],\quad z\in S.
\]
For $g=x_{i_1}^{\pm}(a_1)\cdots x_{i_k}^{\pm}(a_k)$, Proposition~\ref{prop:tran} (4) gives
\[
g\cdot u
=\Big(\cdots\big((u\lhd_1 x_{i_k,1}^{\pm}(-a_k))
\lhd_1 x_{i_{k-1},1}^{\pm}(-a_{k-1})\big)\cdots\Big)
\lhd_1 x_{i_1,1}^{\pm}(-a_1).
\]
In particular, this expression is independent of the chosen factorization of $g$. Set $g\cdot\Psi=\{g\cdot u\mid u\in\Psi\}$. By Lemma~\ref{le:loPsi} and repeated application of Lemma~\ref{le:tran}, the module $(g\cdot\Psi)^{-1}\O_\epsilon[S]$ is free over $(g\cdot\Psi)^{-1}\C[S]$ of rank $\ell^{\dim S}$.

The localization $\Psi^{-1}\C[S]$ is the coordinate algebra of the open subset $B_\theta$. Since $(g\cdot u)(z)=u(g^{-1}z)$, the localization $(g\cdot\Psi)^{-1}\C[S]$ is the coordinate algebra of $gB_\theta$. Thus the quasi-coherent sheaf associated with the $\C[S]$-module $\O_\epsilon[S]$ is free of rank $\ell^{\dim S}$ on each $gB_\theta$. Since $G'T=G$ and $T\subseteq B_+$, we have
\[
G=\bigcup_{g\in G'}gB_+,
\qquad
S=\bigcup_{g\in G'}gB_\theta.
\]
By quasi-compactness, this open cover has a finite subcover. Hence the associated sheaf is locally free of finite rank on the affine variety $S$, and $\O_\epsilon[S]$ is a finitely generated projective $\C[S]$-module of rank $\ell^{\dim S}$.
\end{proof}

\subsection{Poisson orders}\label{sec:poisson-orders}

We briefly recall the definition and the basic fibrewise consequence of Brown--Gordon \cite[Definition 2.1 and Theorem 4.2]{BG03}.

\begin{definition}
Let $A$ be an affine $\C$-algebra and let $Z_0\subseteq Z(A)$ be a central subalgebra such that $A$ is finitely generated as a $Z_0$-module. A \emph{Poisson $Z_0$-order} is given by a $\C$-linear map
\[
\partial:Z_0\longrightarrow \operatorname{Der}_{\C}(A),\qquad z\longmapsto \partial_z,
\]
such that $\partial_z(Z_0)\subseteq Z_0$ for every $z\in Z_0$ and
\[
\{z,z'\}:=\partial_z(z')
\]
defines a Poisson algebra structure on $Z_0$. In this situation, $\operatorname{MaxSpec}(Z_0)$ is a Poisson variety.
\end{definition}

For $x\in\operatorname{MaxSpec}(Z_0)$, let $\mathfrak m_x\subset Z_0$ be the corresponding maximal ideal and define the fibre algebra
\[
A(x):=A/\mathfrak m_xA.
\]

The \emph{Poisson core} of $\mathfrak m_x$ is the largest Poisson ideal of $Z_0$ contained in $\mathfrak m_x$. Two points lie in the same \emph{symplectic core} if their maximal ideals have the same Poisson core \cite[Section 3.3]{BG03}.

\begin{proposition}[\text{\cite[Theorem 4.2]{BG03}}]\label{thm:BG-fibres}
If $x,y\in\operatorname{MaxSpec}(Z_0)$ belong to the same symplectic core, then
\[
A(x)\cong A(y)
\]
as $\C$-algebras. In particular, the fibre algebras are isomorphic when $x$ and $y$ lie in the same symplectic leaf.
\end{proposition}

Set $R=\C[q,q^{-1}]_{(q-\epsilon)}$ and $h=q-\epsilon$, and write
\[
\mathcal A_G=R\otimes_\A\O_\A[G],\qquad
\mathcal A_S=R\otimes_\A\O_\A[\ss].
\]
We use $\pi_\epsilon$ for the specialization maps from $\A_G$ and $\A_S$ to $\O_\epsilon[G]$ and $\O_\epsilon[\ss]$. Choose a $\C$-linear map
\[
\jmath:\O_1[G]\longrightarrow\mathcal A_G
\]
such that $\pi_\epsilon\circ\jmath=\Fr^*$, $\jmath(1)=1$, and
$\jmath(\O_1[\ss])\subseteq\mathcal A_S$. It is clear that such a choice is possible, since one can first choose lifts in $\mathcal A_S$ of the Frobenius images of a $\C$-basis of $\O_1[\ss]$ containing $1$, and then extend this basis to one of $\O_1[G]$ and choose lifts in $\mathcal A_G$.

For $f\in\O_1[G]$ and $g\in\O_\epsilon[G]$, set $\tilde f=\jmath(f)$ and choose a lift $\tilde g\in\mathcal A_G$ of $g$. Define
\begin{equation}\label{eq:partial}
\partial_f(g)=\frac{\epsilon}{\ell^2}\,
\pi_\epsilon\left(\frac{[\tilde f,\tilde g]}{q-\epsilon}\right).
\end{equation}

For $f,g\in \O_1[G]$, take $\tilde f,\tilde g$ to be the lifts in $\A_G$. The following equality is essentially proved in \cite[Section 8]{DCL94}:
For $f,g\in\O_1[G]$, let $\tilde f,\tilde g\in\mathcal A_G$
be lifts of $\Fr^*(f),\Fr^*(g)$, respectively. Then
\begin{equation}\label{eq:poisson-Frobenius}
    \frac{\epsilon}{\ell^2}\,
    \pi_\epsilon\left(\frac{[\tilde f,\tilde g]}{h}\right)
    =\Fr^*(\{f,g\}_1).
\end{equation}
Note that our normalization in \eqref{eq:partial} makes the induced bracket agree with the normalization in \eqref{eq:Poisson-O1G}. We also remark that although the arguments in \cite[Section 8]{DCL94} are for semisimple simply connected group $G$, its generalization to general reductive groups is direct and we shall omit.

\begin{proposition}\label{prop:poisson-order}
$(\O_\epsilon[\ss],\O_1[\ss],\partial)$ is a Poisson order, where $\O_1[\ss]$ is identified with its Frobenius image. Its induced Poisson bracket is the standard bracket on $\O_1[\ss]$.
\end{proposition}

\begin{proof}
    By Theorem \ref{thm:Fr}, $\Fr^*(\O_1[\ss])$ is a central subalgebra of $\O_\epsilon[\ss]$. By Theorem \ref{thm:projectivity}, $\O_\epsilon[\ss]$ is finite projective over this subalgebra. Since $\O_1[\ss]\cong\O[\ss]$ is affine, $\O_\epsilon[\ss]$ is also affine.

    We first check that \eqref{eq:partial} defines a $\C$-linear map from $\O_1[G]$ to $\operatorname{Der}_\C(\O_\epsilon[G])$. The algebra $\mathcal A_G$ is free over $R$, so multiplication by $h$ is injective. Since $\pi_\epsilon(\tilde f)=\Fr^*(f)$ is central, $[\tilde f,\tilde g]\in h\mathcal A_G$, and the quotient in \eqref{eq:partial} is well-defined. Replacing $\tilde g$ by $\tilde g+hv$ changes its specialization by
    \[
    \frac{\epsilon}{\ell^2}\pi_\epsilon([\tilde f,v])
    =\frac{\epsilon}{\ell^2}[\Fr^*(f),\pi_\epsilon(v)]=0.
    \]
    Hence $\partial_f(g)$ is independent of the lift of $g$. Linearity in $f$ follows from the linearity of $\jmath$, and the identity
    \[
    [\tilde f,\tilde g_1\tilde g_2]
    =[\tilde f,\tilde g_1]\tilde g_2
    +\tilde g_1[\tilde f,\tilde g_2]
    \]
    gives $\partial_f(g_1g_2)=\partial_f(g_1)g_2+g_1\partial_f(g_2)$.

    We next show that $\partial_f$ preserves $\O_\epsilon[\ss]$ for $f\in\O_1[\ss]$. The inclusion $\mathcal A_S\subseteq\mathcal A_G$ satisfies
    \begin{equation}\label{eq:poisson-saturation}
    \mathcal A_S\cap h\mathcal A_G=h\mathcal A_S.
    \end{equation}
    Indeed, if $hv\in\mathcal A_S$ for $v\in\mathcal A_G$, then right invariance gives
    \[
    h\big(v(xy)-\eta_R(y)v(x)\big)=0,
    \qquad x\in\dU_R,\ y\in\dU_R^\imath.
    \]
    Canceling $h$ in $R$ shows that $v$ is right invariant, and hence $v\in\mathcal A_S$ by \eqref{eq:quantum-symmetric-space}.
    For $f\in\O_1[\ss]$ and $g\in\O_\epsilon[\ss]$, we may choose both $\tilde f$ and $\tilde g$ in $\mathcal A_S$. Their commutator belongs to $\mathcal A_S\cap h\mathcal A_G$, so \eqref{eq:poisson-saturation} implies
    \[
    \frac{[\tilde f,\tilde g]}{h}\in\mathcal A_S.
    \]
    Thus $\partial_f(g)\in\O_\epsilon[\ss]$.

    Finally, for $f,g\in\O_1[\ss]$, \eqref{eq:poisson-Frobenius} gives
    \[
    \partial_f(\Fr^*(g))=\Fr^*(\{f,g\}_1).
    \]
    Since $\O_1[\ss]$ is a Poisson subalgebra of $\O_1[G]$, the right-hand side belongs to $\Fr^*(\O_1[\ss])$. Thus the Frobenius subalgebra is stable under these derivations, and the induced bracket is precisely its standard Poisson bracket. This proves all the required properties.
\end{proof}

The choice of $\jmath$ is part of the construction of $\partial$. More precisely, replacing $\jmath(f)$ by $\jmath(f)+hv_f$ changes $\partial_f$ by the inner derivation
\[
g\longmapsto\frac{\epsilon}{\ell^2}[\pi_\epsilon(v_f),g].
\]
For two choices respecting $\mathcal A_S$, \eqref{eq:poisson-saturation} gives $v_f\in\mathcal A_S$ when $f\in\O_1[\ss]$. In particular, the induced bracket on the Frobenius subalgebra is independent of this choice; cf. \cite[Section 2.2]{BG03}.



\begin{theorem}\label{thm:fiberiso}
    Suppose $x,x'\in \ss_u$ for some $u\in W$. Then one has $\O_\epsilon[\ss](x)\cong \O_\epsilon[\ss](x')$ as finite-dimensional algebras.
\end{theorem}

\begin{proof}
    For $r\in T$, set $\tilde r=\sum_{\lambda\in X}\lambda(r)\one_\lambda\in\widehat{\U}_{\epsilon}$ and define $l_r:\O_\epsilon[G]\rightarrow\O_\epsilon[G]$ by
    \[
    l_r(f)(v)=f(\tilde r v),\qquad v\in\dU_{\epsilon}.
    \]
    Since $\lambda\mapsto\lambda(r)$ is a character of $X$, we have $\Delta(\tilde r)=\tilde r\otimes\tilde r$. It follows that $l_r$ is an algebra automorphism with inverse $l_{r^{-1}}$. It preserves $\cO_\epsilon[S]$, since left multiplication commutes with the right-invariance condition in \eqref{eq:quantum-symmetric-space}.

    For $t\in T$, write $l_t^{\mathrm{cl}}$ for the classical pullback $g(y)\mapsto g(ty)$ on $\O_1[\ss]$. The formula for $\Fr$ on weight idempotents gives $\Fr(\tilde r)=\widetilde{r^\ell}$, and hence
    \begin{equation}\label{eq:ltFr}
    l_r\circ\Fr^*=\Fr^*\circ l_{r^\ell}^{\mathrm{cl}}.
    \end{equation}
    Thus
    \[
    l_r\big(\Fr^*(\mathfrak m_x)\O_\epsilon[\ss]\big)
    =\Fr^*(\mathfrak m_{r^{-\ell}x})\O_\epsilon[\ss],
    \]
    which induces an isomorphism
    $\O_\epsilon[\ss](x)\cong\O_\epsilon[\ss](r^{-\ell}x)$.
    The power map $T\rightarrow T$, $r\mapsto r^\ell$, is surjective. Therefore the fibre algebras at $x$ and $tx$ are isomorphic for every $t\in T$.

    By Proposition \ref{prop:dimSu}, $\ss_u$ is a $T$-orbit of a symplectic leaf. Thus there exists $t\in T$ such that $tx$ and $x'$ lie in the same symplectic leaf. Proposition \ref{prop:poisson-order} and Proposition \ref{thm:BG-fibres} give
    \[
    \O_\epsilon[\ss](x)\cong\O_\epsilon[\ss](tx)
    \cong\O_\epsilon[\ss](x'),
    \]
    as desired.
\end{proof}

\subsection{Maximal orders}
\label{sec:maximal-order}

In this subsection, we show that $\O_\epsilon[S]$ is a maximal order in its division ring of fractions. Recall that an order $A$ in a division ring $D$ is a maximal order
if every intermediate ring $A\subseteq B\subseteq D$ satisfying
$aBb\subseteq A$ for some nonzero $a,b\in D$ equals $A$.

We begin with a criterion for extending the maximal order
property from an open subset. 

\begin{lemma}\label{lem:maximal-order-chart}
Let $R$ be a normal affine complex domain, and let $A$ be a
domain containing $R$ as a central subalgebra and finite projective as an
$R$-module. Set $X=\operatorname{Spec}R$ and
$D=\operatorname{Frac}A$. Suppose that
\[
X=\bigsqcup_{\alpha\in I}X_\alpha
\]
is a finite stratification by irreducible locally closed
subvarieties such that the algebras $A/\mathfrak m_xA$ are
isomorphic as $x$ varies over the closed points of each
$X_\alpha$.

Let $\Sigma\subset R\setminus\{0\}$ be a multiplicative subset
such that $\operatorname{Spec}R_\Sigma\to X$ is an open
immersion, with image $U$. If $U$ meets every stratum and
$A_\Sigma$ is a maximal order in $D$, then $A$ is a maximal
order in $D$.
\end{lemma}

\begin{proof}
Put $F=\operatorname{Frac}R$ and
$n=\operatorname{rank}_R A$. The algebra $F\otimes_R A$ is
a finite-dimensional domain over $F$, hence a division
algebra, and therefore coincides with $D$. If $\dim R=0$,
then $A=D$ and the assertion is immediate. We henceforth
assume that $\dim R>0$.

Since $A$ is finite projective over $R$, the local matrix
traces of left multiplication define an $R$-linear map
\[
t:A\longrightarrow R,\qquad
t(a)=\operatorname{Tr}_R(L_a).
\]
Consider the homomorphism
\[
\eta:A\longrightarrow A^\vee
:=\operatorname{Hom}_R(A,R),
\qquad
\eta(a)(b)=t(ab).
\]
Its scalar extension to $F$ is an isomorphism. Indeed, for
$0\ne d\in D$,
\[
t_F(dd^{-1})=t_F(1)=n\ne0,
\]
so the trace pairing on $D$ is nondegenerate.

Define the non-isomorphism locus of $\eta$ by
\[
\mathcal D
:=
\left\{
\mathfrak p\in\operatorname{Spec}R
\;\middle|\;
\eta_{\mathfrak p}:A_{\mathfrak p}
\longrightarrow (A^\vee)_{\mathfrak p}
\text{ is not an isomorphism}
\right\}.
\]
Since $A$ is finite projective, there is a canonical
identification
\[
(A^\vee)_{\mathfrak p}
\cong
\operatorname{Hom}_{R_{\mathfrak p}}
(A_{\mathfrak p},R_{\mathfrak p}).
\]
The source and target of $\eta$ are finite locally free
of the same rank. Consequently,
\[
\mathcal D
=
\operatorname{Supp}_R(\operatorname{coker}\eta),
\]
and in particular $\mathcal D$ is a closed subset of $\operatorname{Spec}R$. Since the base change of $\eta$ to $F$ is an isomorphism,
$\mathcal D$ does not contain the generic point and is
therefore proper.

Since matrix trace commutes with base change, the
specialization of $\eta$ at a closed point $x$ is the
regular trace pairing of $A/\mathfrak m_xA$.

For a finite-dimensional complex algebra $H$, the regular
trace pairing is nondegenerate if and only if $H$ is
semisimple. Indeed, if $J=\operatorname{rad}H$, then
$jh$ is nilpotent for all $j\in J$ and $h\in H$, whence
$\operatorname{Tr}_{\mathbb C}(L_{jh})=0$. Conversely,
on each simple summand $\operatorname{Mat}_d(\mathbb C)$
the regular trace pairing is $d$ times the ordinary
matrix trace pairing, and is therefore nondegenerate.
Consequently,
\[
\mathcal D(\mathbb C)
=
\{x\in X(\mathbb C):
  A/\mathfrak m_xA\text{ is not semisimple}\}.
\]
By the hypothesis on the fibres, $\mathcal D$ is a union
of strata. Since the stratification is finite and its
strata are irreducible, every irreducible component of
$\mathcal D$ is the closure of a stratum. Hence every
irreducible component of $\mathcal D$ meets $U$.

Let $A\subseteq B\subseteq D$ be an equivalent overorder;
thus $aBb\subseteq A$ for some nonzero $a,b\in D$.
The inclusion
\[
B\subseteq a^{-1}Ab^{-1}
\]
implies that $B$ is finite over $R$, since $R$ is central
and Noetherian. Localizing at $\Sigma$ and using the
maximality of $A_\Sigma$, we obtain
$B_\Sigma=A_\Sigma$. In particular,
$B_{\mathfrak p}=A_{\mathfrak p}$ for every height-one
prime $\mathfrak p$ belonging to $U$.

Now let $\mathfrak p$ be a height-one prime outside $U$.
Then $V(\mathfrak p)\subseteq X\setminus U$.
If $\mathfrak p$ belonged to $\mathcal D$, the divisor
$V(\mathfrak p)$ would be an irreducible component of
the proper closed subset $\mathcal D$, contradicting
the preceding paragraph. Thus $\eta_{\mathfrak p}$
is an isomorphism.

Note that $R_{\mathfrak p}$ is a discrete
valuation ring. The module $B_{\mathfrak p}$ is finite
and torsion-free over $R_{\mathfrak p}$, and hence is
free. Moreover, it spans $D$
over $F$. For $c\in B_{\mathfrak p}$ and
$x\in A_{\mathfrak p}$, left multiplication by $cx$
preserves $B_{\mathfrak p}$. Computing its trace with
respect to an $R_{\mathfrak p}$-basis gives
\[
t_F(cx)\in R_{\mathfrak p}.
\]
Thus $c$ belongs to the trace-dual lattice
\[
A_{\mathfrak p}^{\#}
:=
\{d\in D:
  t_F(dA_{\mathfrak p})\subseteq R_{\mathfrak p}\}.
\]
Since $\eta_{\mathfrak p}$ is an isomorphism and
$\eta_F$ is nondegenerate,
$A_{\mathfrak p}^{\#}=A_{\mathfrak p}$.
It follows that $B_{\mathfrak p}=A_{\mathfrak p}$ for
every height-one prime $\mathfrak p$.

Finally, $A$ is reflexive over $R$, being finite
projective. By \cite[Tag~0AVB]{Stacks},
\[
A=\bigcap_{\operatorname{ht}\mathfrak p=1}A_{\mathfrak p}
\qquad\text{inside }D.
\]
Every element of $B$ belongs to each
$B_{\mathfrak p}=A_{\mathfrak p}$, and hence belongs
to $A$. Therefore $B=A$, proving the assertion.
\end{proof}

\begin{lemma}\label{lem:big-cell-maximal-order}
The algebra $\O_\epsilon[B_\theta]$ is a Noetherian
domain and a maximal order in its division ring of
fractions.
\end{lemma}

\begin{proof}
Choose a basis of $\breve X$ and a reduced expression
of $w^\bullet$. Set $r=\operatorname{rank}\breve X$,
and let $L$ be the length of $w^\bullet$.
The PBW basis from Section~\ref{sec:PBW}, the
Levendorskii--Soibelman straightening relations, and
the cross-relations \eqref{eq:com} give a presentation
\[
\O_\epsilon[B_\theta]
\cong
\mathbb C[t_1^{\pm1},\ldots,t_r^{\pm1}]
[x_1;\sigma_1,\delta_1]\cdots[x_L;\sigma_L,\delta_L].
\]
Here $x_j$ is the specialization of the root vector
$\tE_{\beta_j}$, each $\sigma_j$ is an automorphism
acting diagonally on the preceding generators, and
$\delta_j$ is a $\sigma_j$-derivation. The lower terms
in the straightening relation for $x_jx_i$, $i<j$,
involve only the intermediate root vectors. The
integral PBW basis ensures that the presentation
remains valid after specialization.

The Laurent polynomial coefficient ring is a
Noetherian maximal order. By
\cite[Lemma~2.2 and Theorem~2.4]{HMU16}, Ore extension
by an automorphism and a skew derivation preserves
Noetherianity and the maximal order property; here
one takes the invertible ideal in that reference to
be the whole coefficient ring. Applying these results
successively proves maximality. The domain property
follows at each step from the injectivity of
$\sigma_j$.
\end{proof}

\begin{theorem}\label{thm:maximal-order}
Under the standing assumptions on $\ell$, the algebra
$\O_\epsilon[S]$ is a Noetherian domain and a maximal
order in its division ring of fractions.
\end{theorem}

\begin{proof}
Set $A=\O_\epsilon[S]$ and identify $R=\O[S]$ with its
Frobenius image in $Z(A)$ via Theorem~\ref{thm:Fr}.
Since $S=G/K$ is smooth and irreducible, $R$ is a normal
affine domain. By Theorem~\ref{thm:projectivity}, $A$
is finite projective over $R$, and in particular is
Noetherian. By the injectivity of $\Gamma_\epsilon$ proved in Section~\ref{sec:open-chart}, we have an embedding
\[
A\hookrightarrow\O_\epsilon[B_\theta].
\]
Together with Lemma~\ref{lem:big-cell-maximal-order}, this implies that
$A$ is a domain.

By Proposition~\ref{prop:integral-localization} and
Lemma~\ref{le:loPsi}, the central localization
$\Psi^{-1}A$ is isomorphic to $\O_\epsilon[B_\theta]$,
and the corresponding open subset of
$\operatorname{Spec}R$ is $B_\theta$.
Indeed, the Frobenius image of the classical
$\psi_\mu$ is $\psi_{\mu,\epsilon}^{\ell}$, and
inverting this central power is equivalent to
inverting $\psi_{\mu,\epsilon}$.
Thus $\Psi^{-1}A$ is a maximal order by
Lemma~\ref{lem:big-cell-maximal-order}.

The strata $S_u$, $u\leq w^\bullet$, are smooth and
connected, hence irreducible.
Theorem~\ref{thm:fiberiso} gives constancy of the
fibre algebras on each stratum, while
Lemma~\ref{le:Bv} gives
\[
B_\theta\cap S_u\ne\varnothing
\qquad\text{for every }u\leq w^\bullet.
\]
Lemma~\ref{lem:maximal-order-chart} now applies and
proves the theorem.
\end{proof}

\subsection{PI degrees}
\label{sec:PI-general}
Since the algebras $\O_\epsilon[S]$ and $\O_\epsilon[B_\theta]$ are finite over their Frobenius central subalgebras, they are polynomial identity (PI) algebras. In this subsection, we determine their PI degrees. Thanks to Proposition \ref{prop:integral-localization}, one has
\begin{equation}\label{eq:PIid}
    \text{PIdeg }\O_\epsilon[S]=\text{PIdeg }\O_\epsilon[B_\theta]
\end{equation}

The conditions on $\ell$ imply that $\gcd(\ell,d_i)=1$ for all $i\in\I$.

Throughout this subsection, put $w=w^\bullet=w_0w_\bullet$, $L=l(w)$, and $m=\operatorname{rank}\breve X$. Choose a reduced expression $w=s_{i_1}\cdots s_{i_L}$, with associated roots $\beta_1,\ldots,\beta_L$, and a $\Z$-basis $\lambda_1,\ldots,\lambda_m$ of $\breve X=(1-\theta)X$. We use the integral pairing
\[
(\beta,\lambda)
=d_\beta\langle\beta^\vee,\lambda\rangle,
\qquad \beta\in\Phi,\quad\lambda\in X,
\]
where $d_{\beta_k}=d_{i_k}$. On the root lattice this is the symmetric form fixed by the Cartan datum.

Define $A_w\in M_L(\Z)$ and $C\in M_{L\times m}(\Z)$ by
\[
(A_w)_{ij}=\begin{cases}
(\beta_i,\beta_j),&i<j,\\
0,&i=j,\\
-(\beta_i,\beta_j),&i>j,
\end{cases}
\qquad C_{ia}=(\beta_i,\lambda_a),
\]
and put
\begin{equation}\label{eq:PI-commutation-matrix}
M=\begin{pmatrix}A_w&-C\\ C^{\mathsf T}&0\end{pmatrix}
\in M_{L+m}(\Z).
\end{equation}
Note that the signs in $M$ correspond to listing the root generators first and the torus generators last.

\begin{theorem}\label{thm:PI-general}
Let $\overline M:(\Z/\ell\Z)^{L+m}\rightarrow(\Z/\ell\Z)^{L+m}$ be the homomorphism induced by $M$. Then
\begin{equation}\label{eq:PI-exact}
\operatorname{PIdeg}\O_\epsilon[S]
=\sqrt{|\operatorname{im}\overline M|}.
\end{equation}
Equivalently, write the nonzero invariant factors of $M$ as
$h_1,h_1,\ldots,h_s,h_s$, with $h_j>0$ and $h_j\mid h_{j+1}$. Then
\begin{equation}\label{eq:PI-invariant-factors}
\operatorname{PIdeg}\O_\epsilon[S]
=\prod_{j=1}^s\frac{\ell}{\gcd(\ell,h_j)}.
\end{equation}
The empty product is understood to be $1$.
\end{theorem}

\begin{proof}
    By \eqref{eq:PIid}, it suffices to compute the PI degree of $\O_\epsilon[B_\theta]$. By the smash-product description of $\O_\epsilon[B_\theta]$, the Proposition follows from \cite[Theorem 6.4]{DCP93}.
\end{proof}

We next control the invariant factors. Put $D=\prod_{i\in\I}d_i$. On the full root datum, write $\tau=-w_\bullet\theta$, extending the given Satake involution. Recall that $w^2=1$ and
\[
w\theta=\theta w=-w_0\tau.
\]
All kernels in characteristic zero below are taken on $X\otimes_\Z\Q$.

\begin{lemma}\label{lem:PI-matrix-rank}
For every field $k$ in which $2D$ is invertible, one has
\begin{equation}\label{eq:PI-rational-rank}
\operatorname{rank}_k(M\otimes k)
=\dim S-\dim_\Q\ker(1-w_0\tau).
\end{equation}
Consequently, every prime divisor of a nonzero invariant factor of $M$ divides $2D$.
\end{lemma}

\begin{proof}
    Write $X_k=X\otimes_\Z k$, $Y_k=Y\otimes_\Z k$, and let $Q$ and $Q^\vee$ be the root and coroot lattices. The natural maps $Q_k\rightarrow X_k$ and $Q_k^\vee\rightarrow Y_k$ need not be injective. We will only use their restrictions to the $-1$ eigenspaces of $w$, where they are isomorphisms. For any vector space $V$ where $w$ (resp., $\theta$) acts, write $V^{-w}$ (resp., $V^{-\theta}$) to denote the $-1$ eigenspace of $w$ (res., $\theta$). Similarly denote by $V^{+w}$ and $V^{+\theta}$ the $+1$ eigenspaces.

    To prove this assertion, define the integral map
    \[
    b_w:X\longrightarrow Q,\qquad b_w(v)=v-wv.
    \]
    It takes values in $Q$ because $w$ belongs to the Weyl group. If $\iota_X:Q\rightarrow X$ denotes the inclusion, then
    \[
    \iota_X b_w=1-w,\qquad b_w\iota_X=1-w,
    \qquad wb_w=-b_w.
    \]
    These identities remain valid after tensoring with $k$. Thus $\iota_X:Q_k^{-w}\rightarrow X_k^{-w}$ is an isomorphism with inverse $\frac12b_w$. The same argument applies to $Q_k^\vee\rightarrow Y_k$. 

    The map
    \[
    H:Q_k\longrightarrow Q_k^\vee,\qquad
    \alpha_i\longmapsto d_i\alpha_i^\vee
    \]
    is an isomorphism, equivariant for $W$ and $\theta$. It therefore induces an isomorphism
    \[
    H_w:X_k^{-w}\overset{\sim}{\longrightarrow}Y_k^{-w}
    \]
    by the preceding identifications. Notice that this construction does not require inverting the index of $Q$ in the character lattice.

    Since $2$ is invertible in $k$, the inclusion $\breve X\subseteq X$ identifies
    \[
    \breve X\otimes_\Z k=X_k^{-\theta}=:E_k.
    \]
    
    Write a vector in $k^{L+m}$ as $(\mathbf a,\lambda)$, with $\lambda\in E_k$, and put
    \[
    \eta=\sum_{i=1}^L a_i\beta_i\in Q_k,
    \qquad z=\iota_X(\eta)\in X_k.
    \]
    The first $L$ equations in $M(\mathbf a,\lambda)=0$ are equivalent to
    \[
    \nu_0=z-\lambda,\qquad
    a_i=\frac{\langle H(\beta_i),\nu_{i-1}\rangle}{2d_{\beta_i}},
    \qquad
    \nu_i=\nu_{i-1}-2a_i\beta_i=s_{\beta_i}\nu_{i-1},
    \]
    where roots in the recursion are viewed in $X_k$ and $d_{\beta_i}=(\beta_i,\beta_i)/2$.
    Since $s_{\beta_L}\cdots s_{\beta_1}=w$, telescoping the reflection differences gives
    \[
    2\eta=b_w(\nu_0)\quad\text{in }Q_k,
    \qquad \nu_L=w\nu_0=-z-\lambda\quad\text{in }X_k.
    \]
    The first identity is valid in the root lattice: it is obtained by expanding $1-s_{\beta_L}\cdots s_{\beta_1}$ as a sum of reflection differences, before tensoring with $k$.
    It follows that $\eta\in Q_k^{-w}$, $wz=-z$, and $w\lambda=\lambda$.

    The remaining $m$ equations say
    \[
    \langle\iota_Y H(\eta),\mu\rangle=0
    \qquad\text{for every }\mu\in E_k.
    \]
    Here $\iota_Y:Q_k^\vee\rightarrow Y_k$ is the natural map. Since $\eta\in Q_k^{-w}$, its image $\iota_Y H(\eta)$ is $H_w(z)$. The pairing between $Y_k$ and $X_k$ is perfect, and the annihilator of $X_k^{-\theta}$ is $Y_k^{+\theta}$. As $H_w$ commutes with $\theta$, the last equations are therefore equivalent to $\theta z=z$.

    Conversely, given
    \[
    z\in X_k^{-w}\cap X_k^{+\theta},\qquad
    \lambda\in X_k^{+w}\cap X_k^{-\theta},
    \]
    the displayed recursion starting at $z-\lambda$ determines a unique $\mathbf a$. The identity $2\eta=b_w(z-\lambda)$ gives $\iota_X(\eta)=z$, and the preceding computations then verify all equations of $M(\mathbf a,\lambda)=0$. Hence
    \[
    \ker(M\otimes k)\cong
    (X_k^{-w}\cap X_k^{+\theta})
    \oplus(X_k^{+w}\cap X_k^{-\theta})
    =X_k^{-w\theta}.
    \]
    The simultaneous eigenspace dimensions of the commuting involutions $w,\theta$ are unchanged in odd characteristic. Indeed, the four projectors $(1\pm w)(1\pm\theta)/4$ split the free module $X\otimes_\Z\Z[1/2]$ into direct summands. Thus
    \[
    \dim_k\ker(M\otimes k)
    =\dim_\Q\ker(1+w\theta)
    =\dim_\Q\ker(1-w_0\tau).
    \]
    Since $L+m=\dim S$, this proves \eqref{eq:PI-rational-rank}. Finally, if a prime $p\nmid2D$ divided a nonzero Smith invariant of $M$, its rank over $\mathbb F_p$ would be smaller than its rank over $\Q$, a contradiction.
\end{proof}

\begin{corollary}\label{cor:PI-geometric}
Assume that $\ell$ is odd and coprime to root lengths. For an arbitrary root datum one has
\begin{equation}\label{eq:PI-geometric-corrected}
\operatorname{PIdeg}\O_\epsilon[S]
=\ell^{\frac12(\dim S-\dim\ker(1-w_0\tau))}
=\ell^{d_e/2}.
\end{equation}
Here $d_e$ is the dimension of the symplectic leaves contained in $S_e$ as in \eqref{eq:du}.
\end{corollary}

\begin{proof}
By Lemma~\ref{lem:PI-matrix-rank}, every nonzero invariant factor $h_j$ of $M$ has all its prime divisors among those dividing $2D$. Since $\gcd(\ell,2D)=1$, Theorem~\ref{thm:PI-general} gives $\operatorname{PIdeg}\O_\epsilon[S]=\ell^s$. The same lemma gives $2s=\dim S-\dim\ker(1-w_0\tau)$.
\end{proof}

\begin{example}
For $SL_2/SO_2$, write $\alpha=2\omega$, so that $(\alpha,\alpha)=2$. Then $w=s$, $\breve X=2\Z\omega=\Z\alpha$, and
\[
M=\begin{pmatrix}0&-2\\2&0\end{pmatrix}.
\]
Thus the PI degree of $\O_\epsilon[SL_2/SO_2]$ is $\ell$ for odd $\ell$. 

For $PGL_2/PO_2$, one has $X=\Z\alpha$ and $\breve X=2\Z\alpha$. Then
\[
M=\begin{pmatrix}0&-4\\4&0\end{pmatrix}.
\]
The PI degree of $\O_\epsilon[PGL_2/PO_2]$ is again $\ell$ for odd $\ell$. 
\end{example}

\subsection{Dimensions of irreducible modules}
Recall that the dimension of symplectic leaves in $\ss_e$ is given by $d_e=\dim S-\dim\ker(1-w_0\tau)$.

\begin{theorem}\label{cor:dim}
\begin{itemize}
\item[(1)] Let $V$ be an irreducible module of $\O_\epsilon[S]$. Then $\dim V \leq \ell^{d_e/2}$.

\item[(2)] Let $x\in S_e$ and $V$ be an irreducible module of $\O_\epsilon[S](x)$. Then $\mathrm{dim} V=\ell^{d_e/2}$.
\end{itemize}
\end{theorem}

\begin{proof}
Set $A=\O_\epsilon[S]$ and identify $R=\O[S]$ with its Frobenius image. Then $A$ is an affine domain finite over $R$, and Corollary~\ref{cor:PI-geometric} gives
\[
n:=\operatorname{PIdeg}A=\ell^{d_e/2}.
\]
The dimensions of irreducible modules of a prime affine PI algebra are bounded by its PI degree; see \cite[Section~2.1]{MNTY24}. This proves (1).

For (2), let $\mathcal F\subseteq S$ be the \emph{fully Azumaya locus} of $A$ with respect to $R$. Thus $y\in\mathcal F$ if every irreducible $A(y)$-module has dimension $n$; see \cite[Definition~2.4]{MNTY24}. By \cite[Lemma~2.5]{MNTY24}, $\mathcal F$ is open and dense in $S$. Since $S_e$ is also open and dense, choose $y\in\mathcal F\cap S_e$. For any $x\in S_e$, Theorem~\ref{thm:fiberiso} gives $A(x)\cong A(y)$. Hence every irreducible $A(x)$-module has dimension $n$, proving (2).
\end{proof}

We propose a conjecture regarding the dimensions of irreducible modules of fibre algebras $\O_\epsilon[S](x)$ for any $x\in S$. 

\begin{conjecture}\label{conj:dimV}
    Let $x\in S_u$ and let $V$ be an irreducible module of $\O_\epsilon[S](x)$. There exists a positive integer $N_u$, such that when $\ell$ is coprime to $N_u$, one has $\dim V=\ell^{d_u/2}$. 
\end{conjecture}

Theorem~\ref{cor:dim} implies the conjecture for $u=e$ and $N_e=2D$ where $D=\prod_{i\in\I}d_i$. For general fibre algebras, we need more restriction on $\ell$ to obtain the uniform dimension formula.

\section{Comparison of fibre algebras}\label{sec:framed}

In this section, we show an isomorphism between the fibre algebras of the quantum symmetric space and the quantum Schubert cell associated to a framed Kac-Moody data. This isomorphism will provide a proof of Conjecture~\ref{conj:dimV}, once dimensions of irreducible modules for the quantum Schubert cell are known.

\subsection{The lattice $\breve X$}

For the remainder of this section, assume that $(Y,X,\langle \cdot, \cdot \rangle,\cdots)$ is semisimple and simply connected. Then $X=\bigoplus_{i\in\I}\Z\omega_i$, where $\omega_i$ are the fundamental weights. Recall that $(\I=\I_\circ\sqcup\I_\bullet,\tau)$ is a Satake diagram and $\I_{\circ,\tau}$ is a fixed set of representatives of the $\tau$-orbits in $\I_\circ$. For $\lambda\in X$, set $\breve{\lambda}=\lambda-\theta\lambda$. Put
\[
\mu_a=\omega_a+\omega_{\tau a},\qquad
\breve X_0=\bigoplus_{a\in\I_{\circ,\tau}}\Z\mu_a,
\qquad m=|\I_{\circ,\tau}|.
\]

\begin{lemma}\label{lem:omega}
\begin{itemize}
\item[(1)] For $i\in\I_\circ$, one has $\breve{\omega}_i=\omega_i+\omega_{\tau i}$.
\item[(2)] One has $2\breve X\subseteq\breve X_0\subseteq\breve X$. In particular, $\breve X_0$ has finite index in $\breve X$, and $\breve X/\breve X_0$ is an elementary abelian $2$-group.
\item[(3)] If $\ell$ is odd, the inclusion induces an isomorphism
\[
\breve X_0/\ell\breve X_0\overset{\sim}{\longrightarrow}
\breve X/\ell\breve X.
\]
\end{itemize}
\end{lemma}

\begin{proof}
    (1) Recall that $\theta=-w_\bullet\tau$. Since $w_\bullet$ fixes $\omega_i$ for $i\in\I_\circ$, we have $\theta\omega_i=-\omega_{\tau i}$.

    (2) Part (1) gives $\breve X_0\subseteq\breve X=(1-\theta)X$. Take $\lambda\in\breve X$. Then $\theta\lambda=-\lambda$. For $j\in\I_\bullet$, the coroot $\alpha_j^\vee$ is fixed by $\theta$, so
    \[
    \langle\alpha_j^\vee,\lambda\rangle
    =\langle\theta\alpha_j^\vee,\theta\lambda\rangle
    =-\langle\alpha_j^\vee,\lambda\rangle=0.
    \]
    Thus $\lambda=\sum_{i\in\I_\circ}c_i\omega_i$, with $c_i\in\Z$. By (1), $\theta\lambda=-\lambda$ now gives $c_i=c_{\tau i}$. Consequently,
    \[
    2\lambda=
    \sum_{\substack{a\in\I_{\circ,\tau}\\\tau a=a}}c_a\mu_a
    +\sum_{\substack{a\in\I_{\circ,\tau}\\\tau a\ne a}}2c_a\mu_a
    \in\breve X_0.
    \]
    The weights $\mu_a$ are linearly independent, so the assertion follows.

    (3) Write $\X=\breve X/\breve X_0$. Multiplication by the odd integer $\ell$ is an automorphism of $\X$, since $2\X=0$. It follows that
    \[
    \breve X=\breve X_0+\ell\breve X,
    \qquad \breve X_0\cap\ell\breve X=\ell\breve X_0,
    \]
    proving the desired isomorphism.
\end{proof}

The inclusion $\breve X_0\subseteq\breve X$ need not be an equality. For example, in type $\text{AII}_3$, with black vertices $1,3$ and white vertex $2$, one has
\[
\breve{\omega}_1 = \omega_2,
\qquad
\breve{\omega}_2 = 2\omega_2,
\qquad
\breve{\omega}_3 = \omega_2.
\]
Hence $\breve X=\Z\omega_2$, whereas $\breve X_0=2\Z\omega_2$. 


\subsection{The framed data and Richardson strata}\label{sec:fdata} 

Let $\tI=\I\sqcup \{r_a | a\in \I_{\circ,\tau} \}$. Define the generalized Cartan matrix $\tC=(\tc_{ij})_{i,j\in \tI}$ by 
\begin{itemize}
\item[(1)] $\tc_{ij}=c_{ij}$ if $i,j\in \I$;

\item[(2)] $\tc_{r_a,r_b}=0$ for $a\neq b\in \I_{\circ,\tau}$ and $\tc_{r_a,r_a}=2$ for $a\in \I_{\circ,\tau}$;

\item[(3)] $\tc_{i,r_a}=0$ if $i\in \I_\circ,a\neq i,\tau i$ or $i\in \I_\bullet$;

\item[(4)] $ \tc_{i,r_a}=-2$ if $\tau i=i=a$;

\item[(5)] $\tc_{i,r_a}=-1$ if $\tau i\neq i, a\in \{i,\tau i\}$;

\item[(6)]
$ \tc_{r_a,i}=\tc_{i,r_a}$ if $i\in \I$ and $a\in \I_{\circ,\tau}$.
\end{itemize}

It is clear that $\tC$ is a generalized Cartan matrix with symmetrizers given by $\tilde{d}_i=d_i$ for $i\in \I$ and $\tilde{d}_{r_a}=d_a$ for $a\in \I_{\circ,\tau}$. 

Let $(\tY,\tX,\langle,\rangle, \cdots)$ be a root datum of type $(\tI,\cdot)$. Let $\tg$ be the Kac--Moody Lie algebra associated with $\tC$ and $\widetilde{W}$ be the Weyl group of $\tg$. Let us identify $W$ with the parabolic subgroup of $\widetilde{W}$ in the obvious way. By abuse of notation, we also write $r_a$ for $s_{r_a}$. Let $\widetilde{\U}$ be the simply connected quantum group associated with the generalized Cartan matrix $\widetilde{C}$.

\begin{table}[h]
\caption{Rank one Satake diagrams and framing}
\label{table:fdata}
\resizebox{5.5 in}{!}{%
\begin{tabular}{| c | c | c |}
\hline
Type & Satake diagram & Framed diagram
\\
\hline
\begin{tikzpicture}[baseline=0]
\node at (0, 0.1) {AI$_1$};
\end{tikzpicture}

&	
    \begin{tikzpicture}[baseline=0]
		\node  at (0,0) {$\circ$};
		\node  at (0,-.3) {\small 1};
	\end{tikzpicture}
& 

    \begin{tikzpicture}[baseline=0]
		\node  at (0,0) {$\circ$};
		\node  at (0,0.6) {$\circ$};
		\node  at (0.2,-.2) {\small 1};
		\node  at (0.3,.4) {\small $r_1$};
                     \draw[double equal sign distance] (0,0.1) to (0,0.55);
	\end{tikzpicture}
\\
\hline
\begin{tikzpicture}[baseline=0]
\node at (0, 0.1) {AII$_3$};
\end{tikzpicture}
&
   \begin{tikzpicture}[baseline=0]
		\node at (0,0) {$\bullet$};
		\draw (0.1, 0) to (0.4,0);
		\node  at (0.5,0) {$\circ$};
		\draw (0.6, 0) to (0.9,0);
		\node at (1,0) {$\bullet$};
		\node at (0,-.3) {\small 1};
		\node  at (0.5,-.3) {\small 2};
		\node at (1,-.3) {\small 3};
	\end{tikzpicture}
& 
   \begin{tikzpicture}[baseline=0]
\node at (0.5,0.6) {$\circ$};
\node at (0.8,0.4) {$r_2$};
\draw[double equal sign distance] (0.5,0.1) to (0.5,0.55);
		\node at (0,0) {$\bullet$};
		\draw (0.05, 0) to (0.4,0);
		\node  at (0.5,0) {$\circ$};
		\draw (0.6, 0) to (0.95,0);
		\node at (1,0) {$\bullet$};
		\node at (0,-.3) {\small 1};
		\node  at (0.5,-.3) {\small 2};
		\node at (1,-.3) {\small 3};
	\end{tikzpicture}
\\
\hline
\begin{tikzpicture}[baseline=0]
\node at (0, -0.45) {AIII$_{11}$};
\end{tikzpicture}
 &
\begin{tikzpicture}[baseline = 6] 
		\node at (-0.5,0) {$\circ$};
		\node at (0.5,0) {$\circ$};
		\draw[bend left, <->] (-0.5, 0.2) to (0.5, 0.2);
		\node at (-0.5,-0.3) {\small 1};
		\node at (0.5,-0.3) {\small 2};
	\end{tikzpicture}
& 
\begin{tikzpicture}[baseline = 6] 
\node at (0,-0.6) {$\circ$};
\node at (0.3,-0.8) {$r_1$};
\draw (-0.45,-0.05) to (-0.1,-0.55);
\draw (0.45,-0.05) to (0.1,-0.55);
		\node at (-0.5,0) {$\circ$};
		\node at (0.5,0) {$\circ$};
		\draw[bend left, <->] (-0.5, 0.2) to (0.5, 0.2);
		\node at (-0.5,-0.3) {\small 1};
		\node at (0.5,-0.3) {\small 2};
	\end{tikzpicture}
\\
\hline
\begin{tikzpicture}[baseline=0]
\node at (0, -0.4) {AIV, n$\geq$2};
\end{tikzpicture}
&
\begin{tikzpicture}	[baseline=6]
		\node at (-0.5,0) {$\circ$};
		\draw[-] (-0.4,0) to (-0.1, 0);
		\node  at (0,0) {$\bullet$};
		\node at (2,0) {$\bullet$};
		\node at (2.5,0) {$\circ$};
		\draw[-] (0.1, 0) to (0.5,0);
		\draw[dashed] (0.5,0) to (1.4,0);
		\draw[-] (1.6,0)  to (1.9,0);
		\draw[-] (2.1,0) to (2.4,0);
		\draw[bend left, <->] (-0.5, 0.2) to (2.5, 0.2);
		\node at (-0.5,-.3) {\small 1};
		\node  at (0,-.3) {\small 2};
		\node at (2.5,-.3) {\small n};
	\end{tikzpicture}
& 
\begin{tikzpicture}	[baseline=6]
\node at (1,-1) {$\circ$};
\node at (1.3,-1.2) {$r_1$};
\draw (-0.45,-0.05) to (0.95,-0.95);
\draw (2.45,-0.05) to (1.05,-0.95);

		\node at (-0.5,0) {$\circ$};
		\draw[-] (-0.4,0) to (-0.1, 0);
		\node  at (0,0) {$\bullet$};
		\node at (2,0) {$\bullet$};
		\node at (2.5,0) {$\circ$};
		\draw[-] (0.1, 0) to (0.5,0);
		\draw[dashed] (0.5,0) to (1.4,0);
		\draw[-] (1.6,0)  to (1.9,0);
		\draw[-] (2.1,0) to (2.4,0);
		\draw[bend left, <->] (-0.5, 0.2) to (2.5, 0.2);
		\node at (-0.5,-.3) {\small 1};
		\node  at (0.15,-.3) {\small 2};
		\node at (2.5,-.3) {\small n};
	\end{tikzpicture}
\\
\hline
\begin{tikzpicture}[baseline=0]
\node at (0, 0.1) {BII, n$\ge$ 2};
\end{tikzpicture} 
&
    	\begin{tikzpicture}[baseline=0, scale=1.5]
		\node at (1.05,0) {$\circ$};
		\node at (1.5,0) {$\bullet$};
		\draw[-] (1.1,0)  to (1.4,0);
		\draw[-] (1.4,0) to (1.9, 0);
		\draw[dashed] (1.9,0) to (2.7,0);
		\draw[-] (2.7,0) to (2.9, 0);
		\node at (3,0) {$\bullet$};
		\draw[-implies, double equal sign distance]  (3.05, 0) to (3.55, 0);
		\node at (3.6,0) {$\bullet$};
		\node at (1,-.2) {\small 1};
		\node at (1.5,-.2) {\small 2};
		\node at (3.6,-.2) {\small n};
	\end{tikzpicture}	
& 
\begin{tikzpicture}[baseline=0, scale=1.5]
\node at (1.05,0.6) {$\circ$};
\node at (1.25, 0.45) {$r_1$};
\draw[double equal sign distance] (1.05, 0.55) to (1.05,0.05);
		\node at (1.05,0) {$\circ$};
		\node at (1.5,0) {$\bullet$};
		\draw[-] (1.1,0)  to (1.4,0);
		\draw[-] (1.4,0) to (1.9, 0);
		\draw[dashed] (1.9,0) to (2.7,0);
		\draw[-] (2.7,0) to (2.9, 0);
		\node at (3,0) {$\bullet$};
		\draw[-implies, double equal sign distance]  (3.05, 0) to (3.55, 0);
		\node at (3.6,0) {$\bullet$};
		\node at (1,-.2) {\small 1};
		\node at (1.5,-.2) {\small 2};
		\node at (3.6,-.2) {\small n};
	\end{tikzpicture}	
\\
\hline
\begin{tikzpicture}[baseline=0]
\node at (0, -0.05) {CII, n$\ge$3};
\end{tikzpicture}
&
		\begin{tikzpicture}[baseline=6]
		\draw (0.6, 0.15) to (0.9, 0.15);
		\node  at (0.5,0.15) {$\bullet$};
		\node at (1,0.15) {$\circ$};
		\node at (1.5,0.15) {$\bullet$};
		\draw[-] (1.1,0.15)  to (1.4,0.15);
		\draw[-] (1.4,0.15) to (1.9, 0.15);
		\draw (1.9, 0.15) to (2.1, 0.15);
		\draw[dashed] (2.1,0.15) to (2.7,0.15);
		\draw[-] (2.7,0.15) to (2.9, 0.15);
		\node at (3,0.15) {$\bullet$};
		\draw[implies-, double equal sign distance]  (3.1, 0.15) to (3.7, 0.15);
		\node at (3.8,0.15) {$\bullet$};
		\node  at (0.5,-0.15) {\small 1};
		\node at (1,-0.15) {\small 2};
		\node at (3.8,-0.15) {\small n};
	\end{tikzpicture}
&	
		\begin{tikzpicture}[baseline=6]
\node at (1.0, 0.7) {$\circ$};
\node at (1.3,0.55) {$r_2$};
\draw [double equal sign distance] (1, 0.65) to (1,0.2);
		\draw (0.6, 0.15) to (0.9, 0.15);
		\node  at (0.5,0.15) {$\bullet$};
		\node at (1,0.15) {$\circ$};
		\node at (1.5,0.15) {$\bullet$};
		\draw[-] (1.1,0.15)  to (1.4,0.15);
		\draw[-] (1.4,0.15) to (1.9, 0.15);
		\draw (1.9, 0.15) to (2.1, 0.15);
		\draw[dashed] (2.1,0.15) to (2.7,0.15);
		\draw[-] (2.7,0.15) to (2.9, 0.15);
		\node at (3,0.15) {$\bullet$};
		\draw[implies-, double equal sign distance]  (3.1, 0.15) to (3.7, 0.15);
		\node at (3.8,0.15) {$\bullet$};
		\node  at (0.5,-0.15) {\small 1};
		\node at (1,-0.15) {\small 2};
		\node at (3.8,-0.15) {\small n};
	\end{tikzpicture}
\\
\hline
\begin{tikzpicture}[baseline=0]
\node at (0, -0.05) {DII, n$\ge$4};
\end{tikzpicture}&
	\begin{tikzpicture}[baseline=0]
		\node at (1,0) {$\circ$};
		\node at (1.5,0) {$\bullet$};
		\draw[-] (1.1,0)  to (1.4,0);
		\draw[-] (1.4,0) to (1.9, 0);
		\draw[dashed] (1.9,0) to (2.7,0);
		\draw[-] (2.7,0) to (2.9, 0);
		\node at (3,0) {$\bullet$};
		\node at (3.5, 0.4) {$\bullet$};
		\node at (3.5, -0.4) {$\bullet$};
		\draw (3.05, 0.05) to (3.4, 0.39);
		\draw (3.05, -0.05) to (3.4, -0.39);
		\node at (1,-.3) {\small 1};
		\node at (1.5,-.3) {\small 2};
		\node at (3.6, 0.25) {\small n-1};
		\node at (3.5, -0.6) {\small n};
	\end{tikzpicture}	
&
\begin{tikzpicture}[baseline=0]
\node at (1.0, 0.7) {$\circ$};
\node at (1.3,0.55) {$r_1$};
\draw [double equal sign distance] (1, 0.65) to (1,0.05);
		\node at (1,0) {$\circ$};
		\node at (1.5,0) {$\bullet$};
		\draw[-] (1.1,0)  to (1.4,0);
		\draw[-] (1.4,0) to (1.9, 0);
		\draw[dashed] (1.9,0) to (2.7,0);
		\draw[-] (2.7,0) to (2.9, 0);
		\node at (3,0) {$\bullet$};
		\node at (3.5, 0.4) {$\bullet$};
		\node at (3.5, -0.4) {$\bullet$};
		\draw (3.05, 0.05) to (3.4, 0.39);
		\draw (3.05, -0.05) to (3.4, -0.39);
		\node at (1,-.3) {\small 1};
		\node at (1.5,-.3) {\small 2};
		\node at (3.6, 0.25) {\small n-1};
		\node at (3.5, -0.6) {\small n};
	\end{tikzpicture}	
\\
\hline
\begin{tikzpicture}[baseline=0]
\node at (0, 0.2) {FII};
\end{tikzpicture}
&
\begin{tikzpicture}[baseline=0][scale=1.5]
	\node at (0,0) {$\bullet$};
	\draw (0.1, 0) to (0.4,0);
	\node at (0.5,0) {$\bullet$};
	\draw[-implies, double equal sign distance]  (0.6, 0) to (1.2,0);
	\node at (1.3,0) {$\bullet$};
	\draw (1.4, 0) to (1.7,0);
	\node at (1.8,0) {$\circ$};
	\node at (0,-.3) {\small 1};
	\node at (0.5,-.3) {\small 2};
	\node at (1.3,-.3) {\small 3};
	\node at (1.8,-.3) {\small 4};
\end{tikzpicture}
& 
\begin{tikzpicture}[baseline=0][scale=1.5]

\node at (1.8, 0.7) {$\circ$};
\node at (2.1,0.55) {$r_4$};
\draw [double equal sign distance] (1.8, 0.65) to (1.8,0.05);

	\node at (0,0) {$\bullet$};
	\draw (0.1, 0) to (0.4,0);
	\node at (0.5,0) {$\bullet$};
	\draw[-implies, double equal sign distance]  (0.6, 0) to (1.2,0);
	\node at (1.3,0) {$\bullet$};
	\draw (1.4, 0) to (1.7,0);
	\node at (1.8,0) {$\circ$};
	\node at (0,-.3) {\small 1};
	\node at (0.5,-.3) {\small 2};
	\node at (1.3,-.3) {\small 3};
	\node at (1.8,-.3) {\small 4};
\end{tikzpicture}
\\
\hline
\end{tabular}
}
\newline
\smallskip
\end{table}

Let $\widetilde G$ be the minimal Kac--Moody group associated with this datum, with opposite Borel subgroups $\tB_\pm$ and flag ind-variety $\widetilde G/\tB_+$.  For $w\in \tW$, the Schubert cell $\tB_+w\tB_+/\tB_+$ is finite dimensional.  Set
\[
 \tN(w)= \tN_+\cap w \tN_- w^{-1},
\]
 and then $i_w:n\mapsto nw\tB_+$ identifies $ \tN(w)$ with this Schubert cell.  For $v\leq w$ define the open Richardson stratum
\[
 \tR_{v,w}:= \tN(w)\cap
 i_w^{-1}(\tB_- v \tB_+/ \tB_+).
\]
Equivalently, $i_w(\tR_{v,w})$ is the intersection of the Schubert cell indexed by $w$ with the opposite Schubert cell indexed by $v$.  The Bruhat decomposition gives the finite stratification
\begin{equation}\label{eq:Richardson-stratification}
  \tN(w)=\bigsqcup_{v\leq w} \tR_{v,w}.
\end{equation}
Each $\tR_{v,w}$ is smooth, irreducible, and of dimension $l(w)-l(v)$; see, for example, \cite{Kum02}.  


The Drinfeld--Jimbo Lie bialgebra
structure on the Kac--Moody algebra integrates to the standard Poisson
ind-group structure on $\widetilde G$; cf. \cite[Chapter 3]{Wil14}.  The quotient
$\widetilde G/\tB_+$ consequently carries its standard Poisson
structure. This Poisson structure restricts to 
Schubert cells, which defines a Poisson bivector $\pi_w$ on $\tN(w)$ through $i_w$. We refer to $\pi_w$ as the
\emph{standard Poisson structure} on $\tN(w)$.  Notice that this is an
ordinary algebraic Poisson structure: although $\widetilde G$ is an
ind-group, both $\tN(w)$ and the corresponding Schubert cells have finite dimension
$l(w)$.  Standard Poisson structures on Schubert cells for
symmetrizable Kac--Moody groups are also treated in \cite{GY23}.

\subsection{The coordinate algebra \texorpdfstring{$\O_\A[\tN(w)]$}{O A[N-tilde(w)]}}
\label{sec:Poisson-Nw}

Let $r=\prod_{a\in \I_{\circ,\tau}} r_a$ and $\tT_r=\prod_{a\in \I_{\circ,\tau}} \tT_{r_a}$ be the braid group operator on $\widetilde{\U}$. 

Following the construction in Section~\ref{sec:PBW}, we have the De Concini--Kac integral form $\widetilde{\mathfrak{U}}^+$ and $\A$-subalgebras $\widetilde{\mathfrak{U}}^+(w)$ for any $w\in\tW$. For $w\in\widetilde{W}$, recall the quantum Schubert cell $\O_\A[\tN(w)]$ and the $\A$-algebra isomorphism $\widetilde{\mathfrak{U}}^+(w)\overset{\sim}{\rightarrow}\O_{\A}[\widetilde{N}(w)]$. We shall identify $\widetilde{\mathfrak{U}}^+(w)$ with $\O_{\A}[\widetilde{N}(w)]$ if there is no confusion.

Fix a reduced expression $\un{w_0 w_\bullet}=s_{i_1}\cdots s_{i_L}$ for $w_0 w_\bullet$. Then $rw_0 w_\bullet$ admits a natural reduced expression $(\prod_{a\in \I_{\circ,\tau}} r_a) s_{i_1}\cdots s_{i_L}$ where $l(r w_0 w_\bullet)=L+|\I_{\circ,\tau}|$. Then the root vectors associated with this reduced expression are given by
\begin{equation}\label{eq:troot}
\tE_{\alpha_{r_a}}, \tE_{r\beta_k}:=\tT_{r}(\tE_{\beta_k}),\qquad a\in \I_{\circ,\tau}, 1\leq k\leq L.
\end{equation}
In particular, $\widetilde{\mathfrak{U}}^+(rw^\bullet)$ is generated by vectors in \eqref{eq:troot} as an $\A$-algebra.

\begin{lemma}\label{lem:rbetak}
Let $a\in \I_{\circ,\tau}$ and $1\leq k \leq L$. Then we have $\tE_{\alpha_{r_a}} \tE_{r\beta_k} =q_{i_k}^{\langle \beta_k^\vee, \omega_a + \omega_{\tau a}\rangle} \tE_{r\beta_k} \tE_{\alpha_{r_a}}$.
\end{lemma}

\begin{proof}
By the construction of $\tC$, we have for any $i\in \I,a\in \I_{\circ,\tau}$
\[
\langle \alpha_i^\vee,\alpha_{r_a}\rangle=-\langle \alpha_i^\vee, \omega_a + \omega_{\tau a}\rangle,
\]
and hence one has 
\[
\langle r\beta_k^\vee, \alpha_{r_a}\rangle=\langle \beta_k^\vee,r\alpha_{r_a}\rangle=-\langle \beta_k^\vee,\alpha_{r_a}\rangle=\langle \beta_k^\vee, \omega_a + \omega_{\tau a}\rangle.
\]
By the Levendorskii--Soibelman straightening law, we have
\[
\tE_{\alpha_{r_a}} \tE_{r\beta_k} =q_{i_k}^{\langle r\beta_k^\vee, \alpha_{r_a}\rangle} \tE_{r\beta_k} \tE_{\alpha_{r_a}}+\mathcal{M}=q_{i_k}^{\langle \beta_k^\vee, \omega_a+\omega_{\tau a}\rangle} \tE_{r\beta_k} \tE_{\alpha_{r_a}}+\mathcal{M}
\]
where $\mathcal{M}$ is a polynomial in root vectors $\tE_{\alpha_{r_{b}}}, \tE_{r\beta_s}, 1\le s <k, b\in \I_{\circ,\tau}\backslash \{a\} $. If $\mathcal{M}$ is nonzero, then by considering the weights, one has 
\begin{equation}\label{eq:abg}
\alpha_{r_a}+r\beta_k=\alpha'+r\gamma\qquad \text{ in }\mathbb{Z}[\alpha_{r_i},\alpha_j\mid i\in \I_{\circ,\tau},j\in \I],
\end{equation}
where $\alpha'$ is a sum of elements in $\alpha_{r_b}$ with $b\in\I_{\circ,\tau}\backslash \{ a\}$ and $\gamma\in \mathbb{Z}[\alpha_i\mid i\in \I]$. Applying $r$ to \eqref{eq:abg} one obtains 
\[
-\alpha_{r_a}+\alpha'=\gamma-\beta_k,
\]
which is a contradiction since $\mathbb{Z}[\alpha_{r_i}\mid i\in\I_{\circ,\tau}]\cap \mathbb{Z}[\alpha_j\mid j\in \I]=0$.
\end{proof}

\begin{proposition}\label{prop:atlas}
There is an $\A$-algebra embedding
\begin{align*}
\Lambda:\O_{\A}[\tN(r w_0 \bw)] &\hookrightarrow \O_\A[B_\theta] 
\cong \O_\A[N(w_0 \bw)] \# \O_\A[T_\theta],
\\
\tE_{\alpha_{r_a}} &\mapsto \one_{\omega_a+\omega_{\tau a}}^*,
\\
\tE_{r\beta_k} & \mapsto \tE_{\beta_k}.
\end{align*}
\end{proposition}

\begin{proof}
By Section~\ref{sec:PBW}, the ordered monomials in $\tE_{\alpha_{r_a}}, \tE_{r\beta_k}$ form a PBW basis of $\O_{\A}[\tN(r w_0 \bw)]$. Hence, it suffices to verify that the commutation relations for these root vectors are satisfied in their images.

First, since $\tc_{r_a,r_b}=0$ for $a\neq b $, elements $\tE_{\alpha_{r_a}},a\in \I_{\circ,\tau}$ commute with each other. By Lemma~\ref{lem:onela}, elements $\one_{\omega_a+\omega_{\tau a}}^*,a\in \I_{\circ,\tau}$ also commute with each other.

Second, for $1\le k < l \le L $, we have
\[
\tE_{\beta_k} \tE_{\beta_l}=q_{i_k}^{\langle \beta_k^\vee,\beta_l\rangle} \tE_{\beta_l} \tE_{\beta_k} + R\big(\tE_{\beta_{k+1}}, \ldots, \tE_{\beta_{l-1}}\big)
\]
where $R\big(\tE_{\beta_{k+1}}, \ldots, \tE_{\beta_{l-1}}\big)$ is a polynomial in $\tE_{\beta_s}, k<s<l$ with coefficients in $\A$. Note that $\langle r\beta_k^\vee,r\beta_l\rangle=\langle \beta_k^\vee,\beta_l\rangle$, and then by applying $\tT_r$, we have
\[
\tE_{r\beta_k} \tE_{r\beta_l}=q_{i_k}^{\langle r\beta_k^\vee,r\beta_l\rangle} \tE_{r\beta_l} \tE_{r\beta_k} + R(\tE_{r\beta_{k+1}}, \ldots, \tE_{r\beta_{l-1}}).
\]

Together with Lemma~\ref{lem:rbetak}, the above arguments show that $\Lambda$ is an algebra homomorphism. Moreover, the image of the PBW basis forms a linearly independent set in $\O_\A[B_\theta]$. Hence, $\Lambda$ is an $\A$-algebra embedding.
\end{proof}

We write $\Lambda_1$ and $\Lambda_\epsilon$ to denote the specialization of $\Lambda$ at $q=1$ and $q=\epsilon$, respectively.

\subsection{Classical framed chart and its strata}
For $a\in \I_{\circ,\tau}$, set
\[
\mathcal C_\A=\O_\A[\widetilde N(rw^\bullet)],\qquad
z_a=\tE_{\alpha_{r_a}},\qquad
\mathcal B_{0,\A}=\O_\A[N(w^\bullet)]\#\A[\breve X_0].
\]
We use the same notation $z_a$ for its specializations, and write $z_{a,1}$ when the specialization at $1$ needs to be distinguished. Here $\A[\breve X_0]$ is the subalgebra of $\O_\A[T_\theta]$ spanned by $\one_\lambda^*$, $\lambda\in\breve X_0$. By Lemma~\ref{lem:rbetak}, the $z_a$ are commuting normal elements. Proposition~\ref{prop:atlas} and the PBW basis therefore give
\[
\mathcal C_\A[z_a^{-1}\mid a\in\I_{\circ,\tau}]
\overset{\sim}{\longrightarrow}\mathcal B_{0,\A}
\subseteq\O_\A[B_\theta].
\]
The last algebra is free as a left and as a right $\mathcal B_{0,\A}$-module of rank
\[
d=[\breve X:\breve X_0].
\]
Indeed, the elements $\one_\nu^*$ for a set of representatives $\nu$ of $\breve X/\breve X_0$ form such a basis, thanks to the smash-product description and \eqref{eq:com}.

Let $T_0=\operatorname{Spec}\C[\breve X_0]$, and let
$\rho:T_\theta\rightarrow T_0$ be the isogeny induced by the inclusion of character lattices. Write
\[
\mathcal U=\{y\in\widetilde N(rw^\bullet)\mid z_a(y)\ne0
\text{ for all }a\in\I_{\circ,\tau}\}.
\]
At $q=1$, the preceding localization identifies
$\mathcal U\cong N(w^\bullet)\times T_0$. The comorphism of $\Lambda_1$ factors as
\begin{equation}\label{eq:framed-open-embedding}
\psi:B_\theta\cong N(w^\bullet)\times T_\theta
\xrightarrow{\ \Id\times\rho\ }
\mathcal U\hookrightarrow\widetilde N(rw^\bullet).
\end{equation}
In particular, $\psi:B_\theta\rightarrow\mathcal U$ is a finite \emph{\'{e}tale} surjection of degree $d$. It is an isomorphism onto $\mathcal U$ precisely when $\breve X_0=\breve X$. 

Recall from Section~\ref{sec:rwg} the stratification $B_\theta=\bigsqcup_{v \leq w^\bullet} B_{\theta,v}$.

\begin{proposition}\label{prop:framed-strata}
One has
\[
\mathcal U=\bigsqcup_{v\leq w^\bullet}\widetilde R_{v,rw^\bullet},
\qquad
\psi^{-1}(\widetilde R_{v,rw^\bullet})=B_{\theta,v}.
\]
For every $v\leq w^\bullet$, the restriction
\[
\psi|_{B_{\theta,v}}:B_{\theta,v}\longrightarrow
\widetilde R_{v,rw^\bullet}
\]
is a finite \'{e}tale surjection of degree $d$.
\end{proposition}

\begin{proof}
    Write $w=w^\bullet$. Choose root subgroup coordinates compatible with the PBW coordinates and representatives $\dot r,\dot w$ of $r,w$. The factorization of the Schubert cell corresponding to the reduced expression $rw$ is
    \[
    (\mathbf c,n)\longmapsto
    \left(\prod_{a\in\I_{\circ,\tau}}x_{r_a}(c_a)\right)
    \dot r n\dot w\widetilde B_+,
    \qquad \mathbf c\in\C^m,\quad n\in N(w).
    \]
    The framing coordinates are $c_a$, and $\mathcal U$ is the locus where all $c_a$ are nonzero. Let $\widetilde P_\I$ be the standard parabolic with Levi subgroup corresponding to $\I$. If $c_a\ne0$, the rank-one Gauss decomposition writes
    \[
    x_{r_a}(c_a)\dot r_a=b_a^-u_a^+,
    \]
    where $b_a^-\in\widetilde B_-$ and $u_a^+$ belongs to the positive root subgroup for $\alpha_{r_a}$. The framing rank-one subgroups commute. Moreover, their positive root subgroups lie in the unipotent radical of $\widetilde P_\I$, which is normalized by the Levi subgroup. Since $n\dot w$ belongs to this Levi subgroup, these positive factors can be moved to the right of $n\dot w$ and absorbed into $\widetilde B_+$. Thus, for $\mathbf c\in(\C^\times)^m$,
    \[
    \left(\prod_a x_{r_a}(c_a)\right)\dot r n\dot w\widetilde B_+
    \in\widetilde B_-v\widetilde B_+/\widetilde B_+
    \quad\Longleftrightarrow\quad
    n\dot wB_+\in B_-vB_+/B_+.
    \]
    Conversely, projection to $\widetilde G/\widetilde P_\I$ shows that a point in $\widetilde R_{v,rw}$ with $v\in W$ has all framing coordinates nonzero: its projection lies in the opposite big cell, and on the commuting rank-one factors this condition is exactly $c_a\ne0$ for every $a$.

    It follows that, under the above factorization,
    \[
    \widetilde R_{v,rw}\cong N(w)_v\times T_0,
    \qquad
    N(w)_v=\{n\in N(w)\mid n\dot wB_+\in B_-vB_+/B_+\}.
    \]
    By Lemma~\ref{le:Bv}, $B_{\theta,v}\cong N(w)_v\times T_\theta$. The restriction of $\psi$ is $\Id\times\rho$, proving the assertions.
\end{proof}

The specialization at $q=1$ gives a canonical
algebra isomorphism
\begin{equation}\label{eq:classical-limit-Nw}
 \C\otimes_{\A,\,q\mapsto1}\O_\A[\tN(w)]
 \overset{\sim}{\longrightarrow}\O[\tN(w)].
\end{equation}
Since the specialization is commutative, $fg-gf$ is divisible by $q-1$
for all $f,g\in\O_\A[\tN(w)]$.  Hence the formula
\begin{equation}\label{eq:semiclassical-bracket-Nw}
 \{\bar f,\bar g\}_{\mathrm{sc}}
 :=\overline{\frac{fg-gf}{q-1}}
\end{equation}
defines a Poisson bracket on the left-hand side of
\eqref{eq:classical-limit-Nw}, independently of the chosen lifts.
Under \eqref{eq:classical-limit-Nw}, this bracket is exactly the bracket
defined by $\pi_w$, up to a nonzero scalar. Let us normalize $\pi_w$ such that \eqref{eq:classical-limit-Nw} is Poisson.  Thus $\O_\A[\tN(w)]$ is an integral quantization of
the standard Poisson variety $(\tN(w),\pi_w)$.  The comparison follows from
the Levendorskii--Soibelman straightening relations for the ordered root
vectors; see \cite{DCP93,Ki12,GY23}.

Since $\Lambda$ is defined over $\A$, it respects the brackets obtained by dividing commutators by $q-1$. Thus $\psi:B_\theta\rightarrow\mathcal U$ is a Poisson morphism. As it is \'{e}tale, for each $x\in B_\theta$ the differential $d\psi_x$ identifies the Poisson tensors at $x$ and $\psi(x)$. In particular, their symplectic leaves have the same dimension.


\subsection{Representations of \texorpdfstring{$\O_\epsilon[\tN(w)]$}{O epsilon[N-tilde(w)]} at roots of unity}
\label{sec:rep-Nw}

Retain the symmetrizable Kac--Moody datum above and the root-of-unity
notation of Section~\ref{sec:bc}: thus $\epsilon$ is a primitive
$\ell$-th root of unity, subject to the stated restrictions on $\ell$.
Set
\[
 \O_\epsilon[\tN(w)]
 :=\C_\epsilon\otimes_\A\O_\A[\tN(w)],
 \qquad
 \O_1[\tN(w)]
 :=\C_1\otimes_\A\O_\A[\tN(w)].
\]
The dual quantum Frobenius homomorphism gives a canonical central
embedding
\begin{equation}\label{eq:Frobenius-center-Nw}
 \Fr_w^*:\O_1[\tN(w)]\lhook\joinrel\longrightarrow
 \O_\epsilon[\tN(w)].
\end{equation}
Under the classical specialization
$\O_1[\tN(w)]\cong\O[\tN(w)]$, its image
\[
 Z_0(\tN(w)):=\Fr_w^*\bigl(\O_1[\tN(w)]\bigr)
\]
is a canonical central subalgebra of $\O_\epsilon[\tN(w)]$.  It is
isomorphic to the coordinate algebra $\O[\tN(w)]$, and we call it the
\emph{Frobenius center}.

For $x\in \tN(w)$, let $\mathfrak m_x\subset\O[\tN(w)]$ be the maximal ideal
of functions vanishing at $x$.  Via \eqref{eq:Frobenius-center-Nw}, the
corresponding fibre algebra is
\begin{equation}\label{eq:fibre-Nw}
 \O_\epsilon[\tN(w)](x)
 :=\O_\epsilon[\tN(w)]\big/
 \mathfrak m_x\O_\epsilon[\tN(w)].
\end{equation}
It is a finite-dimensional algebra.  Let $\mathcal L_x$ be the symplectic
leaf of the standard Poisson variety $(\tN(w),\pi_w)$ containing $x$, and
write
\[
 \tilde{d}_x:=\dim\mathcal L_x.
\]

\begin{conjecture}\label{conj:HMY-Nw}
For every $x\in \tN(w)$ and every simple
$\O_\epsilon[\tN(w)](x)$-module $V$, there exists a positive integer $\widetilde{N}_w$, such that when $\ell$ is coprime to $\widetilde{N}_w$, one has
\[
 \dim V=\ell^{\tilde{d}_x/2}.
\]
\end{conjecture}

\subsection{Root-of-unity fibres on the framed chart}
The maps $\Lambda_1$ and $\Lambda_\epsilon$ are compatible with the Frobenius embeddings. More precisely, one has the commutative diagram
\begin{equation}\label{cd:Lambda}
\begin{tikzcd}
\O_1[\widetilde N(rw^\bullet)] \arrow[r,"\Lambda_1"] \arrow[d,"\Fr_{rw^\bullet}^*"'] & \O_1[B_\theta] \arrow[d,"\Fr_\theta^*"] \\
\O_\epsilon[\widetilde N(rw^\bullet)] \arrow[r,"\Lambda_\epsilon"] & \O_\epsilon[B_\theta].
\end{tikzcd}
\end{equation}
Indeed, on the framing generators the two compositions send $z_{a,1}$ to $\one_{\ell\mu_a}^*$, and on the remaining dual root vectors the equality follows from the PBW formula for Frobenius. Thus, if $x\in B_\theta$ and $y=\psi(x)\in\mathcal U$, $\Lambda_\epsilon$ induces a homomorphism between the fibre algebras at $y$ and $x$.

\begin{proposition}\label{prop:tNB}
For any $x\in B_\theta$, the map $\Lambda_\epsilon$ induces an isomorphism
\[
\O_\epsilon[\widetilde N(rw^\bullet)](\psi(x))
\overset{\sim}{\longrightarrow}\O_\epsilon[B_\theta](x)
\]
as finite-dimensional algebras.
\end{proposition}

\begin{proof}
    Put $y=\psi(x)$ and $\mathcal C_\epsilon=\O_\epsilon[\widetilde N(rw^\bullet)]$. Take $a\in \I_{\circ,\tau}$. Since $y\in\mathcal U$, the image of $z_a^\ell=\Fr_{rw^\bullet}^*(z_{a,1})$ in $\mathcal C_\epsilon(y)$ is the nonzero scalar $z_{a,1}(y)$. Hence $z_a$ is invertible in this fibre, with inverse $z_{a,1}(y)^{-1}z_a^{\ell-1}$. Localizing at all $z_a$ therefore gives
    \[
    \mathcal C_\epsilon(y)\cong\mathcal B_{0,\epsilon}(y),
    \qquad
    \mathcal B_{0,\epsilon}=\C_\epsilon\otimes_\A\mathcal B_{0,\A}.
    \]
    Here the fibre on the right is taken over $\O[\mathcal U]$.

    We show that the induced map $\mathcal B_{0,\epsilon}(y)\rightarrow\O_\epsilon[B_\theta](x)$ is surjective. Write $e_\lambda=\one_\lambda^*$ for the quantum torus element of weight $\lambda\in\breve X$, and $\chi_\lambda$ for the corresponding classical character. Lemma~\ref{lem:omega} (3) gives
    \[
    \lambda=\mu+\ell\nu,\qquad
    \mu\in\breve X_0,\quad\nu\in\breve X.
    \]
    In the fibre at $x$, one has
    \[
    e_\lambda=e_\mu e_{\ell\nu}
    =\chi_\nu(x)e_\mu,
    \]
    since $e_{\ell\nu}=\Fr_\theta^*(\chi_\nu)$. Thus all torus elements lie in the image. The remaining root vectors already belong to $\mathcal B_{0,\epsilon}$, so the map is surjective.

    Finally, the PBW basis and the Frobenius formulas give
    \[
    \dim_\C\mathcal B_{0,\epsilon}(y)=\ell^{L+m}
    =\dim_\C\O_\epsilon[B_\theta](x).
    \]
    Indeed, a basis for each fibre is given by ordered root monomials with exponents in $\{0,\ldots,\ell-1\}$, multiplied by torus monomials representing the corresponding lattice modulo $\ell$. There are $L$ root exponents and $\ell^m$ torus classes on both sides. The surjective map is therefore an isomorphism.
\end{proof}

\begin{lemma}\label{le:tNxy}
For any $x,y\in\widetilde R_{v,rw^\bullet}$, with $v\leq w^\bullet$, one has an algebra isomorphism
\[
\O_\epsilon[\widetilde N(rw^\bullet)](x)
\cong\O_\epsilon[\widetilde N(rw^\bullet)](y).
\]
\end{lemma}

\begin{proof}
    By Proposition~\ref{prop:framed-strata}, there exist $x',y'\in B_{\theta,v}$ with $\psi(x')=x$ and $\psi(y')=y$. Proposition~\ref{prop:tNB}, Proposition~\ref{le:GB}, and Theorem~\ref{thm:fiberiso} give
    \begin{align*}
    \O_\epsilon[\widetilde N(rw^\bullet)](x)
    &\cong\O_\epsilon[B_\theta](x')
    \cong\O_\epsilon[S](x')\\
    &\cong\O_\epsilon[S](y')
    \cong\O_\epsilon[B_\theta](y')
    \cong\O_\epsilon[\widetilde N(rw^\bullet)](y).
    \end{align*}
    This completes the proof.
\end{proof}

\begin{theorem}\label{thm:OSN}
For any $v\leq w^\bullet$, $x\in S_v$, and $y\in\widetilde R_{v,rw^\bullet}$, one has
\[
\O_\epsilon[S](x)\cong\O_\epsilon[\widetilde N(rw^\bullet)](y)
\]
as finite-dimensional algebras.
\end{theorem}

\begin{proof}
    By Proposition~\ref{prop:framed-strata}, choose $x'\in B_{\theta,v}$ such that $\psi(x')=y$. Theorem~\ref{thm:fiberiso}, Proposition~\ref{le:GB}, and Proposition~\ref{prop:tNB} give
    \[
    \O_\epsilon[S](x)\cong\O_\epsilon[S](x')
    \cong\O_\epsilon[B_\theta](x')
    \cong\O_\epsilon[\widetilde N(rw^\bullet)](y).
    \]
    This completes the proof.
\end{proof}

\begin{remark}
The semisimple simply connected assumption in the begininng of this section is not essential for the fibre-algebra comparison. For a general reductive root datum, one can similarly construct a framed root datum adapted to the spherical lattice and obtain isomorphisms with the corresponding fibre algebras of quantum Schubert cells. We work in the semisimple simply connected case to keep the framed construction explicit and uniform.
\end{remark}

As a consequence of Theorem \ref{thm:OSN}, Conjecture \ref{conj:HMY-Nw} implies Conjecture \ref{conj:dimV}.

\end{document}